\documentclass[11pt,reqno]{amsart}

\usepackage[a4paper,margin=25mm]{geometry}
\usepackage[T1]{fontenc}
\usepackage{lmodern}
\usepackage{amsmath,amssymb,amsthm,mathtools}
\usepackage{bm}
\usepackage{microtype}
\usepackage{enumitem}
\usepackage{booktabs}
\usepackage{graphicx}
\usepackage{placeins}
\usepackage{xcolor}
\usepackage{hyperref}
\usepackage[nameinlink,capitalize,noabbrev]{cleveref}

\colorlet{additiongreen}{black}
\newenvironment{analystcheck}{}{}
\newenvironment{revisioncheck}{}{}
\newenvironment{additioncheck}{}{}
\newenvironment{analystadditioncheck}{}{}
\newenvironment{hyperaddition}{}{}
\newenvironment{hyperreplacement}{}{}
\newcommand{\revisionmark}[1]{{#1}}
\newcommand{\additionmark}[1]{{#1}}
\newcommand{\hyperadd}[1]{{#1}}
\newcommand{\hyperrewrite}[1]{{#1}}

\hypersetup{
  hidelinks,
  pdfauthor={Bohan Hu, Jungang Li, Wen Sun},
  pdftitle={Nonconvex Carnot Limit Shapes for Branching Random Walks on the Heisenberg Group}
}

\numberwithin{equation}{section}

\newtheorem{theorem}{Theorem}[section]
\newtheorem{proposition}[theorem]{Proposition}
\newtheorem{corollary}[theorem]{Corollary}
\newtheorem{lemma}[theorem]{Lemma}

\theoremstyle{definition}

\theoremstyle{remark}
\newtheorem{remark}[theorem]{Remark}

\newcommand{\HH}{\mathbb H}
\newcommand{\ZZ}{\mathbb Z}
\newcommand{\RR}{\mathbb R}

\newcommand{\PP}{\mathbb P}
\newcommand{\EE}{\mathbb E}
\newcommand{\QQ}{\mathbb Q}
\newcommand{\cF}{\mathcal F}
\newcommand{\cI}{\mathcal I}
\newcommand{\cC}{\mathcal C}
\newcommand{\cD}{\mathcal D}
\newcommand{\cK}{\mathcal K}
\newcommand{\cQ}{\mathcal Q}
\newcommand{\cR}{\mathcal R}
\newcommand{\cS}{\mathcal S}
\newcommand{\cU}{\mathcal U}
\newcommand{\cV}{\mathcal V}
\newcommand{\cX}{\mathcal X}
\newcommand{\ac}{\mathrm{AC}}
\newcommand{\Var}{\operatorname{Var}}

\newcommand{\conv}{\operatorname{conv}}

\newcommand{\one}{\mathbf 1}
\newcommand{\dd}{\,\mathrm d}
\newcommand{\e}{\mathrm e}
\newcommand{\dH}{d_{\mathrm H}}
\newcommand{\norm}[1]{\lVert #1\rVert}
\newcommand{\abs}[1]{\lvert #1\rvert}
\newcommand{\ip}[2]{\langle #1,#2\rangle}
\newcommand{\R}{\mathsf R}
\newcommand{\Lft}{\mathsf L}
\newcommand{\U}{\mathsf U}
\newcommand{\Dwn}{\mathsf D}

\title[Nonconvex Carnot limit shapes for Heisenberg BRW]
{Nonconvex Carnot Limit Shapes for Branching Random Walks on the Heisenberg Group}

\author{Bohan Hu \and Jungang Li \and Wen Sun}
\address{School of Mathematical Sciences, University of Science and
Technology of China\\Jinzhai Road 96, Hefei 230026, China}
\email[Bohan Hu]{hbh995d8@mail.ustc.edu.cn}
\email[Jungang Li]{jungangli@ustc.edu.cn}
\email[Wen Sun]{wsun@ustc.edu.cn}
\date{October 16, 2026}


\begin{document}

\begin{abstract}
We study a supercritical branching random walk on the discrete Heisenberg
group driven by the four standard horizontal generators.  Horizontal
position has scale \(n\), whereas the central coordinate records discrete
area and has scale \(n^2\).  On the Galton--Watson survival event, we prove
almost sure Hausdorff convergence, under Carnot scaling, of both the
rescaled particle positions and the collections of enhanced ancestral
paths.  The path limit consists of area lifts of horizontal paths
\(\omega\) satisfying \(J_t(\omega)\le t\log\rho\) for every
\(t\in[0,1]\), where \(\rho\) is the offspring mean.

In contrast to the classical Euclidean case, the deterministic limit
shape is nonconvex for every \(\rho>1\).  We identify its horizontal
projection and interval vertical fibers, determine the saturated shape
for \(\rho\ge4\), and characterize the origin fiber below saturation by
Wulff duality.  A hyperbolic representation gives an elliptic integral
formula for the Wulff area.  The maximal central coordinate
\(\Gamma_\rho\) is nondecreasing and concave in \(\rho\), equals \(1/4\)
for \(\rho\ge4\), and satisfies
\(\Gamma_\rho=(\log\rho)/\pi-(\log\rho)^2/(4\pi)
+O((\log\rho)^3)\) as \(\rho\downarrow1\).
For fixed \(1<\rho<4\), we obtain
\(M_n^Z=\Gamma_\rho n^2+O(n^{3/2}\sqrt{\log n})\) almost surely.
Above saturation the first correction is a random linear term;
at \(\rho=4\), under a finite second moment assumption, the deficit
\(\Gamma_4n^2-M_n^Z\) has order \(n\log\log n\) in probability.
\end{abstract}

\maketitle

\section{Introduction}

What is the asymptotic shape of a supercritical branching population whose
positions lie in a noncommutative group?  We answer this question for a
branching random walk on the discrete Heisenberg group, where the central
coordinate records an area functional of the horizontal ancestral path.
The correct normalization is the anisotropic Carnot scaling
\((x,y,z)\mapsto(n^{-1}x,n^{-1}y,n^{-2}z)\).
Our primary result is an almost sure path space shape theorem on survival;
evaluating the limiting paths at time one yields the spatial limit shape.
Unlike the convex Cram\'er sets arising in classical Euclidean branching
random walks, this limit is nonconvex for every supercritical offspring
mean.  Figure~\ref{fig:rho4-cloud-envelopes} illustrates its saddle-like
geometry in the saturated case \(\rho=4\).

\begingroup
\setlength{\intextsep}{6pt}
\begin{figure}[!ht]
  \centering
  \includegraphics[width=0.85\textwidth]
    {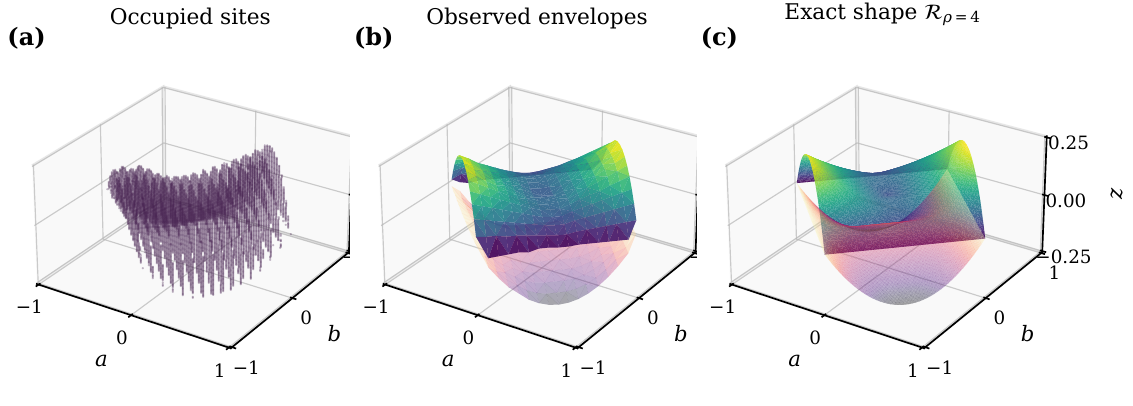}
  \setlength{\abovecaptionskip}{6pt}
  \caption{Particle positions and the limit shape at offspring mean
  \(\rho=4\).
  \textup{(a)} Distinct sites at generation \(n=16\).
  \textup{(b)} Empirical lower and upper fiber boundaries.
  \textup{(c)} Exact boundary surfaces.
  Horizontal and central coordinates are divided by \(n\) and \(n^2\),
  respectively.  Simulations use Poisson\((4)\) offspring.}
  \label{fig:rho4-cloud-envelopes}
\end{figure}
\endgroup

\subsection*{The Euclidean benchmark and the Heisenberg mechanism}

A branching random walk balances exponential population growth against the
large deviation cost of atypical trajectories.  In Euclidean space with the
same horizontal step law, positions are additive sums of increments, and the
limit shape is the convex Cram\'er set
\(\{v:L(v)\le\log\rho\}\), where \(L\) is the increment Cram\'er transform
\cite{Biggins1977Martingale,Biggins1977Chernoff,Biggins1978,Biggins2026Corrigendum,Shi2015}.

Our genealogy is a Galton--Watson tree with offspring number \(N\), mean
\(\rho:=\EE N>1\), and
\(\EE[N\log_+N]<\infty\).  The survival event
\(\cS\) is the event that the genealogy never becomes extinct.  Writing
\(\ZZ\) and \(\RR\) for the integers and reals, particle positions lie in
\[
  \HH(\ZZ)=\ZZ^3,
  \qquad
  (x,y,z)(a,b,c)=(x+a,y+b,z+c+xb).
\]
Each genealogical edge carries an independent increment chosen uniformly
from
\((1,0,0),\allowbreak(-1,0,0),\allowbreak(0,1,0),\allowbreak(0,-1,0)\).  If
\((X_k,Y_k,Z_k)\) denotes the successive positions along an ancestral line
starting at the identity, then the group law gives the discrete area identity
\begin{equation}\label{eq:intro-area}
  Z_n=\sum_{k=1}^nX_{k-1}\Delta Y_k,
  \qquad
  \Delta Y_k:=Y_k-Y_{k-1}.
\end{equation}
Under the Carnot rescaling
\((X,Y,Z)\mapsto(n^{-1}X,n^{-1}Y,n^{-2}Z)\), the limiting central coordinate
associated with an absolutely continuous horizontal curve
\(\omega=(x,y)\) is its area lift
\[
  A_t(\omega):=\int_0^t x(s)\dot y(s)\dd s.
\]

For an absolutely continuous curve, set
\(J_t(\omega):=\int_0^tL(\dot\omega(s))\dd s\), using the transform \(L\)
above.  At any fixed macroscopic time \(t\in(0,1]\), the generation
\(\lfloor nt\rfloor\) has exponential size
\(\exp\{nt\log\rho+o(n)\}\) almost surely on \(\cS\).  Comparison with the
sample path large deviation cost leads to the limiting admissibility
conditions
\begin{equation}\label{eq:intro-time-constraints}
  J_t(\omega)\le t\log\rho
  \qquad\text{for every }t\in[0,1].
\end{equation}
Let \(\cC_\rho\) be the absolutely continuous curves
\(\omega:[0,1]\to\RR^2\) with \(\omega(0)=0\) that satisfy
\eqref{eq:intro-time-constraints}.  A large action cost incurred early
cannot be offset by later motion: if the action exceeds the available
branching growth before time \(t\), later descendants cannot create the
required ancestral segments.  Since the central endpoint depends on the
complete horizontal path, these intermediate constraints cannot be replaced
by a terminal condition.

Our main results fall into three groups.

\subsection*{Main results (i): Carnot path and spatial shapes}

For each \(n\), let \(\cQ_n\) be the collection of the enhanced ancestral
paths of all particles in generation \(n\), with the horizontal and central
coordinates interpolated at scales \(n^{-1}\) and \(n^{-2}\), respectively.
The path space shape
theorem, \Cref{thm:path-shape}, states that, almost surely on \(\cS\),
\(\cQ_n\) converges in Hausdorff distance to the deterministic compact set
\[
  \cQ_\rho
  :=\{(\omega,A_\cdot(\omega)):\omega\in\cC_\rho\},
\]
where \(A_\cdot(\omega)\) denotes the path
\(t\mapsto A_t(\omega)\).  Let \(\cR_\rho\) be the set of terminal values of
the paths in \(\cQ_\rho\).  Evaluation at time one gives Hausdorff
convergence of the rescaled particle positions to \(\cR_\rho\)
(\Cref{cor:endpoint-shape}), and the corresponding directional support
functions converge uniformly (\Cref{cor:front-support}).

Write \(M_n^Z\) for the largest central coordinate among the particles in
generation \(n\), and let
\(\Gamma_\rho:=\max\{c:(a,b,c)\in\cR_\rho\}\).  Then, almost surely on
\(\cS\),
\[
  \frac{M_n^Z}{n^2}\longrightarrow\Gamma_\rho.
\]
The path space conclusion is stronger than a spatial shape theorem: it retains
precisely the ancestral information that produces the Heisenberg central
coordinate.

\subsection*{Main results (ii): Geometry of the limit shape and the central front}

The limit shape is not a Euclidean Cram\'er set.  Our second group of
results proves that \(\cR_\rho\) is nonconvex for every \(\rho>1\)
(\Cref{thm:nonconvex-main}), although its horizontal projection is the
classical Cram\'er set and every nonempty vertical fiber is a compact
interval.  Its support functions therefore determine only its Euclidean
convex hull.

The vertical geometry is organized around the saddle sheet \(c=ab/2\).
Straight horizontal paths place the part of this sheet above the horizontal
Cram\'er set inside \(\cR_\rho\).  Reversing the increments of the
unbranched walk makes the sheet the conditional symmetry center of the
central coordinate whenever the horizontal endpoint is reachable; the
annealed expected particle counts inherit the same symmetry.  Thus reversal
identifies the center of the annealed multiplicity profile, while
\(\cR_\rho\) describes the corresponding unweighted support.

The geometry changes at \(\rho=4\) (\Cref{thm:geometry-summary}).  At \(\rho=4\), the
branching exponent \(\log\rho\) reaches \(\log4\), the maximum finite value
of \(L\).
For every \(\rho\ge4\), the shape saturates at
\(\cR_4\), whose complete piecewise quadratic description is given in
\eqref{eq:complete-saturated-shape-main}.  If \(1<\rho<4\), \(\cR_\rho\)
contains the Carnot dilation of \(\cR_4\) by the factor
\(\sqrt{\log\rho/\log4}\) and is itself contained in \(\cR_4\).

Below saturation we obtain an exact Wulff characterization of the fiber over
the horizontal origin.  Write
\(\psi(\alpha,\beta):=\log((\cosh\alpha+\cosh\beta)/2)\) and, for \(h>0\),
let
\(K_h:=\{(\alpha,\beta)\in\RR^2:\psi(\alpha,\beta)\le h\}\);
\(|K_h|\) denotes its Euclidean area.  Then
\Cref{thm:origin-fiber-main} gives, for \(1<\rho<4\),
\begin{equation}\label{eq:intro-origin-fiber}
  \{c:(0,0,c)\in\cR_\rho\}=[-Z_0(\rho),Z_0(\rho)],
  \qquad
  Z_0(\rho)=\min_{h>0}\frac{(\log\rho+h)^2}{4|K_h|}.
\end{equation}
The minimizing parameter is unique, and the upper endpoint is attained by a
cyclically rooted and translated Wulff loop.  A half loop derived from this
optimizer provides the nonconvexity witness below saturation; the explicit
description of \(\cR_4\) proves nonconvexity at and above saturation.

\begin{hyperaddition}
The dual potential also admits a hyperbolic interpretation.  In suitable
Fermi coordinates on the hyperbolic plane of curvature \(-1\), it is
\(\log\cosh r\), where \(r\) is distance from a fixed point.  Thus
\(\partial K_h\) corresponds to a hyperbolic circle.  The Euclidean area
\(|K_h|\) becomes a weighted hyperbolic area, whose derivative is a complete
elliptic integral.  This formula gives both the small and large parameter
asymptotics used in the Wulff optimization
(\Cref{lem:ah-asymptotics}).
\end{hyperaddition}

The same pathwise variational structure governs the central front.  With
\(\Gamma_1:=0\), the map \(\rho\mapsto\Gamma_\rho\) is nondecreasing,
concave and continuous on \([1,\infty)\), and
\(\Gamma_\rho=1/4\) for \(\rho\ge4\)
(\Cref{thm:Gamma-global}).  The proof reduces the maximization to a convex
problem by a change to squared variables.
As \(\rho\downarrow1\),
\begin{equation}\label{eq:intro-weak-expansion}
  \Gamma_\rho
  =\frac{\log\rho}{\pi}-\frac{(\log\rho)^2}{4\pi}
   +O((\log\rho)^3).
\end{equation}
\Cref{thm:weak-branching} further shows that, for any choice of maximizing
paths as \(\rho\downarrow1\), the distance in \(L^2(0,1)\) of their velocity
profiles, normalized by \(\sqrt{\log\rho}\), from the two element set
\(\{t\mapsto\pm(\cos(\pi t/2),\sin(\pi t/2))\}\) tends to zero.  Moreover, for \(1<\rho<4\), retaining only the terminal
action constraint gives the larger value
\(4Z_0(\rho)>\Gamma_\rho\) (\Cref{prop:relaxation-main}), confirming that
the intermediate constraints remain active throughout the regime below
saturation.

\subsection*{Main results (iii): Finite generation corrections}

Finite generation behavior changes at the saturation threshold.  Below
saturation we prove a quantitative window, while at and above saturation we
identify the leading correction.
\Cref{thm:below-saturation-window} shows that for every fixed
\(1<\rho<4\), there is a deterministic \(C<\infty\), depending only on
\(\rho\) and the offspring law, such that, almost surely on \(\cS\),
\begin{equation}\label{eq:intro-subcritical-window}
  \abs{M_n^Z-\Gamma_\rho n^2}
  \le Cn^{3/2}\sqrt{\log n}
  \qquad\text{for all sufficiently large }n.
\end{equation}
At and above saturation, \Cref{thm:saturated-corrections-main} identifies the
leading loss exactly.  For \(\rho>4\), let \(\mathfrak d\) be the minimum
directional deficit along an ancestral segment leading to a vertex that
starts an infinite descendant subtree whose edges all point right or all
point left; see \eqref{eq:D-main}.  This random variable is finite almost
surely on \(\cS\).
If \(\rho=4\), assume in addition that \(\EE N^2<\infty\).  Then
\[
  \begin{aligned}
  \frac{\Gamma_\rho n^2-M_n^Z}{n}
    &\longrightarrow\frac{\mathfrak d}{2}
    &&\text{almost surely on }\cS\text{ if }\rho>4,\\
  \frac{\Gamma_4 n^2-M_n^Z}{n\log\log n}
    &\longrightarrow\frac1{2\log2}
    &&\text{in probability conditional on }\cS\text{ if }\rho=4.
  \end{aligned}
\]
The qualitative change in mechanism at \(\rho=4\) parallels the geometric
saturation transition.

\subsection*{Proof ideas}

The sample path ingredient is an enhanced Mogulskii principle
(\Cref{thm:enhanced-ldp}).  An exact discrete area identity places all
interpolated horizontal paths in a common bounded variation set on which the
area lift is continuous.  The contraction principle then yields a large
deviation principle for the horizontal path and its Heisenberg area
simultaneously.

The path shape theorem is proved through the two directed components of
Hausdorff convergence.  Outer containment uses the many to one formula: a
strict excess in action over the branching exponent leaves an exponentially
small and summable bound.  Inner coverage uses a time inhomogeneous Cram\'er
tilt and an additive branching martingale to track an admissible path with a
positive action margin.
The Kesten--Stigum theorem \cite{LyonsPemantlePeres1995} supplies many
independent subtree roots on \(\cS\), amplifying the tracking probability
to simultaneous realization.  When ancestral pieces are joined, the Heisenberg group law
produces a noncommutative cross term; the proof keeps this term explicit and
controls it at the Carnot scale.

The geometric arguments turn the path space limit into explicit information
about \(\cR_\rho\).  Cyclic rerooting reduces the origin fiber to a total
action problem, and Wulff duality with anisotropic isoperimetry gives
\eqref{eq:intro-origin-fiber}.
\hyperadd{The hyperbolic representation reduces the required Wulff area
calculation to a one dimensional elliptic integral.}
A change of variables in the action functional reveals the convex structure
behind \(\Gamma_\rho\), while the spectral gap of the Volterra operator
yields the weak branching expansion and convergence of the maximizing profiles.  The critical finite generation result uses the
survival event form of the
Dekking--Host minimal displacement theorem \cite{DekkingHost1991}.

\subsection*{Relation to earlier work}

Classical Euclidean shape theorems describe the asymptotic set of particle
positions through the competition between branching growth and the
Cram\'er rate function
\cite{Biggins1977Chernoff,Biggins1978,Biggins2026Corrigendum};
see also \cite{Shi2015}.
For the horizontal step law considered here, the limit is the convex set
\(\{v:L(v)\le\log\rho\}\), which is also the horizontal projection of our
spatial limit.  The difference lies in the central coordinate: it records
the ancestral area and produces a nonconvex three-dimensional set, not
determined by its support functions.

Pathwise growth laws for branching Brownian motion were studied by
Git \cite{Git1998} and Harris and Roberts \cite{HarrisRoberts2012}.
In particular, Harris and Roberts give almost sure bounds for population
counts in closed and open subsets of path space, with admissibility
depending on the action accumulated along every initial segment.
The constraints \(J_t(\omega)\le t\log\rho\) have the same genealogical
origin.  Here we establish a compact-set limit for discrete ancestral
paths together with their Heisenberg area lifts.  Evaluation at time one
then gives the spatial shape theorem.  Retaining the area is essential:
paths with the same horizontal terminal position can have different
central coordinates.

Related path constraints occur in branching random walks with absorption
and in time-inhomogeneous environments
\cite{Jaffuel2012,FangZeitouni2012,Mallein2015Interfaces,Mallein2015Inhomogeneous}.
Callegaro and Roberts \cite{CallegaroRoberts2025} obtain pathwise growth
and variational descriptions for a spatially dependent fragmentation
model, represented as a two-dimensional branching random walk.
In our model the branching and increment laws are homogeneous;
the path dependence enters through the noncommutative position
coordinate rather than through the environment.

For branching random walks on groups, existing work treats traces,
intersection properties and boundary geometry
\cite{BenjaminiMuller2012,CandelleroGilchMuller2012,Hutchcroft2020,DussauleWangYang2025}.
More directly related to spatial spread, Kaiser, Kl\"otzer, Kolesko and
Sava-Huss \cite{KaiserEtAl2026} identify maximal and minimal word-metric
speeds on free products of finite groups using the large deviation rate
function of the underlying walk.  Our concern is instead the full
anisotropically rescaled set of positions in a nilpotent group, including
the geometry lost when positions are reduced to a scalar distance.

The deterministic geometry at saturation belongs to the established
sub-Finsler description of Heisenberg asymptotic shapes
\cite{Pansu1983,DuchinMooney2014,Lokutsievskiy2021}.
More precisely, the saturated shape is the unit ball obtained from
horizontal \(\ell^1\) length.  We identify when the branching population
attains this reachable set and how the action constraints restrict the
shape below saturation.  The origin-fiber formula combines cyclic
rerooting with Wulff duality and anisotropic isoperimetry
\cite{EspositoFuscoTrombetti2005}.
These ballistic scales also differ from the diffusive nilpotent scales
in single-walk central and local limit theory
\cite{Alexopoulos2002,BenardBreuillard2025}.

Finally, the critical finite-generation correction is connected to
minimal-displacement results for branching random walks with
nonnegative increments \cite{Bramson1978,DekkingHost1991}.
At \(\rho=4\), the zero-defect directional subprocess is critical,
and the resulting \(\log\log n\) displacement scale becomes an
\(n\log\log n\) loss of central area.
Above this threshold, surviving one-direction descendant trees instead
give a random linear correction.  Thus the second-order behavior
reflects a directional branching transition within the same model.

\subsection*{Organization}

\Cref{sec:model} defines the model and states the principal shape and geometry
results, together with the finite generation results at and above saturation.
\Cref{sec:convex} develops the required convex analytic and path space
tools.  \Cref{sec:proof} proves the path space theorem and the finite
generation estimates, including the quantitative window below saturation.
\Cref{sec:geometry} establishes the geometry of the limit shape and the Wulff
description of the origin fiber.  Finally, \Cref{sec:central-front} studies
the dependence of the central front on the offspring mean.
\section{Model and main results}\label{sec:model}

This section introduces the Heisenberg branching random walk and the large
deviation cost of one ancestral path.  Its principal probabilistic result is
the almost sure Carnot shape theorem for the full enhanced path cloud.  We
then state the universal nonconvexity and saturation results for the limit
shape, analyze the central front as a function of the offspring mean, and
conclude with the sharp finite generation corrections at and above
saturation.

\subsection{The model and enhanced ancestral paths}

\subsubsection{Branching genealogy}

All random variables below are constructed on one complete probability
space.  Let \(\mathbb T\) be a Galton--Watson tree rooted at
\(\varnothing\).  We use the Ulam--Harris labeling, so the children of each
vertex are ordered.  For a vertex \(u\), \(\abs{u}\) denotes its generation, and \(u_k\) its
ancestor in generation \(k\) for \(0\le k\le\abs{u}\).  Each vertex has an offspring number, and these numbers are
independent copies of a nonnegative integer valued random variable \(N\).  We
write \(\PP\) and \(\EE\) for probability and expectation.  For \(j\ge0\),
let
\[
  \mathcal N_j:=\#\{u\in\mathbb T:\abs{u}=j\}
\]
be the number of particles in generation \(j\).  With
\(\log_+x:=\max\{\log x,0\}\) for \(x>0\), and with
\(N\log_+N\) defined to be zero on \(\{N=0\}\), we assume throughout that
\begin{equation}\label{eq:offspring-assumption}
  \rho:=\EE N>1,
  \qquad
  \EE[N\log_+N]<\infty.
\end{equation}
Since \(\rho>1\), the survival event
\[
  \cS:=\{\mathbb T\text{ is infinite}\}
\]
has positive probability.  The logarithmic moment condition is used only in
the amplification step through the Kesten--Stigum theorem.

\subsubsection{Heisenberg increments and coordinates}
\label{sec:heisenberg-increments}

The discrete Heisenberg group is
\[
  \HH(\ZZ)=\ZZ^3,
  \qquad
  (x,y,z)(a,b,c)=(x+a,y+b,z+c+xb).
\]
The map \((x,y,z)\mapsto(x,y)\) is the horizontal projection, so \(x\) and
\(y\) are the horizontal coordinates.  The center of \(\HH(\ZZ)\) is
\(\{(0,0,z):z\in\ZZ\}\), and \(z\) is the central coordinate.  If
\(g=(x,y,z):[0,T]\to\HH(\ZZ)\) is a path, then \((x,y)\) is its horizontal
path.

We denote the four horizontal generators by
\[
  \R=(1,0,0),\qquad \Lft=(-1,0,0),\qquad
  \U=(0,1,0),\qquad \Dwn=(0,-1,0).
\]

Let
\[
  S:=\{(1,0),(-1,0),(0,1),(0,-1)\},
\]
and let \(\mu\) be the uniform probability law on \(S\).  Independently of
\(\mathbb T\), attach to every edge \(e\) an independent increment \(\xi_e\)
with law \(\mu\); all offspring numbers and all edge increments are mutually
independent.  We
view \(\xi_e=(\xi_e^{(1)},\xi_e^{(2)})\) as the Heisenberg element
\((\xi_e^{(1)},\xi_e^{(2)},0)\), and multiply increments on the right.  For
\(u\in\mathbb T\), let
\[
  V(u)=(X_u,Y_u,Z_u)\in\HH(\ZZ)
\]
be the ordered product along the ancestral line from the root to \(u\), with
\(V(\varnothing)=(0,0,0)\).  For \(1\le j\le\abs{u}\), write
\(\xi_{u_j}\) for the increment on the edge from \(u_{j-1}\) to \(u_j\).
Then
\begin{equation}\label{eq:coordinate-recursion}
  X_{u_k}=\sum_{j=1}^k\xi_{u_j}^{(1)},
  \qquad
  Y_{u_k}=\sum_{j=1}^k\xi_{u_j}^{(2)},
  \qquad
  Z_{u_k}=\sum_{j=1}^kX_{u_{j-1}}\xi_{u_j}^{(2)}.
\end{equation}
The sequence \(((X_{u_k},Y_{u_k}))_{k=0}^{\abs{u}}\) is the horizontal
ancestral path of \(u\).  We call the last sum in
\eqref{eq:coordinate-recursion} the discrete Heisenberg area of this path up
to generation \(k\).  Thus \(Z_{u_k}\) is simultaneously the central
coordinate of \(V(u_k)\) and the discrete Heisenberg area of its horizontal
ancestral path segment through generation \(k\).  The continuous analogue of this functional is defined in \eqref{eq:action-area}.

\subsubsection{Enhanced ancestral paths}

For a particle \(u\) with \(\abs{u}=n\), the sequence
\((V(u_k))_{k=0}^n\) is its ancestral trajectory.  Define its rescaled
enhanced ancestral path by
\[
  Q_u^{(n)}=(\omega_u^{(n)},\zeta_u^{(n)})
  \in
  \cX:=C([0,1];\RR^2)\times C([0,1];\RR)
\]
as follows.  Both components are linear on each interval
\([k/n,(k+1)/n]\), for \(0\le k<n\), and their grid values are
\begin{align}
  \omega_u^{(n)}(k/n)
  &=\frac1n(X_{u_k},Y_{u_k}),\label{eq:horizontal-interp}\\
  \zeta_u^{(n)}(k/n)
  &=\frac1{n^2}Z_{u_k},\label{eq:central-interp}
\end{align}
for \(k=0,\ldots,n\).  Thus \(\omega_u^{(n)}\) is the rescaled interpolation
of the horizontal ancestral path, while \(\zeta_u^{(n)}\) is the rescaled
interpolation of its discrete Heisenberg area, equivalently its central
coordinate.  The factors \(n^{-1}\) and \(n^{-2}\) are the Carnot scaling factors of the
horizontal and central coordinates.

We equip \(\cX\) with
\begin{equation}\label{eq:path-metric}
  d_{\cX}((\omega,\zeta),(\widetilde\omega,\widetilde\zeta))
  :=\max\{\norm{\omega-\widetilde\omega}_\infty,
          \norm{\zeta-\widetilde\zeta}_\infty\}.
\end{equation}
For vectors in a finite dimensional Euclidean space, \(\norm{\cdot}_1\),
\(\norm{\cdot}_2\), and
\(\norm{\cdot}_\infty\) denote the usual \(\ell^1\), Euclidean, and maximum
norms.  For a vector valued continuous path \(f\), we use
\(\norm{f}_\infty:=\sup_t\norm{f(t)}_\infty\); for a scalar path this is the
usual sup norm.

\subsubsection{Random sets and Hausdorff distance}

The random enhanced path set and set of rescaled particle positions in
generation \(n\) are
\begin{equation}\label{eq:random-sets}
  \cQ_n:=\{Q_u^{(n)}:\abs{u}=n\},
  \qquad
  \cR_n:=\left\{
  \left(\frac{X_u}{n},\frac{Y_u}{n},\frac{Z_u}{n^2}\right):\abs{u}=n
  \right\}.
\end{equation}
The set \(\cQ_n\) records the complete enhanced ancestral paths, whereas
\(\cR_n\) records the rescaled positions of the particles in generation
\(n\).  On \(\cS\), both sets are nonempty and finite, hence compact, for
every \(n\).

For a point \(q\) and a nonempty set \(E\) in a metric space
\((\mathsf Y,d_\mathsf Y)\), write
\[
  d_\mathsf Y(q,E):=\inf_{r\in E}d_\mathsf Y(q,r).
\]
For nonempty compact sets \(E,F\subset\mathsf Y\), their Hausdorff distance is
\[
  \dH^\mathsf Y(E,F)
  :=\max\left\{\sup_{q\in E}d_\mathsf Y(q,F),
                 \sup_{r\in F}d_\mathsf Y(r,E)\right\}.
\]
For \(\mathsf Y=\cX\), this definition uses \(d_\cX\).  For subsets of a
finite dimensional Euclidean space, the Euclidean metric is used unless
stated otherwise.

\subsection{An enhanced Mogulskii principle for the Heisenberg walk}

This subsection formulates the sample path large deviation principle for one
unbranched walk.  Its rate functions quantify the costs of individual
ancestral paths in the branching model.

\subsubsection{The single walk and its Cram\'er transform}
\label{sec:single-walk}

Let \(\xi\) have law \(\mu\), and let \((\xi_j)_{j\ge1}\) be independent
copies of \(\xi\).  Starting from the identity, define the unbranched
Heisenberg walk as follows.  Set
\(V_0=(0,0,0)\) and
\(V_k=(X_k,Y_k,Z_k):=V_{k-1}(\xi_k^{(1)},\xi_k^{(2)},0)\) for \(k\ge1\).
For \(n\ge1\), let \(Q_n=(\omega_n,\zeta_n)\in\cX\) have components that
are linear on every interval \([k/n,(k+1)/n]\), for \(0\le k<n\), with
grid values
\begin{equation}\label{eq:Qn-grid-values}
  \begin{aligned}
  &\omega_n(0)=(0,0),\qquad \zeta_n(0)=0,\qquad\text{and}\qquad\\
  &\omega_n(k/n)=\frac1n\sum_{j=1}^k\xi_j
    =\frac1n(X_k,Y_k),\qquad
    \zeta_n(k/n)=\frac1{n^2}Z_k,\qquad 1\le k\le n.
  \end{aligned}
\end{equation}

For \(\theta=(\theta_1,\theta_2)\in\RR^2\), define the logarithmic moment
generating function of \(\xi\) by
\begin{equation}\label{eq:psi-def}
  \psi(\theta_1,\theta_2)
  :=\log\EE\exp\{\theta_1\xi^{(1)}+\theta_2\xi^{(2)}\}
  =\log\frac{\cosh\theta_1+\cosh\theta_2}{2}.
\end{equation}
We write \(\ip{\cdot}{\cdot}\) for the Euclidean inner product.
For \(v\in\RR^2\), define the Legendre transform of \(\psi\) by
\begin{equation}\label{eq:L-def}
  L(v):=\sup_{\theta\in\RR^2}\{\ip{\theta}{v}-\psi(\theta)\}.
\end{equation}
The effective domain of \(L\), namely the set on which \(L\) is finite, is
the diamond
\begin{equation}\label{eq:D-def}
  D:=\{v=(v_1,v_2):\abs{v_1}+\abs{v_2}\le1\}.
\end{equation}

\subsubsection{Path functionals and the enhanced Mogulskii principle}

We first fix the path notation.  For an interval \([a,b]\) with \(a<b\),
\(\ac([a,b];\RR^d)\) denotes the space of
absolutely continuous paths from \([a,b]\) to \(\RR^d\), and a dot denotes
the derivative defined almost everywhere.  For a path
\(\omega=(x,y)\in\ac([0,1];\RR^2)\) with \(\omega(0)=0\), and for
\(0\le t\le1\), define the horizontal action and continuous Heisenberg area
accumulated by time \(t\) by
\begin{equation}\label{eq:action-area}
  J_t(\omega):=\int_0^tL(\dot\omega(s))\dd s,
  \qquad
  A_t(\omega):=\int_0^t x(s)\dot y(s)\dd s.
\end{equation}
We write \(A(\omega):=A_1(\omega)\).  With \(A_\cdot(\omega)\) denoting the
path \(t\mapsto A_t(\omega)\), define the enhanced area lift of \(\omega\) by
\begin{equation}\label{eq:lift-def}
  q_\omega:=(\omega,A_\cdot(\omega))\in\cX.
\end{equation}
By \Cref{lem:exact-area}, \(Q_n=q_{\omega_n}\) exactly for the step set
\(S\).

The horizontal action and area lift determine the rate functions on the full
path space and after restriction to \([0,t]\) for each \(t\in(0,1]\).
For \(t\in(0,1]\), let
\begin{equation}\label{eq:restriction-map}
  \cX_t:=C([0,t];\RR^2)\times C([0,t];\RR),
  \qquad
  \pi_t(\omega,\zeta)
  :=\bigl(\omega|_{[0,t]},\zeta|_{[0,t]}\bigr).
\end{equation}
We equip \(\cX_t\) with the product uniform metric obtained by restricting
\eqref{eq:path-metric} to \([0,t]\); thus
\(\pi_t:\cX\to\cX_t\) is continuous.  For \(t\in(0,1]\), define
\(I_t:\cX_t\to[0,\infty]\) by
\begin{equation}\label{eq:It-rate}
  I_t(q)=
  \begin{cases}
    \displaystyle\int_0^tL(\dot\omega(s))\dd s,
      &\begin{gathered}
        q=(\omega,\zeta),\\
        \omega=(x,y)\in\ac([0,t];\RR^2),\quad \omega(0)=0,\\
        \displaystyle \zeta(r)=\int_0^r x(s)\dot y(s)\dd s
        \quad\text{for every }r\in[0,t]
       \end{gathered},\\[3mm]
    +\infty,&\text{otherwise}.
  \end{cases}
\end{equation}
Under the natural identification \(\cX_1=\cX\), write \(I:=I_1\).

\begin{theorem}[Enhanced Mogulskii principle]\label{thm:enhanced-ldp}
The sequence of laws of \(Q_n\) satisfies a good large deviation principle
on \(\cX\), with speed \(n\) and rate function \(I\).  For every fixed
\(t\in(0,1]\), the sequence of laws of \(\pi_tQ_n\) satisfies a good large
deviation principle on \(\cX_t\), again with speed \(n\), and rate function
\(I_t\).
\end{theorem}

\subsection{Almost sure Carnot shape theorems}

This subsection combines the branching growth rate with the single walk rate
function.  It first defines the path space and spatial limit sets and
states their almost sure Hausdorff limits.  It then derives uniform convergence
of the spatial support functions and the limit of the maximal central
coordinate.

\subsubsection{Path space and spatial Hausdorff limits}

At time \(t\), the exponential growth of the branching population is measured
by \(t\log\rho\).  Comparing this growth with the action \(J_t\) accumulated
on \([0,t]\) leads to the admissible horizontal path set
\begin{equation}\label{eq:C-rho}
  \cC_\rho
  :=\left\{
  \omega\in\ac([0,1];\RR^2):
  \omega(0)=0,\quad
  J_t(\omega)\le t\log\rho\text{ for every }t\in[0,1]
  \right\}.
\end{equation}
For \(\omega\in\cC_\rho\), write \(\omega=(x,y)\).  The
corresponding path space and spatial limit sets are
\begin{equation}\label{eq:Q-R-rho}
  \cQ_\rho:=\{q_\omega:\omega\in\cC_\rho\},
  \qquad
  \cR_\rho:=\{(x(1),y(1),A(\omega)):\omega\in\cC_\rho\}.
\end{equation}

The first result identifies \(\cQ_\rho\) as the almost sure Hausdorff limit
of the full enhanced ancestral path cloud.

\begin{theorem}[Path space Carnot shape theorem]\label{thm:path-shape}
Assume \eqref{eq:offspring-assumption}.  The deterministic set \(\cQ_\rho\)
is nonempty and compact.  On the survival event \(\cS\), almost surely,
\begin{equation}\label{eq:path-shape-conv}
  \dH^{\cX}(\cQ_n,\cQ_\rho)\longrightarrow0.
\end{equation}
\end{theorem}

Outer containment is a first moment large deviation bound; inner coverage is the
branching part and uses the Kesten--Stigum theorem.

Evaluation at time \(1\) is Lipschitz and sends \(\cQ_n\) and
\(\cQ_\rho\) to \(\cR_n\) and \(\cR_\rho\), respectively.  Thus the path
space theorem immediately yields the spatial shape theorem.

\begin{corollary}[Spatial Carnot shape theorem]\label{cor:endpoint-shape}
Under the assumptions of \Cref{thm:path-shape}, the set \(\cR_\rho\) is a
nonempty compact subset of \(\RR^3\).  On \(\cS\), almost surely,
\begin{equation}\label{eq:endpoint-shape-conv}
  \dH^{\RR^3}(\cR_n,\cR_\rho)\longrightarrow0.
\end{equation}
\end{corollary}

\subsubsection{Support functions and the central maximum}

We next derive convergence of the directional support functions and the
maximal central displacement from spatial convergence.  Define the central
front, the largest central coordinate in \(\cR_\rho\), by
\begin{equation}\label{eq:Gamma-def}
  \Gamma_\rho
  :=\max_{(a,b,c)\in\cR_\rho}c
  =\sup_{\omega\in\cC_\rho}A(\omega),
\end{equation}
and the generation \(n\) maximum
\[
  M_n^Z:=\max_{\abs{u}=n}Z_u.
\]
Here and below, a maximum over an empty generation is assigned the value
\(-\infty\).  All limit assertions involving these maxima are made on
\(\cS\), where the generations are nonempty.

For \(d\ge1\) and a nonempty compact set \(E\subset\RR^d\), define its
directional support function by
\begin{equation}\label{eq:support-function-def}
  h_E(\lambda):=\sup_{z\in E}\ip{\lambda}{z},
  \qquad \lambda\in\RR^d.
\end{equation}
On nonempty generations, write
\(h_n:=h_{\cR_n}\) and \(h_\rho:=h_{\cR_\rho}\).  Explicitly,
\begin{equation}
  h_n(\lambda)
  :=\max_{\abs{u}=n}\left\{
  \lambda_1\frac{X_u}{n}
  +\lambda_2\frac{Y_u}{n}
  +\lambda_3\frac{Z_u}{n^2}
  \right\}.
\end{equation}
We use \(\conv(E)\) for the Euclidean convex hull of a set \(E\).

Hausdorff convergence controls the support functions uniformly over the
Euclidean unit ball.  Taking the vertical direction \(e_3=(0,0,1)\) then
gives the central front limit.

\begin{corollary}[Support functions and the central maximum]\label{cor:front-support}
Under \eqref{eq:offspring-assumption}, on \(\cS\), almost surely,
\begin{equation}\label{eq:support-uniform}
  \sup_{\norm{\lambda}_2\le1}\abs{h_n(\lambda)-h_\rho(\lambda)}
  \longrightarrow0.
\end{equation}
In particular,
\begin{equation}\label{eq:front-law}
  \frac{M_n^Z}{n^2}\longrightarrow\Gamma_\rho.
\end{equation}
\end{corollary}

Applying \Cref{lem:support-reconstruction} to \(\cR_\rho\) gives
\begin{equation}\label{eq:support-halfspace-reconstruction}
  \conv(\cR_\rho)
  =\bigcap_{\lambda\in\RR^3}
  \left\{z\in\RR^3:\ip{\lambda}{z}\le h_\rho(\lambda)\right\}.
\end{equation}
Thus \(h_\rho\) determines \(\conv(\cR_\rho)\).  Since
\(
h_{\conv(\cR_\rho)}=h_\rho
\), support functions cannot distinguish \(\cR_\rho\) from its convex hull.

This loss of information occurs for every supercritical offspring mean and
is a principal departure from classical Euclidean branching random walk
shape theory, where the limit shape is a convex Cram\'er sublevel set
\cite{Biggins1977Chernoff,Biggins1978,Biggins2026Corrigendum,Shi2015}.

\begin{analystcheck}
\begin{theorem}[Universal nonconvexity]\label{thm:nonconvex-main}
\begin{revisioncheck}
For every \(\rho>1\), the limit shape \(\cR_\rho\) is nonconvex as a
subset of \(\RR^3\).
\end{revisioncheck}
\end{theorem}
\end{analystcheck}

The next subsection determines the saturated limit shape and the exact
origin fiber below saturation, which together provide the geometric
mechanism behind this theorem.

\subsection{\texorpdfstring{Saturation and universal nonconvexity of
\(\cR_\rho\)}{Saturation and universal nonconvexity of R rho}}

\begin{additioncheck}
For the unbranched walk \(V_n=(X_n,Y_n,Z_n)\) from
\Cref{sec:single-walk}, reversing the increment order preserves the
probability and horizontal endpoint of every step sequence and sends
\(Z_n\) to \(X_nY_n-Z_n\).  Hence, whenever
\(\PP((X_n,Y_n)=(x,y))>0\),
\begin{align}
  \operatorname{Law}(Z_n\mid X_n=x,Y_n=y)
  &=\operatorname{Law}(xy-Z_n\mid X_n=x,Y_n=y),
  \label{eq:conditional-saddle-symmetry}\\
  \EE[Z_n\mid X_n=x,Y_n=y]
  &=\frac{xy}{2}.
  \label{eq:conditional-saddle-mean}
\end{align}
Since the expected branching random walk count at a site is \(\rho^n\) times
the corresponding single walk probability, the same reflection symmetry
holds for the expected vertical particle counts.  Under Carnot scaling its
fixed point becomes \(c=ab/2\).  We call this surface the \emph{saddle sheet}.
\end{additioncheck}

\subsubsection{Geometry relative to the saddle sheet}

To describe the part of the saddle sheet \(c=ab/2\) relevant to the limiting
shape \(\cR_\rho\), let
\begin{equation}\label{eq:horizontal-projection-def}
  \cD_\rho:=\{(a,b)\in\RR^2:(a,b,c)\in\cR_\rho
  \text{ for some }c\in\RR\}
\end{equation}
be the horizontal projection of \(\cR_\rho\).  For each
\((a,b)\in\cD_\rho\), define the vertical fiber above \((a,b)\) by
\begin{equation}\label{eq:fiber-def}
  \cI_\rho(a,b):=\{c\in\RR:(a,b,c)\in\cR_\rho\}.
\end{equation}
For \(r\ge0\), the step two Carnot dilation by factor \(r\) is
\begin{equation}\label{eq:dilation}
  \delta_r(a,b,c):=(ra,rb,r^2c).
\end{equation}
For a set \(E\subset\RR^3\), write
\(\delta_rE:=\{\delta_r(p):p\in E\}\).

\begin{additioncheck}
To state the saturated shape, for \(u,v\ge0\) with \(u+v\le1\), define the
fiber half width
\begin{equation}\label{eq:Delta4-def}
  \Delta_4(u,v):=
  \begin{cases}
    \displaystyle\frac{(1+u+v)^2-8uv}{16},
      &3u-v\le1\ \text{and}\ 3v-u\le1,\\[2mm]
    \displaystyle\frac{u(1-u)}2,
      &3u-v\ge1,\\[2mm]
    \displaystyle\frac{v(1-v)}2,
      &3v-u\ge1.
  \end{cases}
\end{equation}
The expressions agree on their common boundaries, so \(\Delta_4\) is
well defined and continuous on the horizontal diamond in the first quadrant.
\end{additioncheck}

The geometry of the limit shape changes at \(\rho=4\).  At and above this threshold
the next theorem determines the complete shape.  Below the threshold it
compares \(\cR_\rho\) quantitatively with the saturated body.

\begin{analystcheck}
\begin{theorem}[Limit shape across saturation]\label{thm:geometry-summary}
\begin{enumerate}[label=\textup{(\roman*)},leftmargin=2.1em]
\item If \(\rho\ge4\), the limit shape is explicit:
\begin{analystadditioncheck}
\begin{equation}\label{eq:complete-saturated-shape-main}
  \cR_\rho
  =
  \left\{
  (a,b,c)\in\RR^3:
  \abs{a}+\abs{b}\le1,\quad
  \abs{c-\frac{ab}{2}}\le\Delta_4(\abs{a},\abs{b})
  \right\}.
\end{equation}
\end{analystadditioncheck}
\begin{revisioncheck}
Consequently, the limit shape saturates:
\begin{equation}\label{eq:saturation}
  \cR_\rho=\cR_4.
\end{equation}
\end{revisioncheck}

\item If \(1<\rho<4\), put
\revisionmark{\(r_\rho:=\sqrt{\log\rho/\log4}\).}
\begin{analystadditioncheck}
Then
\begin{align}
  \delta_{r_\rho}\cR_4&\subset\cR_\rho\subset\cR_4,
  \label{eq:saturation-sandwich-main}\\
  \dH^{\RR^3}(\cR_\rho,\cR_4)
  &\le
  \left\{(1-r_\rho)^2+\frac{(1-r_\rho^2)^2}{16}\right\}^{1/2}.
  \label{eq:saturation-Hausdorff-main}
\end{align}
\end{analystadditioncheck}
\end{enumerate}
\end{theorem}
\end{analystcheck}

These results are established in \Cref{sec:geometry}.

\subsubsection{The origin fiber and Wulff duality}

For \(1<\rho<4\), define the upper endpoint of the fiber above the
horizontal origin by
\begin{equation}\label{eq:Z0-main}
  Z_0(\rho):=\max\{c:(0,0,c)\in\cR_\rho\}.
\end{equation}
At the horizontal origin, every admissible path is a closed loop.  The cyclic
rooting lemma in \Cref{lem:cyclic-rooting} shows that a closed loop with total
action at
most \(\log\rho\) can be cyclically shifted and translated to satisfy every
intermediate action constraint, without changing its total action or signed
area.  The origin fiber is therefore governed by a single total action
constraint, to which the planar Wulff inequality applies.

For \(h>0\), let \(K_h\) be the log Laplace sublevel set and let \(a(h)\) be
its Euclidean area:
\begin{equation}\label{eq:Kh-ah-main}
  K_h:=\{\theta\in\RR^2:\psi(\theta)\le h\},
  \qquad
  a(h):=|K_h|.
\end{equation}
\begin{hyperaddition}
Let \(\mathbf K\) denote the complete elliptic integral of the first kind,
with modulus \(k\):
\[
  \mathbf K(k):=\int_0^{\pi/2}
      \frac{\dd\varphi}{\sqrt{1-k^2\sin^2\varphi}},
  \qquad 0\le k<1.
\]
The hyperbolic representation in \Cref{lem:ah-asymptotics} gives
\[
  a'(h)=8\mathbf K\!\left(\sqrt{1-\e^{-2h}}\right),
  \qquad
  a(h)=8\int_0^h\mathbf K\!\left(\sqrt{1-\e^{-2s}}\right)\dd s.
\]
\end{hyperaddition}
\hyperrewrite{In particular, \(a\) is smooth and \(a'(h)>0\) for \(h>0\).}
The following theorem evaluates the fiber and exhibits an extremizing
Wulff loop.

\begin{analystcheck}
\begin{theorem}[Exact origin fiber below saturation]\label{thm:origin-fiber-main}
\begin{revisioncheck}
For every \(1<\rho<4\), \additionmark{the origin fiber is}
\begin{equation}\label{eq:origin-fiber-interval-main}
  \additionmark{\cI_\rho(0,0)=[-Z_0(\rho),Z_0(\rho)]}.
\end{equation}
Its upper endpoint satisfies
\begin{equation}\label{eq:origin-Wulff-main}
  Z_0(\rho)=\min_{h>0}\frac{(\log\rho+h)^2}{4a(h)}.
\end{equation}
The minimizer \additionmark{\(h_\rho^\star\in(0,\infty)\)} is unique.  It is
characterized by the first identity below, while the second gives the
optimal value:
\begin{equation}\label{eq:origin-critical-main}
  \log\rho=\frac{2a(\additionmark{h_\rho^\star})}
    {a'(\additionmark{h_\rho^\star})}
    -\additionmark{h_\rho^\star},
  \qquad
  Z_0(\rho)=\frac{a(\additionmark{h_\rho^\star})}
    {a'(\additionmark{h_\rho^\star})^2}.
\end{equation}
The upper endpoint is attained by a cyclically rooted and translated Wulff
loop associated with \(K_{\additionmark{h_\rho^\star}}\).
\end{revisioncheck}
\end{theorem}
\end{analystcheck}

The Wulff loop supplies the nonconvexity witness below saturation, while the
explicit saturated shape and the identity \(\cR_\rho=\cR_4\) cover
\(\rho\ge4\).  These are the two mechanisms behind
\Cref{thm:nonconvex-main}.

\begin{additioncheck}
Figure~\ref{fig:rho4-cloud-envelopes} compares the exact boundary sheets
\(c=ab/2\pm\Delta_4(\abs{a},\abs{b})\) in
\eqref{eq:complete-saturated-shape-main} with simulated particle positions
and empirical fiber boundaries.
The simulations use Poisson\((4)\) offspring.
Panel~\textup{(a)} shows the distinct sites in \(\cR_{n=16}\) from one
realization with \(\mathcal N_{16}>0\), suppressing particle multiplicities.
Panel~\textup{(b)} shows the observed lower and upper fiber boundaries from
\(20\) realizations with \(\mathcal N_{16}>0\), retaining horizontal sites
seen in at least five realizations.
Panel~\textup{(c)} shows the exact boundary sheets, with \(\cR_4\) lying
between them.
The plot uses \(z\) for the normalized central coordinate, denoted by \(c\)
in the shape formulas.
All panels use the same coordinates, viewing angle, and vertical scale.
\end{additioncheck}

\subsection{The central front across offspring means}

\begin{revisioncheck}
\Cref{cor:front-support} identifies \(\Gamma_\rho\) as the almost sure limit
of the normalized maximal central coordinate, while the preceding subsection
places it within the geometry of \(\cR_\rho\).  We now study
\(\Gamma_\rho\) as a deterministic function of the offspring mean.  We first
give its global parameter dependence, then derive its weak branching
expansion.  A further estimate near saturation is given in
\Cref{prop:near-four}.
\end{revisioncheck}

\subsubsection{Global concavity and saturation}

\begin{revisioncheck}
The next theorem describes the dependence of the central front on the
offspring mean throughout the full parameter range.
\end{revisioncheck}

\begin{analystcheck}
\begin{theorem}[Global concavity of the central front]
\label{thm:Gamma-global}
\begin{revisioncheck}
With \(\Gamma_1:=0\), the map \(\rho\mapsto\Gamma_\rho\) is
nondecreasing, concave, and continuous on \([1,\infty)\).  Moreover,
\(\Gamma_\rho=1/4\) for every \(\rho\ge4\).
\end{revisioncheck}
\end{theorem}
\end{analystcheck}

\subsubsection{Weak branching asymptotics}

\begin{revisioncheck}
Near \(\rho=1\), the velocities of maximizing paths are small and the
quadratic part of \(L\) governs the leading variational problem.
\end{revisioncheck}

Define the rotating quadratic extremal \(v_0:[0,1]\to\RR^2\) by
\begin{equation}\label{eq:v0-main}
  v_0(t):=\left(\cos\frac{\pi t}{2},\sin\frac{\pi t}{2}\right).
\end{equation}
The next theorem gives the weak branching expansion and shows that the
normalized velocity of every maximizing path approaches one of
\(\pm v_0\).  Such maximizing paths exist because \(\cR_\rho\) is compact.

\begin{analystcheck}
\begin{theorem}[Sharp weak branching asymptotics]
\label{thm:weak-branching}
\begin{revisioncheck}
\begin{enumerate}[label=\textup{(\roman*)},leftmargin=2.1em]
\item As \(\rho\downarrow1\),
\begin{equation}\label{eq:weak-two-sided}
  \Gamma_\rho
  =\frac{\log\rho}{\pi}-\frac{(\log\rho)^2}{4\pi}
    +\additionmark{O\bigl((\log\rho)^3\bigr)}.
\end{equation}
\item For every choice of maximizing paths
\(\omega_\rho\in\cC_\rho\), \(\rho>1\), satisfying
\(A(\omega_\rho)=\Gamma_\rho\), one has, as \(\rho\downarrow1\),
\begin{equation}\label{eq:weak-rigidity-main}
  \inf_{\sigma\in\{-1,1\}}
  \left\|\frac{\dot\omega_\rho}{\sqrt{\log\rho}}-\sigma v_0\right\|_{L^2(0,1)}
  \longrightarrow0.
\end{equation}
\end{enumerate}
\end{revisioncheck}
\end{theorem}
\end{analystcheck}

\begin{revisioncheck}
At weak branching, velocities have scale \(\sqrt{\log\rho}\); the quadratic
area energy problem selects \(\pm v_0\), while the quartic part of \(L\)
gives the second order correction.
\end{revisioncheck}

For comparison, \Cref{prop:relaxation-main} shows that retaining only the
action constraint at time one gives the value \(4Z_0(\rho)\), strictly larger
than \(\Gamma_\rho\) for every \(1<\rho<4\).  Thus the intermediate time
constraints remain active throughout the regime below saturation.

\subsection{Sharp finite generation corrections at and above saturation}

At and above saturation, \(\Gamma_\rho=1/4\), and retained descendant trees
in extremal horizontal directions determine the first correction to
\(\Gamma_\rho n^2\).  To state the result, recall the right and left
horizontal generators
\(\R\) and \(\Lft\) from
\Cref{sec:heisenberg-increments}.  For a vertex \(u\), let
\(\mathbb T_u^{\R}\) be the descendant tree obtained by retaining only edges
labelled \(\R\), and define \(\mathbb T_u^{\Lft}\) analogously.  When
\(\rho>4\), use the convention that a minimum over the empty set equals
\(+\infty\), and define the right and left defects by
\begin{align}
  \mathfrak d^+&:=\min\{\abs{u}-X_u:\mathbb T_u^{\R}\text{ is infinite}\},
    \label{eq:Dplus-main}\\
  \mathfrak d^-&:=\min\{\abs{u}+X_u:\mathbb T_u^{\Lft}\text{ is infinite}\},
    \label{eq:Dminus-main}\\
  \mathfrak d&:=\mathfrak d^+\wedge\mathfrak d^- .\label{eq:D-main}
\end{align}
Both minima are finite almost surely on \(\cS\).  Thus \(\mathfrak d\)
measures the smallest ancestral defect relative to one of the two extremal
horizontal directions.  The following theorem identifies the first
correction to \(\Gamma_\rho n^2\) in the two saturated regimes.

\begin{theorem}[Sharp saturated corrections]\label{thm:saturated-corrections-main}
Assume \eqref{eq:offspring-assumption}.
\begin{enumerate}[label=\textup{(\roman*)},leftmargin=2.1em]
\item If \(\rho>4\), then on \(\cS\), almost surely,
\begin{equation}\label{eq:super-saturated-main}
  \frac{\Gamma_\rho n^2-M_n^Z}{n}
  \longrightarrow\frac{\mathfrak d}{2}.
\end{equation}
\item Let \(\rho=4\) and assume in addition that \(\EE N^2<\infty\).
Conditionally on \(\cS\),
\begin{equation}\label{eq:critical-saturated-main}
  \frac{\Gamma_\rho n^2-M_n^Z}{n\log\log n}
  \xrightarrow{\PP}\frac{1}{2\log2}.
\end{equation}
\end{enumerate}
\end{theorem}

\section{Convex and path space preliminaries}\label{sec:convex}

This section establishes the analytic tools used throughout the paper: convex
duality for the horizontal Cram\'er transform, reconstruction of convex hulls
from support functions, continuity of the area lift on bounded variation
paths, a compact ambient path set, and density of strict admissible paths.

\begin{analystcheck}

\subsection{The horizontal Cram\'er transform}

Recall the horizontal step set \(S\) and its uniform law \(\mu\) from
\Cref{sec:heisenberg-increments}.

\begin{revisioncheck}
Define the quartic polynomial \(Q_4:\RR^2\to\RR\) by
\[
  Q_4(v):=\frac16(v_1^4+v_2^4)+v_1^2v_2^2.
\]
The next lemma states the properties of \(\psi\) and \(L\) used throughout.
\end{revisioncheck}

\begin{lemma}\label{lem:convex-facts}
The following statements hold.
\begin{enumerate}[label=\textup{(\roman*)},leftmargin=2.1em]
\item \(\psi\) is real analytic and strictly convex.  Its gradient is a diffeomorphism from \(\RR^2\) onto \(\operatorname{int}D\).
\item \(L\) is lower semicontinuous on \(\RR^2\), finite and continuous on \(D\), and infinite outside \(D\).  For \(v\in D\), minimizing over probability weights \((p_s)_{s\in S}\) and using the convention \(0\log0=0\),
\begin{equation}\label{eq:entropy-L}
  L(v)=\min\left\{
  \sum_{s\in S}p_s\log(4p_s):
  p_s\ge0,\ \sum_{s\in S}p_s=1,\
  \sum_{s\in S}p_s s=v
  \right\}.
\end{equation}
In particular,
\begin{equation}\label{eq:L-range}
  0\le L(v)\le\log4,
  \qquad v\in D.
\end{equation}
\item
\begin{revisioncheck}
Let \(\phi:\RR\to[0,\infty]\) be the one dimensional Rademacher Cram\'er
transform,
\begin{equation}\label{eq:phi-def}
  \begin{aligned}
  \phi(z)&:=\sup_{s\in\RR}\{sz-\log\cosh s\}\\
  &=
    \begin{cases}
      \displaystyle\frac12\bigl\{(1+z)\log(1+z)
        +(1-z)\log(1-z)\bigr\},&\abs z\le1,\\[1mm]
      +\infty,&\abs z>1.
    \end{cases}
  \end{aligned}
\end{equation}
where \(0\log0:=0\).  For \(\abs z\le1\),
\begin{equation}\label{eq:phi-series}
  \phi(z)=\sum_{k\ge1}\frac{z^{2k}}{2k(2k-1)}.
\end{equation}
For every \(v=(v_1,v_2)\in\RR^2\),
\begin{equation}\label{eq:L-diagonal-decomposition}
  L(v)=\phi(v_1+v_2)+\phi(v_1-v_2).
\end{equation}
Consequently, \(L\) is invariant under sign changes and exchange of the two
coordinates.  For \(v\in D\) and \(0\le r\le1\),
\begin{equation}\label{eq:quadratic-scaling-L}
  L(rv)\le r^2L(v)\le rL(v).
\end{equation}
\end{revisioncheck}
\item For every \(v\in\RR^2\),
\begin{equation}\label{eq:L-energy-lower}
  L(v)\ge\norm{v}_2^2.
\end{equation}
There is \(c_0>0\) such that, for every \(v\in D\),
\begin{equation}\label{eq:L-global-quartic}
  L(v)\ge\norm{v}_2^2+c_0\norm{v}_2^4.
\end{equation}
Moreover, as \(v\to0\),
\begin{equation}\label{eq:L-local-expansion}
  L(v)=\norm{v}_2^2+Q_4(v)+O(\norm{v}_2^6).
\end{equation}
\begin{revisioncheck}
The remainder in \eqref{eq:L-local-expansion} is uniform in direction.
\end{revisioncheck}
\end{enumerate}
\end{lemma}

\begin{proof}
\begin{hyperreplacement}
\emph{Strict duality.}
Put \(p=(\alpha+\beta)/2\) and \(q=(\alpha-\beta)/2\).  The identity
\(\cosh\alpha+\cosh\beta=2\cosh p\cosh q\) gives
\begin{equation}\label{eq:psi-cosh-factor}
  \begin{aligned}
  \psi(\alpha,\beta)&=\log\cosh p+\log\cosh q,\\
  \alpha v_1+\beta v_2&=p(v_1+v_2)+q(v_1-v_2).
  \end{aligned}
\end{equation}
The first expression is real analytic.  Its Hessian in the coordinates
\((p,q)\) is
\(\operatorname{diag}(\operatorname{sech}^2p,
\operatorname{sech}^2q)\), which is positive definite.  Since the change
of coordinates is invertible and linear, \(\psi\) is strictly convex.
Moreover, \(v=\nabla\psi(\alpha,\beta)\) is equivalent to
\[
  v_1+v_2=\tanh p,\qquad v_1-v_2=\tanh q.
\]
The identity
\(\max\{|v_1+v_2|,|v_1-v_2|\}=|v_1|+|v_2|\) shows that these equations
have a unique solution precisely when \(v\in\operatorname{int}D\).  The
solution is
\[
  p=\operatorname{arctanh}(v_1+v_2),\qquad
  q=\operatorname{arctanh}(v_1-v_2),\qquad
  (\alpha,\beta)=(p+q,p-q).
\]
Thus \(\nabla\psi\) is a diffeomorphism from \(\RR^2\) onto
\(\operatorname{int}D\), with the displayed smooth inverse.
\end{hyperreplacement}

\emph{Entropy representation.}
The effective domain of the Legendre transform is the closed convex hull \(D\).  The entropy representation \eqref{eq:entropy-L} follows from the Gibbs variational formula for
\(
\log\sum_{s\in S}\frac14\e^{\ip{\theta}{s}}
\)
and finite dimensional convex duality.  Since
\[
  \sum_sp_s\log(4p_s)=\log4-\Bigl(-\sum_sp_s\log p_s\Bigr),
\]
it lies between \(0\) and \(\log4\), proving \eqref{eq:L-range}.  Continuity on the compact polytope \(D\) follows either from \eqref{eq:entropy-L} or general properties of finite support Cram\'er transforms.

\begin{revisioncheck}
\emph{Diagonal decomposition and scaling.}
\hyperrewrite{Taking the supremum separately in \(p\) and \(q\) in
\eqref{eq:psi-cosh-factor} gives \eqref{eq:L-diagonal-decomposition}.}  For \(\abs z<1\), the scalar
supremum in \eqref{eq:phi-def} is attained at
\(s=\operatorname{arctanh}z\), which gives the displayed formula for
\(\phi\).  At \(z=1\),
\[
  s-\log\cosh s
  =\log2-\log(1+\e^{-2s})\uparrow\log2
  \qquad(s\to\infty),
\]
and the case \(z=-1\) follows by replacing \(s\) with \(-s\).  The supremum
is infinite for \(\abs z>1\).  Since
\(\phi'(z)=\operatorname{arctanh}z\) for \(\abs z<1\), integrating the
power series for \(\operatorname{arctanh}\), followed by monotone
convergence at \(z=\pm1\), proves \eqref{eq:phi-series}.

The function \(\phi\) is even, so \eqref{eq:L-diagonal-decomposition}
also proves the stated symmetries.  For \(v\in D\), substitution of
\eqref{eq:phi-series} gives
\[
  L(rv)
  =\sum_{k\ge1}r^{2k}
    \frac{(v_1+v_2)^{2k}+(v_1-v_2)^{2k}}
         {2k(2k-1)}
  \le r^2L(v).
\]
The second inequality in \eqref{eq:quadratic-scaling-L} follows from
\(r^2\le r\) and \(L\ge0\).

\emph{Quadratic bounds and expansion.}
The terms with \(k=1\) and \(k=2\) are, respectively,
\[
  \frac{(v_1+v_2)^2+(v_1-v_2)^2}{2}=\norm v_2^2,
  \qquad
  \frac{(v_1+v_2)^4+(v_1-v_2)^4}{12}=Q_4(v).
\]
All remaining terms are nonnegative.  Hence
\[
  L(v)\ge\norm v_2^2+Q_4(v)
  \ge\norm v_2^2+\frac16\norm v_2^4,
  \qquad v\in D,
\]
which proves \eqref{eq:L-energy-lower} and
\eqref{eq:L-global-quartic}; outside \(D\),
\eqref{eq:L-energy-lower} follows from \(L=+\infty\).  As \(v\to0\),
\(\abs{v_1\pm v_2}\le\sqrt2\norm v_2\), and for \(\abs u\le1/2\),
\[
  \sum_{k\ge3}\frac{\abs{u}^{2k}}{2k(2k-1)}
  \le\frac{\abs{u}^6}{1-u^2}.
\]
Hence the terms with \(k\ge3\) contribute \(O(\norm v_2^6)\), uniformly
in direction.
\end{revisioncheck}
\end{proof}

\begin{lemma}
\label{lem:support-reconstruction}
Let \(E\subset\RR^d\) be nonempty and compact, and set
\(K:=\conv(E)\).  Then \(K\) is compact,
\begin{equation}\label{eq:support-convex-hull-general}
  h_K=h_E,
\end{equation}
and
\begin{equation}\label{eq:support-halfspace-general}
  K=\bigcap_{\lambda\in\RR^d}
  \left\{z\in\RR^d:\ip{\lambda}{z}\le h_E(\lambda)\right\}.
\end{equation}
Thus a support function determines the closed convex hull of \(E\), but
not a nonconvex set \(E\) itself.
\end{lemma}

\begin{proof}
By Carath\'eodory's theorem \cite{Rockafellar1970}, every point of \(K\) is a
convex combination of at most \(d+1\) points of \(E\).  Hence \(K\) is the continuous image
of a compact set and is compact.  The inclusion \(E\subset K\) gives
\(h_E\le h_K\).  Conversely, if
\(z=\sum_{j=1}^m\alpha_jz_j\in K\), where \(z_j\in E\),
\(\alpha_j\ge0\), and \(\sum_j\alpha_j=1\), then
\[
  \ip{\lambda}{z}
  =\sum_{j=1}^m\alpha_j\ip{\lambda}{z_j}
  \le h_E(\lambda).
\]
Taking the supremum over \(z\in K\) proves
\eqref{eq:support-convex-hull-general} and the left to right inclusion in
\eqref{eq:support-halfspace-general}.

For the reverse inclusion, let \(z\notin K\), and choose
\(p\in K\) minimizing \(\norm{z-p}_2\).  For \(y\in K\) and
\(t\in(0,1]\), convexity gives \(p+t(y-p)\in K\), so
\[
  0\le\norm{z-p-t(y-p)}_2^2-\norm{z-p}_2^2
  =-2t\ip{z-p}{y-p}+t^2\norm{y-p}_2^2.
\]
After division by \(t\) and passage to \(t\downarrow0\),
\(\ip{z-p}{y-p}\le0\).  Taking \(\lambda:=z-p\) yields
\[
  h_E(\lambda)=h_K(\lambda)=\ip{\lambda}{p}
  <\ip{\lambda}{z},
\]
because \(\ip{\lambda}{z-p}=\norm{z-p}_2^2>0\).  Thus \(z\) does not
belong to the intersection in \eqref{eq:support-halfspace-general}.
Finally, \eqref{eq:support-convex-hull-general} shows that \(E\) and
\(\conv(E)\) have the same support function.
\end{proof}

\subsection{Bounded variation and the area lift}

\begin{revisioncheck}
Let \((E,\norm{\cdot})\) be a normed finite dimensional space.  For a
continuous path \(f:[0,1]\to E\), define its total variation by
\[
  \Var(f):=\sup_{m\ge1}\ \sup_{0=t_0<t_1<\cdots<t_m=1}
  \sum_{i=1}^m\norm{f(t_i)-f(t_{i-1})}.
\]
Define the unit total variation path set by
\end{revisioncheck}
\begin{equation}\label{eq:V-support}
  \cV:=\left\{(x,y)\in C([0,1];\RR^2):
  (x(0),y(0))=0,\quad
  \Var(x)+\Var(y)\le1
  \right\}.
\end{equation}
The set \(\cV\) is closed in the uniform topology because total variation is lower semicontinuous under uniform convergence.

For \((x,y)\in\cV\), define its Riemann--Stieltjes area lift path by
\begin{revisioncheck}
\begin{equation}\label{eq:BV-area}
  \mathcal A_{\mathrm{BV}}(x,y)(t):=\int_0^t x\dd y.
\end{equation}
\end{revisioncheck}

\begin{lemma}\label{lem:area-continuity}
For \(\gamma=(x,y)\) and \(\widetilde\gamma=(\widetilde x,\widetilde y)\) in \(\cV\),
\begin{equation}\label{eq:area-lipschitz}
  \sup_{0\le t\le1}
  \abs{\int_0^t x\dd y-\int_0^t\widetilde x\dd\widetilde y}
  \le3\norm{\gamma-\widetilde\gamma}_\infty,
\end{equation}
where the norm on the right is the maximum of the two coordinate sup norms.  Hence the bounded variation area lift map
\begin{revisioncheck}
\begin{equation}\label{eq:Phi-lift}
  \Phi:\cV\to\cX,
  \qquad
  \Phi(x,y)=((x,y),\mathcal A_{\mathrm{BV}}(x,y)),
\end{equation}
\end{revisioncheck}
is continuous.
\end{lemma}

\begin{proof}
For every \(t\),
\begin{align*}
  \int_0^t x\dd y-\int_0^t\widetilde x\dd\widetilde y
  &=\int_0^t(x-\widetilde x)\dd y
    +\int_0^t\widetilde x\dd(y-\widetilde y).
\end{align*}
The first term is bounded by
\(
\norm{x-\widetilde x}_\infty\Var(y).
\)
Integration by parts in the second term gives
\begin{align*}
  \abs{\int_0^t\widetilde x\dd(y-\widetilde y)}
  &\le
  \norm{\widetilde x}_\infty\norm{y-\widetilde y}_\infty
  +\norm{y-\widetilde y}_\infty\Var(\widetilde x).
\end{align*}
The boundary term at zero vanishes because all paths in \(\cV\) start at zero; the endpoint term at \(t\) is the first term on the right.  All three variation or sup norm factors are at most one on \(\cV\), proving \eqref{eq:area-lipschitz}.
\end{proof}

The following exact identity is specific to the four horizontal generators and is the reason that no discretization correction appears in the enhanced LDP.

\begin{lemma}\label{lem:exact-area}
For the linearly interpolated single walk path \(Q_n=(\omega_n,\zeta_n)\),
\begin{equation}\label{eq:exact-area-identity}
  \zeta_n(t)=\int_0^t x_n\dd y_n,
  \qquad 0\le t\le1.
\end{equation}
Moreover, \(\omega_n\in\cV\) almost surely.
\end{lemma}

\begin{proof}
On step \(k+1\), write \(\Delta X_{k+1}=X_{k+1}-X_k\) and similarly for \(Y\).  Since each allowed increment changes exactly one horizontal coordinate,
\(
\Delta X_{k+1}\Delta Y_{k+1}=0.
\)
On the interval \([k/n,(k+r)/n]\), direct integration gives
\[
  \int_{k/n}^{(k+r)/n}x_n\dd y_n
  =\frac{rX_k\Delta Y_{k+1}}{n^2}
   +\frac{r^2\Delta X_{k+1}\Delta Y_{k+1}}{2n^2}
  =\frac{r(Z_{k+1}-Z_k)}{n^2}.
\]
Summing the completed intervals proves \eqref{eq:exact-area-identity}.  Each step contributes exactly \(1/n\) to \(\Var(x_n)+\Var(y_n)\), so the total is one.
\end{proof}

\subsection{Compactness and strict admissibility}

Define the deterministic ambient set of rooted coordinate paths with Lipschitz constant at most one by
\begin{equation}\label{eq:Kstar}
  \cK_*:=\left\{
  (x,y,\zeta)\in\cX:
  x(0)=y(0)=\zeta(0)=0,
  \quad x,y,\zeta\text{ are }1\text{-Lipschitz}
  \right\}.
\end{equation}

\begin{lemma}\label{lem:Kstar}
The set \(\cK_*\) is compact in \(\cX\).  Every discrete enhanced path \(Q_u^{(n)}\) belongs to \(\cK_*\).  Every \(q\in\cQ_\rho\) belongs to \(\cK_*\) as well.
\end{lemma}

\begin{proof}
Compactness follows from Arzel\`a--Ascoli.  The slopes of each horizontal coordinate of \(Q_u^{(n)}\) have absolute value at most one.  On the interval corresponding to step \(k+1\), the slope of the central interpolation is
\[
  \frac{X_{u_k}\xi_{u_{k+1}}^{(2)}}{n},
\]
whose absolute value is at most \(k/n\le1\).  Hence \(Q_u^{(n)}\in\cK_*\).

If \(q=q_\omega\in\cQ_\rho\), finiteness of the action implies \(\dot\omega(t)\in D\) almost everywhere, so each horizontal coordinate is \(1\) Lipschitz and \(\abs{x(t)}\le t\le1\).  Therefore
\(
\abs{\dot A_t}=\abs{x(t)\dot y(t)}\le1
\)
almost everywhere, proving \(q\in\cK_*\).
\end{proof}

Let \(\omega=(x,y)\in\ac([0,1];\RR^2)\) satisfy \(\omega(0)=0\).  We call \(\omega\) \emph{strictly admissible at parameter \(\rho\)} if there exist \(\eta,\kappa>0\) such that
\begin{equation}\label{eq:strict-admissibility}
  J_t(\omega)\le t(\log\rho-\eta)
  \quad\text{for all }t\in[0,1],
  \qquad
  \abs{\dot x(t)}+\abs{\dot y(t)}\le1-\kappa
  \quad\text{a.e.}
\end{equation}
Let \(\cC_\rho^{<}\) be the set of such paths.

\begin{lemma}\label{lem:compact-density}
The set \(\cQ_\rho\) is nonempty and compact.  The lifted paths generated by \(\cC_\rho^{<}\) are dense in \(\cQ_\rho\).  If \(\omega\in\cC_\rho\) and \(0<r<1\), then \(r\omega\in\cC_\rho^{<}\) and
\begin{equation}\label{eq:scaled-lift}
  q_{r\omega}=(r\omega,r^2A_\cdot(\omega))\longrightarrow q_\omega
  \quad\text{in }\cX\quad(r\uparrow1).
\end{equation}
\end{lemma}

\begin{proof}
By the good rate property in \Cref{thm:enhanced-ldp}, each \(I_t\) is lower
semicontinuous and \(I=I_1\) is good.  From the definitions,
\begin{equation}\label{eq:Q-intersection}
  \cQ_\rho
  =\bigcap_{0<t\le1}
  \left\{q\in\cX:I_t(\pi_tq)\le t\log\rho\right\}.
\end{equation}
\revisionmark{Indeed, the condition at \(t=1\) first forces \(q\) to be an
area lift, and the remaining inequalities impose the action bound at every
\(t\in(0,1)\).}  Thus \(\cQ_\rho\) is closed and lies in the compact level set \(\{I\le\log\rho\}\).  The zero path shows nonemptiness.

For \(0<r<1\), \eqref{eq:quadratic-scaling-L} gives
\[
  J_t(r\omega)\le r^2J_t(\omega)\le r^2t\log\rho
  =t\bigl(\log\rho-(1-r^2)\log\rho\bigr).
\]
Also \(\abs{r\dot x}+\abs{r\dot y}\le r\).  Hence \(r\omega\) is strict.  The area is quadratic under horizontal scaling, which yields \eqref{eq:scaled-lift}.
\end{proof}

\end{analystcheck}

\section{Path space convergence and finite generation corrections}
\label{sec:proof}

This section proves the probabilistic shape results and develops the finite
generation estimates for the central maximum.  The first four subsections
establish the enhanced Mogulskii principle, the path space shape theorem,
and its spatial shape and support function consequences.  The fifth subsection
states and proves a quantitative window below saturation.  The final
subsection proves the sharp corrections at and above saturation by a
separate retained subtree argument.

\subsection{The enhanced Mogulskii principle}\label{sec:ldp}

The horizontal walk satisfies the classical Mogulskii principle.  The
Heisenberg area map is not continuous on all of
\(C([0,1];\RR^2)\) under the uniform topology, but
\eqref{eq:area-lipschitz} makes it continuous on the closed set \(\cV\),
which contains the discrete interpolations and all finite rate paths.
Restricting the horizontal LDP to \(\cV\) and applying the contraction
principle yields the enhanced principle.

\begin{proof}[Proof of \Cref{thm:enhanced-ldp}]
Let \(\mathsf X_{\mathrm{bd}}\) denote the bounded paths with values in \(\RR^2\) on
\([0,1]\), with the uniform metric induced by \(\norm{\cdot}_\infty\).
This metric and the one induced by the Euclidean norm have the same topology.
Define the scaled partial sum path
\begin{equation}\label{eq:step-partial-sum-path}
  \overline\omega_n(s)
  :=\frac1n\sum_{j=1}^{\lfloor ns\rfloor}\xi_j
  =\frac1n\bigl(X_{\lfloor ns\rfloor},Y_{\lfloor ns\rfloor}\bigr),
  \qquad 0\le s\le1.
\end{equation}
Since \(\psi\) in \eqref{eq:psi-def} is finite on \(\RR^2\), Mogulskii's theorem
\cite{DemboZeitouni2010} gives a good LDP for
\((\overline\omega_n)\) on \(\mathsf X_{\mathrm{bd}}\), with speed \(n\).  Its
rate \(J\) equals \(J_1\) from \eqref{eq:action-area} on rooted absolutely
continuous paths and is \(+\infty\) otherwise.

For the linear interpolation \(\omega_n\) in \eqref{eq:Qn-grid-values},
\[
  \omega_n(s)-\overline\omega_n(s)
  =\left(s-\frac{k}{n}\right)\xi_{k+1},
  \qquad \frac{k}{n}\le s<\frac{k+1}{n}.
\]
Thus
\begin{equation}\label{eq:step-linear-distance}
  \sup_{0\le s\le1}
  \norm{\overline\omega_n(s)-\omega_n(s)}_\infty\le\frac1n.
\end{equation}
The bound is deterministic, so the two sequences are exponentially
equivalent at speed \(n\).  Therefore
\cite{DemboZeitouni2010} gives the same good LDP for
\((\omega_n)\) on \(\mathsf X_{\mathrm{bd}}\).

If \(J(\omega)<\infty\), then \(\dot\omega\in D\) almost everywhere by
\Cref{lem:convex-facts}; hence, writing \(\omega=(x,y)\),
\[
  \Var(x)+\Var(y)
  =\int_0^1\bigl(\abs{\dot x(s)}+\abs{\dot y(s)}\bigr)\dd s
  \le1.
\]
Thus the effective domain of \(J\) lies in \(\cV\), while
\(\omega_n\in\cV\) almost surely by \Cref{lem:exact-area}.  Since \(\cV\) is
closed, \cite{DemboZeitouni2010} restricts the LDP to \(\cV\),
with the same good rate.

By \Cref{lem:area-continuity}, the map \(\Phi:\cV\to\cX\) in
\eqref{eq:Phi-lift} is continuous.  The exact identity
\eqref{eq:exact-area-identity} gives \(Q_n=\Phi(\omega_n)\).  The contraction
principle \cite{DemboZeitouni2010} yields a good LDP for
\((Q_n)\), with speed \(n\) and rate
\[
  I^\Phi(q)
  :=\inf\{J(\omega):\omega\in\cV,\ \Phi(\omega)=q\}.
\]
The first component of \(\Phi(\omega)\) is \(\omega\), and for absolutely
continuous \(\omega=(x,y)\), its area component is
\(\int_0^\cdot x(s)\dot y(s)\dd s\).  Hence \(I^\Phi\) is exactly the rate
\(I=I_1\) in \eqref{eq:It-rate}.

Fix \(t\in(0,1]\).  The restriction map \(\pi_t\) in
\eqref{eq:restriction-map} and the contraction principle give a good LDP for
\((\pi_tQ_n)\), with speed \(n\) and rate
\[
  \widetilde I_t(q)=\inf\{I(\widetilde q):\pi_t\widetilde q=q\}.
\]
If \(I_t(q)<\infty\), write \(q=(\omega,\zeta)\), extend its horizontal
component by
\[
  \widehat\omega(s):=
  \begin{cases}
    \omega(s),&0\le s\le t,\\
    \omega(t),&t<s\le1,
  \end{cases}.
\]
This extension belongs to \(\cV\), because \(I_t(q)<\infty\) forces
\(\dot\omega\in D\) almost everywhere.  Then
\(\widehat q:=\Phi(\widehat\omega)\) satisfies
\(\pi_t\widehat q=q\), and \(L(0)=0\), from \eqref{eq:entropy-L}, gives
\(I(\widehat q)=I_t(q)\).  Conversely, every extension
\(\widetilde q\) of \(q\) has \(I(\widetilde q)\ge I_t(q)\), by
nonnegativity of \(L\) in \eqref{eq:L-range}.  If \(I_t(q)=+\infty\), no
finite rate extension exists.  Therefore \(\widetilde I_t=I_t\).
\end{proof}

\subsection{Tracking a strict path under bounded offspring}\label{sec:tilt}

We now prove the following uniform realization estimate needed for inner
coverage.

\begin{proposition}[Uniform positive probability under bounded offspring]\label{prop:bounded-success}
Consider a branching random walk with offspring number \(\widehat N\le K\)
almost surely and mean \(m=\EE\widehat N>1\).  Let
\(\omega=(x,y)\in\cC_m^{<}\), as defined through
\eqref{eq:strict-admissibility}, and choose \(\eta,\kappa>0\) such that
\begin{equation}\label{eq:m-strict}
  J_t(\omega)\le t(\log m-\eta)
  \quad(0\le t\le1),
  \qquad
  \abs{\dot x}+\abs{\dot y}\le1-\kappa\quad\text{a.e.}
\end{equation}
Then for every \(\varepsilon>0\) there exist \(p_\varepsilon>0\) and \(n_0\) such that
\begin{equation}\label{eq:positive-success}
  \PP\left(\exists \abs{u}=n:
  d_{\cX}(Q_u^{(n)},q_\omega)<\varepsilon\right)
  \ge p_\varepsilon,
  \qquad n\ge n_0.
\end{equation}
Here \(Q_u^{(n)}\) is the rescaled enhanced ancestral path whose components
are defined in \eqref{eq:horizontal-interp} and
\eqref{eq:central-interp}, \(q_\omega\) is the enhanced area lift in
\eqref{eq:lift-def}, and \(d_{\cX}\) is the path metric in
\eqref{eq:path-metric}.
\end{proposition}

The remainder of this subsection proves \Cref{prop:bounded-success} by
combining a time inhomogeneous Cram\'er tilt, a uniform \(L^p\) bound for the
associated additive branching martingale, and the weighted many to one
identity.

\subsubsection{Cram\'er tilting and a uniform \texorpdfstring{\(L^p\)}{Lp} bound}

We use the classical exponential tilt and additive branching martingale from
\cite{Shi2015}; compare the fixed parameter criterion in
\cite{Biggins1977Martingale}.  Because the tilt varies with \(i\) and
\(n\), \Cref{lem:Wp} derives the required uniform \(L^p\) bound from
\eqref{eq:m-strict}.

For \(k\ge0\), let \(\cF_k\) be the sigma field generated by the offspring
numbers of vertices in generations strictly below \(k\) and the increments
on edges ending in generations at most \(k\).  Thus neither the offspring
numbers nor the increments on edges leaving generation \(k\) have been
revealed.  For a truncated process, the same symbol denotes the filtration
generated by the truncated offspring numbers and the same revealed
increments.

For \(1\le i\le n\), define the averaged target velocity
\begin{equation}\label{eq:v-in}
  v_{i,n}:=n\left(\omega(i/n)-\omega((i-1)/n)\right).
\end{equation}
Jensen's inequality and \eqref{eq:m-strict} give, for every \(k\le n\),
\begin{equation}\label{eq:discrete-action-bound}
  \sum_{i=1}^kL(v_{i,n})
  \le nJ_{k/n}(\omega)
  \le k(\log m-\eta).
\end{equation}
Moreover, \(\norm{v_{i,n}}_1\le1-\kappa\), so \(v_{i,n}\) remains in a compact subset of \(\operatorname{int}D\).  By \Cref{lem:convex-facts}, there is a unique exponential tilt parameter \(\theta_{i,n}\) satisfying
\begin{equation}\label{eq:theta-in}
  \nabla\psi(\theta_{i,n})=v_{i,n},
\end{equation}
and all \(\theta_{i,n}\) lie in one compact subset of \(\RR^2\).

Recall that \(\mu\) is the original increment law on \(S\), as defined in
\Cref{sec:heisenberg-increments}.  Define the one step exponential factor \(R_{i,n}\) and
the edge weight \(a_{i,n}\), normalized by the offspring mean, by
\begin{equation}\label{eq:weights}
  R_{i,n}(\xi):=\exp\{\ip{\theta_{i,n}}{\xi}-\psi(\theta_{i,n})\},
  \qquad
  a_{i,n}(\xi):=m^{-1}R_{i,n}(\xi).
\end{equation}
Since \(\int_S R_{i,n}\,\mathrm d\mu=1\), the function \(R_{i,n}\) is the
Radon--Nikodym derivative of an increment law \(\nu_{i,n}\) with respect to
\(\mu\):
\begin{equation}\label{eq:tilted-law}
  \frac{\mathrm d\nu_{i,n}}{\mathrm d\mu}=R_{i,n}.
\end{equation}
These derivatives define the following additive branching martingale.
For a particle \(u\) in generation \(k\), let \(\xi_{u,i}\) be the increment on the \(i\)th ancestral edge.  Define its lineage weight \(H_{u,k}^{(n)}\) and the generation sum \(W_{k,n}\) by
\begin{equation}\label{eq:H-W}
  H_{u,k}^{(n)}:=\prod_{i=1}^ka_{i,n}(\xi_{u,i}),
  \qquad
  W_{k,n}:=\sum_{\abs{u}=k}H_{u,k}^{(n)}.
\end{equation}
For each fixed \(n\), \((W_{k,n})_{0\le k\le n}\) is a nonnegative mean one
time inhomogeneous additive branching martingale with respect to
\((\cF_k)_{0\le k\le n}\).  To verify this, fix \(1\le k\le n\).  For
\(\abs{v}=k-1\), let \(\widehat N_v\) be the number of children of \(v\) and
\(\xi_{v,j}\) the increment from \(v\) to its \(j\)th child.  Given
\(\cF_{k-1}\), one has \(\EE\widehat N_v=m\), and these increments are
independent with law \(\mu\) and independent of \(\widehat N_v\).  Moreover,
\(H_{v,k-1}^{(n)}\) is \(\cF_{k-1}\) measurable, and the \(j\)th child's
weight is
\(H_{v,k-1}^{(n)}a_{k,n}(\xi_{v,j})\),
\[
\begin{aligned}
  \EE[W_{k,n}\mid\cF_{k-1}]
  &=\sum_{\abs{v}=k-1}H_{v,k-1}^{(n)}
    \EE\left[
      \left.\sum_{j=1}^{\widehat N_v}a_{k,n}(\xi_{v,j})
      \,\right|\,\cF_{k-1}\right]\\
  &=\sum_{\abs{v}=k-1}H_{v,k-1}^{(n)}
    m\int_S a_{k,n}(\xi)\,\mathrm d\mu(\xi)
    =W_{k-1,n},
\end{aligned}
\]
where the last equality uses \(a_{k,n}=m^{-1}R_{k,n}\) and
\(\int_S R_{k,n}\,\mathrm d\mu=1\).  Since \(W_{0,n}=1\), the martingale has
mean one.

We now prove the uniform moment bound supplied by the strict action gap.

\begin{lemma}\label{lem:Wp}
Under \eqref{eq:m-strict}, there exists \(p\in(1,2)\) such that
\begin{equation}\label{eq:Wp-uniform}
  \sup_{n\ge1}\EE[W_{n,n}^p]<\infty.
\end{equation}
\end{lemma}

\begin{proof}
For \(p>1\), define the one generation \(p\) moment factor
\begin{align}
  m_{i,n}(p)
  &:=m\int_S a_{i,n}(\xi)^p\,\mathrm d\mu(\xi)\notag\\
  &=\exp\left\{-(p-1)\log m
    +\psi(p\theta_{i,n})-p\psi(\theta_{i,n})\right\}.
  \label{eq:moment-factor}
\end{align}
Because all \(\theta_{i,n}\) range over a fixed compact set, Taylor expansion in \(p\), uniformly for \(p\) near one, yields, with a finite constant \(C\) independent of \(i,n,p\),
\begin{equation}\label{eq:p-taylor}
  \psi(p\theta)-p\psi(\theta)
  \le(p-1)\bigl(\ip{\theta}{\nabla\psi(\theta)}-\psi(\theta)\bigr)
  +C(p-1)^2.
\end{equation}
At \(\theta=\theta_{i,n}\), the bracket is \(L(v_{i,n})\).  Summing \eqref{eq:p-taylor} and using \eqref{eq:discrete-action-bound}, we obtain, for every \(k\le n\),
\begin{equation}\label{eq:moment-product}
  \prod_{i=1}^km_{i,n}(p)
  \le\exp\{-k[(p-1)\eta-C(p-1)^2]\}.
\end{equation}
Choose \(p\in(1,2)\) close enough to one that \(C(p-1)\le\eta/2\).  Then, with \(\gamma=(p-1)\eta/2\),
\begin{equation}\label{eq:moment-product-decay}
  \prod_{i=1}^km_{i,n}(p)\le\e^{-\gamma k}
  \qquad(k\le n).
\end{equation}

Let \(D_{k,n}:=W_{k,n}-W_{k-1,n}\) be the \(k\)th martingale difference.
Writing
\(\Xi_{v,k,n}:=\sum_{j=1}^{\widehat N_v}a_{k,n}(\xi_{v,j})\) for the
normalized total weight of the children of \(v\), the preceding
decomposition gives
\begin{equation}\label{eq:mart-diff}
  D_{k,n}=\sum_{\abs{v}=k-1}H_{v,k-1}^{(n)}(\Xi_{v,k,n}-1).
\end{equation}
Conditionally on \(\cF_{k-1}\), the centered factors in parentheses are
independent.  Bounded offspring and compactness of the tilts imply
\begin{equation}\label{eq:Xi-uniform}
  \sup_{v,k,n}\EE\abs{\Xi_{v,k,n}-1}^p<\infty.
\end{equation}
The conditional von Bahr--Esseen inequality
\cite{vonBahrEsseen1965} therefore gives, for a finite constant
\(C_p\) depending only on \(p\),
\begin{align}
  \EE\left[\abs{D_{k,n}}^p\mid\cF_{k-1}\right]
  &\le C_p\sum_{\abs{v}=k-1}\left(H_{v,k-1}^{(n)}\right)^p.
  \label{eq:conditional-vbe}
\end{align}
Taking expectations and using the branching property,
\begin{equation}\label{eq:D-moment}
  \EE\abs{D_{k,n}}^p
  \le C_p\prod_{i=1}^{k-1}m_{i,n}(p)
  \le C_p\e^{-\gamma(k-1)}.
\end{equation}
A second application of the same theorem to the martingale difference
sequence \((D_{k,n})_{1\le k\le n}\) gives
\begin{equation}\label{eq:martingale-p-ineq}
  \EE\abs{\sum_{k=1}^nD_{k,n}}^p
  \le C_p\sum_{k=1}^n\EE\abs{D_{k,n}}^p.
\end{equation}
Consequently,
\[
  \sup_n\EE\abs{W_{n,n}-1}^p
  \le C_p\sum_{k\ge0}\e^{-\gamma k}<\infty,
\]
and \(\abs{W_{n,n}}^p\le2^{p-1}(1+\abs{W_{n,n}-1}^p)\) proves \eqref{eq:Wp-uniform}.
\end{proof}

\subsubsection{Tilted concentration and return to the branching law}

Let \(\QQ_n\) be the product law under which the increments are independent
with respective laws \(\nu_{1,n},\ldots,\nu_{n,n}\).  Relative to the
original product law \(\mu^{\otimes n}\), its Radon--Nikodym derivative is
\begin{equation}\label{eq:path-density}
  \frac{\mathrm d\QQ_n}{\mathrm d\mu^{\otimes n}}
  (\xi_1,\ldots,\xi_n)
  =\prod_{i=1}^nR_{i,n}(\xi_i).
\end{equation}
Moreover,
\begin{equation}\label{eq:tilted-mean}
  \EE_{\QQ_n}\xi_i=\nabla\psi(\theta_{i,n})=v_{i,n}.
\end{equation}

\begin{lemma}\label{lem:tilted-concentration}
Under \(\QQ_n\), the enhanced interpolation \(Q_n\) converges in probability in \(\cX\) to \(q_\omega\).
\end{lemma}

\begin{proof}
Let \(\omega_n\) denote the horizontal component of \(Q_n\), and set
\[
  M_k:=\sum_{i=1}^k(\xi_i-v_{i,n}),
  \qquad 0\le k\le n.
\]
By \eqref{eq:tilted-mean}, the two coordinates of \(M_k\) are sums of
independent centered random variables.  Each summand in either coordinate
lies in an interval of length \(2\).  Thus Hoeffding's inequality
\cite{Hoeffding1963} gives, for \(\ell\in\{1,2\}\),
\[
  \QQ_n\left(\abs{M_k^{(\ell)}}>n\delta\right)
  \le 2\exp\left\{-\frac{n^2\delta^2}{2k}\right\},
  \qquad 1\le k\le n.
\]
Taking the union over the two coordinates and \(1\le k\le n\) yields
\begin{equation}\label{eq:tilted-maximal-bound}
  \QQ_n\left(\max_{0\le k\le n}\norm{M_k}_\infty>n\delta\right)
  \le 4\sum_{k=1}^n
       \exp\left\{-\frac{n^2\delta^2}{2k}\right\}
  \le 4n\e^{-n\delta^2/2}.
\end{equation}

Let \(\omega^{[n]}\) be the piecewise linear interpolation of the values
\(\omega(k/n)\).  Thus, if \(t=(k+r)/n\), where \(0\le k<n\) and
\(0\le r\le1\), then
\(\omega^{[n]}(t)=(1-r)\omega(k/n)+r\omega((k+1)/n)\).
For the same \(t\), \eqref{eq:v-in} gives
\[
  \sum_{i=1}^kv_{i,n}=n\omega(k/n)
\]
and hence
\begin{equation}\label{eq:horizontal-tilted-error}
  \omega_n(t)-\omega^{[n]}(t)
  =\frac{(1-r)M_k+rM_{k+1}}{n}.
\end{equation}
The second inequality in \eqref{eq:m-strict} implies that \(\omega\) is
\(1\) Lipschitz in the maximum norm.  Therefore
\[
\begin{aligned}
  \norm{\omega(t)-\omega^{[n]}(t)}_\infty
  &\le (1-r)\norm{\omega(t)-\omega(k/n)}_\infty
     +r\norm{\omega(t)-\omega((k+1)/n)}_\infty\\
  &\le \frac{2r(1-r)}{n}\le\frac1n.
\end{aligned}
\]
Combining this estimate with \eqref{eq:horizontal-tilted-error} gives
\begin{equation}\label{eq:horizontal-uniform-error}
  \norm{\omega_n-\omega}_\infty
  \le \frac1n\max_{0\le k\le n}\norm{M_k}_\infty+\frac1n.
\end{equation}
Consequently, by \eqref{eq:tilted-maximal-bound},
\begin{equation}\label{eq:horizontal-tilted-concentration}
  \QQ_n\left(
    \norm{\omega_n-\omega}_\infty>\delta+\frac1n
  \right)
  \le4n\e^{-n\delta^2/2}.
\end{equation}

By \Cref{lem:exact-area}, \(Q_n=\Phi(\omega_n)\) and
\(\omega_n\in\cV\).  Moreover, \eqref{eq:m-strict} gives
\(\Var(x)+\Var(y)
=\int_0^1(\abs{\dot x(s)}+\abs{\dot y(s)})\dd s\le1-\kappa\), so
\(\omega\in\cV\) and \(q_\omega=\Phi(\omega)\).  Hence
\eqref{eq:area-lipschitz} and \eqref{eq:path-metric} give
\[
  d_{\cX}(Q_n,q_\omega)
  \le3\norm{\omega_n-\omega}_\infty.
\]
For any \(\varepsilon>0\), take \(\delta=\varepsilon/6\).  If
\(n\ge6/\varepsilon\), then \eqref{eq:horizontal-tilted-concentration}
implies
\[
  \QQ_n\left(d_{\cX}(Q_n,q_\omega)>\varepsilon\right)
  \le4n\exp\left\{-\frac{n\varepsilon^2}{72}\right\}
  \longrightarrow0.
\]
\end{proof}

The weighted many to one identity now transfers the preceding concentration
statement to a positive probability under the original branching law.  It is
the product tilt version of the time inhomogeneous many to one formula; compare
\cite{Mallein2015Inhomogeneous}.

\begin{proof}[Proof of \Cref{prop:bounded-success}]
Let \(G^\varepsilon\) be the open \(\varepsilon\) ball about \(q_\omega\), namely
\[
  G^\varepsilon:=\{q\in\cX:d_{\cX}(q,q_\omega)<\varepsilon\}.
\]
For an event \(E\), let \(\one_E\) denote its indicator.  Define the restricted weighted mass by
\begin{equation}\label{eq:Y-def}
  Y_n^\varepsilon
  :=\sum_{\abs{u}=n}H_{u,n}^{(n)}
  \one_{\{Q_u^{(n)}\in G^\varepsilon\}}.
\end{equation}
The weighted many to one identity gives
\begin{equation}\label{eq:weighted-m21}
  \EE Y_n^\varepsilon=\QQ_n(Q_n\in G^\varepsilon).
\end{equation}
Indeed, the expected number \(m^n\) cancels the product of the factors
\(m^{-1}\), while \eqref{eq:path-density} identifies the remaining product
of the \(R_{i,n}\)'s as the Radon--Nikodym derivative of \(\QQ_n\).  By
\Cref{lem:tilted-concentration}, the right hand side tends to one.

Since \(0\le Y_n^\varepsilon\le W_{n,n}\), \Cref{lem:Wp} gives a uniform \(L^p\) bound for \(Y_n^\varepsilon\).  H\"older's inequality yields
\[
  \EE Y_n^\varepsilon
  \le
  \bigl(\EE[(Y_n^\varepsilon)^p]\bigr)^{1/p}
  \PP(Y_n^\varepsilon>0)^{(p-1)/p}.
\]
Thus \(\PP(Y_n^\varepsilon>0)\) is bounded below by a positive constant for
all large \(n\).  All weights are strictly positive, so
\(Y_n^\varepsilon>0\) is equivalent to
\[
  \exists\,\abs{u}=n:\qquad
  d_{\cX}(Q_u^{(n)},q_\omega)<\varepsilon.
\]
\end{proof}

\subsection{Offspring truncation and amplification on survival}
\label{sec:amplification}

We now remove the bounded offspring restriction and strengthen
\Cref{prop:bounded-success} to almost sure eventual realization on the
survival event.

\begin{proposition}[Eventual realization of every strict path]\label{prop:eventual-strict}
Assume \eqref{eq:offspring-assumption}.  If
\(\omega\in\cC_\rho^{<}\), as defined through
\eqref{eq:strict-admissibility}, then for every \(\varepsilon>0\), on
\(\cS\), almost surely there exists \(n_0(\omega,\varepsilon)\) such that
for every \(n\ge n_0\) one can find \(\abs{u}=n\) with
\begin{equation}\label{eq:eventual-tube}
  d_{\cX}(Q_u^{(n)},q_\omega)<\varepsilon.
\end{equation}
\end{proposition}

The proof truncates the offspring law and uses the Kesten--Stigum theorem to apply
\Cref{prop:bounded-success} in one of many independent descendant trees,
following the amplification scheme in
\cite{HarrisRoberts2012}.  The final
splice uses the noncommutative Heisenberg group law.  The condition
\(\EE[N\log_+N]<\infty\) enters through the Kesten--Stigum limit
\eqref{eq:KS}; the large deviation, tilting, splicing and outer containment
arguments require only \(\rho<\infty\).

\begin{proof}
\emph{Truncation and a uniform descendant estimate.}
Choose \(\eta,\kappa>0\) as in \eqref{eq:strict-admissibility}, and put
\[
  N^{(K)}:=N\wedge K,
  \qquad
  \rho_K:=\EE N^{(K)}.
\]
Since \(\rho_K\uparrow\rho\), choose \(K\) so that \(\rho_K>1\) and
\begin{equation}\label{eq:rhoK}
  \log\rho_K\ge\log\rho-\eta/2.
\end{equation}
For every \(t\in[0,1]\),
\begin{align*}
  J_t(\omega)
  &\le t(\log\rho-\eta)
   \le t(\log\rho_K-\eta/2),\\
  \abs{\dot x}+\abs{\dot y}&\le1-\kappa\quad\text{a.e.}
\end{align*}
At a particle \(v\), retain
\(v1,\ldots,v(N_v\wedge K)\) and their original edge increments.  This
canonical pruning realizes the bounded offspring process as a subtree of the
original process.  Set \(\delta:=\varepsilon/4\).  Applied to the retained
process based at the identity, \Cref{prop:bounded-success} gives
\(p_0>0\) and \(m_0\) such that
\begin{equation}\label{eq:truncated-tube-probability}
  \PP\left(\exists\,\abs{w}=m:
    d_{\cX}(Q_w^{(m)},q_\omega)<\delta\right)\ge p_0,
  \qquad m\ge m_0.
\end{equation}

\emph{Kesten--Stigum growth.}
The Kesten--Stigum theorem
\cite{LyonsPemantlePeres1995} and
\eqref{eq:offspring-assumption} give a random variable
\(W_{\mathrm{GW}}\) such that
\begin{equation}\label{eq:KS}
  \rho^{-j}\mathcal N_j\longrightarrow W_{\mathrm{GW}}
  \quad\text{a.s.},
  \qquad
  \PP(W_{\mathrm{GW}}=0)=\PP(\cS^c).
\end{equation}
In particular, \(W_{\mathrm{GW}}>0\) almost surely on \(\cS\).
Choose \(C>3/\log\rho\), and define the initial amplification depth \(r_n\) and the remaining descendant depth \(m_n\) by
\begin{equation}\label{eq:rn}
  r_n:=\lceil C\log n\rceil,
  \qquad
  m_n:=n-r_n.
\end{equation}
Since \(C\log\rho>3\), \eqref{eq:KS} gives, almost surely on \(\cS\),
\begin{equation}\label{eq:many-roots}
  \frac{\mathcal N_{r_n}}{n^3}
  =\bigl(\rho^{-r_n}\mathcal N_{r_n}\bigr)
    \frac{\rho^{r_n}}{n^3}
  \longrightarrow\infty,
  \qquad
  \mathcal N_{r_n}\ge n^3
  \quad\text{eventually}.
\end{equation}
Since \(m_n\to\infty\), the probability bound \(p_0\) applies at depth
\(m_n\) for all sufficiently large \(n\).  On
\(\{\mathcal N_{r_n}\ge n^3\}\), choose the first \(n^3\) particles in
Ulam--Harris order in generation \(r_n\); this choice is measurable with respect to
\(\cF_{r_n}\).  Apply the canonical truncation above separately in the
descendant tree of each chosen particle.  By the branching property,
conditional on \(\cF_{r_n}\), the retained subtrees are independent copies
of the bounded offspring process.  By
\eqref{eq:truncated-tube-probability}, the conditional probability that none
of these \(n^3\) retained descendant processes satisfies the event in
\eqref{eq:truncated-tube-probability} with \(m=m_n\) is at most
\begin{equation}\label{eq:independent-subtree-bound}
  (1-p_0)^{n^3}\le\e^{-p_0n^3}.
\end{equation}
After intersecting with \(\{\mathcal N_{r_n}\ge n^3\}\) and taking
expectations,
\[
  \sum_{n=1}^\infty
  \PP\left(
    \substack{\mathcal N_{r_n}\ge n^3\ \text{and none of the selected trees}\\
    \text{satisfies \eqref{eq:truncated-tube-probability} with }m=m_n}
  \right)
  \le\sum_{n=1}^\infty\e^{-p_0n^3}<\infty.
\]
The displayed summability, \eqref{eq:many-roots}, and the first
Borel--Cantelli lemma imply that, almost surely on \(\cS\), for every
sufficiently large \(n\) there are a particle \(v_n\) in generation
\(r_n\) and a descendant \(u^{(n)}\) in generation \(n\), retained by the
truncation, such that \(v_n=u^{(n)}_{r_n}\) and
\eqref{eq:relative-tube} holds.

\emph{Heisenberg splicing and the noncommutative cross term.}
For \(0\le j\le m_n\), define
\[
  V(v_n)^{-1}V\bigl(u^{(n)}_{r_n+j}\bigr)
  =\bigl(\widetilde X_{j,n},\widetilde Y_{j,n},
    \widetilde Z_{j,n}\bigr).
\]
Let \((\widetilde\omega_n,\widetilde\zeta_n)\in\cX\) be linear on the
intervals with endpoints \(j/m_n\), with grid values
\[
  \widetilde\omega_n(j/m_n)
  :=\frac1{m_n}(\widetilde X_{j,n},\widetilde Y_{j,n}),
  \qquad
  \widetilde\zeta_n(j/m_n):=\frac1{m_n^2}\widetilde Z_{j,n}.
\]
By this choice,
\begin{equation}\label{eq:relative-tube}
  d_{\cX}\bigl((\widetilde\omega_n,\widetilde\zeta_n),
    q_\omega\bigr)<\delta.
\end{equation}
Write \(V(v_n)=(X_{v_n},Y_{v_n},Z_{v_n})\), let \(\widetilde y_n\) be the
second coordinate of \(\widetilde\omega_n\), and put
\begin{equation}\label{eq:time-change}
  a_n:=\frac{r_n}{n},
  \qquad
  b_n:=\frac{m_n}{n}=1-a_n,
  \qquad
  s_n(t):=\frac{nt-r_n}{m_n}.
\end{equation}
For \(a_n\le t\le1\),
\begin{equation}\label{eq:time-change-error}
  s_n(t)-t=-\frac{r_n(1-t)}{m_n},
  \qquad
  \sup_{a_n\le t\le1}\abs{s_n(t)-t}
  =a_n.
\end{equation}

At the interpolation points, the Heisenberg group law gives
\begin{align*}
  &X_{u^{(n)}_{r_n+j}}=X_{v_n}+\widetilde X_{j,n},\quad
   Y_{u^{(n)}_{r_n+j}}=Y_{v_n}+\widetilde Y_{j,n},\\
  &Z_{u^{(n)}_{r_n+j}}=Z_{v_n}+\widetilde Z_{j,n}+X_{v_n}\widetilde Y_{j,n}.
\end{align*}
The cross term \(X_{v_n}\widetilde Y_{j,n}\) is forced by the \(xb\) factor in
the group law and has no Euclidean analogue.  After Carnot scaling, these
identities and linear interpolation yield
\begin{align}
  \omega_{u^{(n)}}^{(n)}(t)
  &=\frac{(X_{v_n},Y_{v_n})}{n}
    +b_n\widetilde\omega_n(s_n(t)),
    \label{eq:splice-horizontal}\\
  \zeta_{u^{(n)}}^{(n)}(t)
  &=\frac{Z_{v_n}}{n^2}
    +b_n^2\widetilde\zeta_n(s_n(t))
    +\frac{X_{v_n}m_n}{n^2}\widetilde y_n(s_n(t)).
    \label{eq:splice-central}
\end{align}
Deterministically,
\begin{equation}\label{eq:root-bounds}
  \abs{X_{v_n}}+\abs{Y_{v_n}}\le r_n,
  \qquad
  \abs{Z_{v_n}}\le\sum_{k=1}^{r_n}(k-1)\le r_n^2,
  \qquad
  \norm{\widetilde y_n}_\infty\le1.
\end{equation}
By \eqref{eq:relative-tube}, \Cref{lem:Kstar}, and
\eqref{eq:time-change-error},
\begin{gather*}
  \norm{\widetilde\omega_n-\omega}_\infty<\delta,
  \qquad
  \norm{\widetilde\zeta_n-A_\cdot(\omega)}_\infty<\delta,\\
  \norm{\omega(s)-\omega(t)}_\infty\le\abs{s-t},
  \qquad
  \abs{A_s(\omega)-A_t(\omega)}\le\abs{s-t},\\
  \norm{\omega}_\infty\le1,
  \qquad
  \norm{A_\cdot(\omega)}_\infty\le1.
\end{gather*}
For \(a_n\le t\le1\), \eqref{eq:splice-horizontal} gives
\begin{align*}
  \norm{\omega_{u^{(n)}}^{(n)}(t)-\omega(t)}_\infty
  &\le \frac{\norm{(X_{v_n},Y_{v_n})}_\infty}{n}
    +b_n\norm{\widetilde\omega_n(s_n(t))
      -\omega(s_n(t))}_\infty\\
  &\quad +(1-b_n)\norm{\omega(s_n(t))}_\infty
    +\norm{\omega(s_n(t))-\omega(t)}_\infty\\
  &\le a_n+b_n\delta+a_n+a_n
   \le\delta+3a_n.
\end{align*}
The central coordinate in \eqref{eq:splice-central} satisfies
\begin{align*}
  \abs{\zeta_{u^{(n)}}^{(n)}(t)-A_t(\omega)}
  &\le \frac{\abs{Z_{v_n}}}{n^2}
    +b_n^2\abs{\widetilde\zeta_n(s_n(t))-A_{s_n(t)}(\omega)}\\
  &\quad +(1-b_n^2)\abs{A_{s_n(t)}(\omega)}
    +\abs{A_{s_n(t)}(\omega)-A_t(\omega)}\\
  &\quad
    +\frac{\abs{X_{v_n}}m_n}{n^2}\abs{\widetilde y_n(s_n(t))}\\
  &\le a_n^2+b_n^2\delta+(1-b_n^2)+a_n+a_nb_n\\
  &\le\delta+4a_n.
\end{align*}
Consequently,
\begin{equation}\label{eq:splicing-error}
\begin{aligned}
  &
  \sup_{a_n\le t\le1}\max\bigl\{
    \norm{\omega_{u^{(n)}}^{(n)}(t)-\omega(t)}_\infty,
    \abs{\zeta_{u^{(n)}}^{(n)}(t)-A_t(\omega)}\bigr\}\\
  &\hspace{8em}\le\delta+4a_n.
\end{aligned}
\end{equation}
On \([0,a_n]\), both \(Q_{u^{(n)}}^{(n)}\) and \(q_\omega\) start at zero
and belong to \(\cK_*\), so
\[
  \sup_{0\le t\le a_n}\max\bigl\{
    \norm{\omega_{u^{(n)}}^{(n)}(t)-\omega(t)}_\infty,
    \abs{\zeta_{u^{(n)}}^{(n)}(t)-A_t(\omega)}\bigr\}
  \le2a_n.
\]
Since \(a_n\to0\), \eqref{eq:splicing-error} yields
\[
  \limsup_{n\to\infty}
  d_{\cX}\bigl(Q_{u^{(n)}}^{(n)},q_\omega\bigr)
  \le\delta=\varepsilon/4<\varepsilon.
\]
Thus \eqref{eq:eventual-tube} holds for all sufficiently large \(n\).
\end{proof}

\subsection{Proof of the path space shape theorem and its corollaries}
\label{sec:shape-proof}

Hausdorff convergence in Theorem~\ref{thm:path-shape} has two directed
components:
\begin{equation}\label{eq:directed-Hausdorff-components}
  \begin{gathered}
  \sup_{q\in\cQ_n}d_{\cX}(q,\cQ_\rho)\longrightarrow0
  \qquad\text{(outer containment)},\\
  \sup_{q\in\cQ_\rho}d_{\cX}(q,\cQ_n)\longrightarrow0
  \qquad\text{(inner coverage)}.
  \end{gathered}
\end{equation}
The finite cover argument
below proves outer containment using the many to one identity, the
enhanced path large deviation principle and Borel--Cantelli.  Inner coverage
then follows from \Cref{prop:eventual-strict,lem:compact-density} and a finite
net of \(\cQ_\rho\).  The spatial shape and support function statements in
Corollaries~\ref{cor:endpoint-shape} and~\ref{cor:front-support} are
deterministic consequences of the path space limit.

\begin{lemma}\label{lem:finite-violation-cover}
Let \(K\subset\cX\) be compact and disjoint from \(\cQ_\rho\).  There exist
finitely many times \(t_1,\ldots,t_m\in(0,1]\), compact sets
\(F_i\subset\cX_{t_i}\), and a number \(\delta>0\) such that
\begin{equation}\label{eq:F-cost-gap}
  \inf_{q\in F_i}I_{t_i}(q)\ge t_i\log\rho+2\delta,
  \qquad 1\le i\le m,
\end{equation}
and every \(q\in K\) satisfies \(\pi_{t_i}q\in F_i\) for at least one \(i\).
\end{lemma}

\begin{proof}
For each \(q\in K\), \eqref{eq:Q-intersection} gives a time \(t(q)\in(0,1]\) with
\[
  I_{t(q)}(\pi_{t(q)}q)>t(q)\log\rho.
\]
Choose \(\delta(q)>0\) such that
\[
  I_{t(q)}(\pi_{t(q)}q)>t(q)\log\rho+4\delta(q).
\]
Lower semicontinuity of \(I_{t(q)}\) implies that
\[
  O_q:=\{r\in\cX_{t(q)}:
  I_{t(q)}(r)>t(q)\log\rho+3\delta(q)\}
\]
is an open neighborhood of \(\pi_{t(q)}q\).  Thus
\(
U_q:=\pi_{t(q)}^{-1}(O_q)
\)
is an open neighborhood of \(q\).

Because \(K\) is compact metric, for every \(q\) choose an open \(V_q\) with
\(
q\in V_q\subset\overline{V_q}\subset U_q.
\)
A finite subcollection \(V_{q_1},\ldots,V_{q_m}\) covers \(K\).  Set
\[
  C_i:=K\cap\overline{V_{q_i}},
  \qquad
  t_i:=t(q_i),
  \qquad
  F_i:=\pi_{t_i}(C_i).
\]
Each \(C_i\) and \(F_i\) is compact, the \(C_i\)'s cover \(K\), and
\(
F_i\subset O_{q_i}.
\)
Since a lower semicontinuous function attains its minimum on a compact set,
\[
  \inf_{F_i}I_{t_i}>t_i\log\rho+3\delta(q_i).
\]
Taking \(\delta:=\min_{1\le i\le m}\delta(q_i)>0\) gives \eqref{eq:F-cost-gap}.
\end{proof}

The next lemma combines this deterministic reduction with the many to one
identity and the enhanced path large deviation principle.

\begin{lemma}\label{lem:outer-summable}
For every \(\varepsilon>0\),
\begin{equation}\label{eq:outer-sum}
  \sum_{n=1}^\infty
  \PP\left(
  \exists \abs{u}=n:
  d_{\cX}(Q_u^{(n)},\cQ_\rho)\ge\varepsilon
  \right)<\infty.
\end{equation}
\end{lemma}

\begin{proof}
By \Cref{lem:Kstar}, all discrete paths and all paths in \(\cQ_\rho\) lie in
the deterministic compact set \(\cK_*\).  Define its subset separated from
\(\cQ_\rho\) by at least \(\varepsilon\) as
\[
  K_\varepsilon
  :=\{q\in\cK_*:d_{\cX}(q,\cQ_\rho)\ge\varepsilon\}.
\]
It is compact and disjoint from \(\cQ_\rho\).  Apply
\Cref{lem:finite-violation-cover}, obtaining \(t_i,F_i,\delta\).

For each \(t_i\), define the corresponding discrete generation by
\begin{equation}\label{eq:mi}
  m_i(n):=\lceil nt_i\rceil.
\end{equation}
If a generation \(n\) path belongs to \(K_\varepsilon\), then for some \(i\)
its restriction to \([0,t_i]\) lies in \(F_i\).  That restriction is already
determined by its ancestor in generation \(m_i(n)\), with the path still
scaled by the total horizon \(n\).

Let \(A_{i,n}\) be the number of generation \(m_i(n)\) particles whose
ancestral paths, scaled by the total horizon \(n\) and restricted to
\([0,t_i]\), lie in \(F_i\).  The event on the left below implies
\(A_{i,n}\ge1\) for at least one \(i\).  Markov's inequality and the many to
one identity \cite{Shi2015} therefore give
\begin{align}
  &\PP\left(\exists \abs{u}=n:
       Q_u^{(n)}\in K_\varepsilon\right)\notag\\
  &\quad\le
  \sum_{i=1}^m\EE A_{i,n}
  =\sum_{i=1}^m
  \rho^{m_i(n)}
  \PP\left(\pi_{t_i}Q_n\in F_i\right).
  \label{eq:outer-m21}
\end{align}
By \Cref{thm:enhanced-ldp}, the LDP for \(\pi_{t_i}Q_n\) has speed \(n\),
so \eqref{eq:F-cost-gap} yields
\[
  \limsup_{n\to\infty}\frac1n
  \log\PP(\pi_{t_i}Q_n\in F_i)
  \le-(t_i\log\rho+2\delta).
\]
Since
\(
\rho^{m_i(n)}=\exp\{nt_i\log\rho+O(1)\},
\)
each summand in \eqref{eq:outer-m21} is at most \(\e^{-\delta n}\) for all sufficiently large \(n\).  The finite sum is summable, proving \eqref{eq:outer-sum}.
\end{proof}

The summable outer containment estimate~\eqref{eq:outer-sum} and eventual
realization of strict paths~\eqref{eq:eventual-tube} now give the two directed
Hausdorff estimates in
\eqref{eq:directed-Hausdorff-components}.

\begin{proof}[Proof of \Cref{thm:path-shape}]
Compactness and nonemptiness of \(\cQ_\rho\) were proved in \Cref{lem:compact-density}.

\emph{Outer containment: realized paths lie near the limit set.}
\Cref{lem:outer-summable} and Borel--Cantelli imply that, for every fixed \(\varepsilon>0\), almost surely every nonempty \(\cQ_n\) eventually satisfies
\begin{equation}\label{eq:outer-H}
  \sup_{q\in\cQ_n}d_{\cX}(q,\cQ_\rho)<\varepsilon.
\end{equation}
On \(\cS\), the random set is nonempty for every \(n\), so \eqref{eq:outer-H}
gives outer containment in \eqref{eq:directed-Hausdorff-components}.

\emph{Inner coverage: every limit path is approximated.}
Fix \(\varepsilon>0\).  Choose a finite \(\varepsilon/3\) net
\(
q_1,\ldots,q_M
\)
of the compact set \(\cQ_\rho\).  By the density part of
\Cref{lem:compact-density}, choose \(\omega_j^{<}\in\cC_\rho^{<}\) so that,
with \(q_j^{<}:=q_{\omega_j^{<}}\),
\[
  d_{\cX}(q_j,q_j^{<})<\varepsilon/3.
\]
By \Cref{prop:eventual-strict}, on \(\cS\), for each \(j\) and all
sufficiently large \(n\), a generation \(n\) particle has enhanced path
within \(\varepsilon/3\) of \(q_j^{<}\).  There are only finitely many \(j\),
so this holds simultaneously.  It follows that eventually
\begin{equation}\label{eq:inner-H}
  \sup_{q\in\cQ_\rho}d_{\cX}(q,\cQ_n)<\varepsilon.
\end{equation}
This proves the inner coverage assertion in
\eqref{eq:directed-Hausdorff-components}.
Taking a countable intersection over \(\varepsilon=1/k\), and combining
\eqref{eq:outer-H} and \eqref{eq:inner-H}, proves
\eqref{eq:path-shape-conv} on \(\cS\).
\end{proof}

The spatial shape and support function statements in
Corollaries~\ref{cor:endpoint-shape} and~\ref{cor:front-support} are
deterministic consequences of the path space convergence.

\begin{proof}[Proof of \Cref{cor:endpoint-shape}]
The endpoint map
\[
  e_1:\cX\to\RR^3,
  \qquad
  e_1(x,y,\zeta)=(x(1),y(1),\zeta(1)),
\]
is Lipschitz with constant \(\sqrt3\) from \(d_{\cX}\) to the Euclidean norm.  Moreover,
\(
e_1(\cQ_n)=\cR_n
\)
and
\(
e_1(\cQ_\rho)=\cR_\rho.
\)
Hausdorff convergence is preserved under a Lipschitz map, giving \eqref{eq:endpoint-shape-conv}.
\end{proof}

It remains to use the elementary stability of support functions under
Hausdorff convergence.

\begin{proof}[Proof of \Cref{cor:front-support}]
Work on the almost sure subset of \(\cS\) on which
\eqref{eq:endpoint-shape-conv} holds.  Every generation is then nonempty, so
\(\cR_n\) is a nonempty finite set, while \(\cR_\rho\) is nonempty and
compact.  Put
\[
  \delta_n:=\dH^{\RR^3}(\cR_n,\cR_\rho).
\]
Fix \(\lambda\in\RR^3\) with \(\norm{\lambda}_2\le1\).  For each
\(z\in\cR_n\), compactness of \(\cR_\rho\) gives a point
\(w_z\in\cR_\rho\) such that
\(\norm{z-w_z}_2\le\delta_n\).  Therefore Cauchy--Schwarz gives
\[
  \ip{\lambda}{z}
  =\ip{\lambda}{w_z}+\ip{\lambda}{z-w_z}
  \le h_\rho(\lambda)
     +\norm{\lambda}_2\norm{z-w_z}_2
  \le h_\rho(\lambda)+\delta_n.
\]
Taking the maximum over \(z\in\cR_n\) yields
\(
h_n(\lambda)\le h_\rho(\lambda)+\delta_n
\).

Conversely, for each \(w\in\cR_\rho\), choose \(z_w\in\cR_n\) with
\(\norm{w-z_w}_2\le\delta_n\).  The same calculation gives
\[
  \ip{\lambda}{w}
  \le\ip{\lambda}{z_w}
     +\norm{\lambda}_2\norm{w-z_w}_2
  \le h_n(\lambda)+\delta_n.
\]
Taking the supremum over \(w\in\cR_\rho\) gives
\(
h_\rho(\lambda)\le h_n(\lambda)+\delta_n
\).  Hence
\[
  \sup_{\norm{\lambda}_2\le1}
  \abs{h_n(\lambda)-h_\rho(\lambda)}
  \le\delta_n\longrightarrow0
\]
by \Cref{cor:endpoint-shape}, proving \eqref{eq:support-uniform}.

The central maximum is the support function in direction \(e_3=(0,0,1)\):
\[
  \frac{M_n^Z}{n^2}=h_n(e_3),
  \qquad
  \Gamma_\rho=h_\rho(e_3).
\]
Thus \eqref{eq:front-law} follows from \eqref{eq:support-uniform}.
\end{proof}

\subsection{A quantitative window below saturation}
\label{sec:below-saturation-window}

The central front law \eqref{eq:front-law} gives convergence on the
\(n^2\) scale.  The following result gives an explicit almost sure window
for each fixed offspring mean below saturation.

\begin{theorem}[Quantitative central maximum below saturation]
\label{thm:below-saturation-window}
Assume \eqref{eq:offspring-assumption} and fix \(1<\rho<4\).  There is a
finite deterministic constant \(C=C(\rho,\operatorname{Law}(N))\) such that,
on \(\cS\), almost surely,
\begin{equation}\label{eq:below-saturation-window}
  \abs{M_n^Z-\Gamma_\rho n^2}
  \le Cn^{3/2}\sqrt{\log n}
\end{equation}
for all sufficiently large \(n\).
\end{theorem}

The proof has two complementary parts.  A block large deviation estimate
simultaneously bounds the action of every lineage and gives the upper bound.
For the lower bound, we contract a maximizing path, track it in many
independent descendant trees, and let the offspring truncation vary with the
contraction parameter.  The next lemma supplies the required one tree
estimate with a vanishing action margin.

\begin{lemma}
\label{lem:vanishing-margin-tracking}
Fix \(1<\rho<4\), and let \(\omega^\star\in\cC_\rho\) satisfy
\(A(\omega^\star)=\Gamma_\rho\).  There are constants
\(\varepsilon_0\in(0,1)\) and
\(C_{\mathrm{tr}},H_0,C_0,C_1\in(0,\infty)\),
depending only on \(\rho\) and the offspring law, with the following
property.  For \(0<\varepsilon\le\varepsilon_0\), define
\begin{equation}\label{eq:dynamic-truncation}
  K_\varepsilon^{\mathrm{tr}}
  :=\left\lceil\e^{C_{\mathrm{tr}}/\varepsilon}\right\rceil,
  \qquad
  N^{(\varepsilon)}:=N\wedge K_\varepsilon^{\mathrm{tr}},
  \qquad
  m_\varepsilon:=\EE N^{(\varepsilon)},
  \qquad
  \omega_\varepsilon:=(1-\varepsilon)\omega^\star.
\end{equation}
At a vertex \(v\), retain its first
\(N_v\wedge K_\varepsilon^{\mathrm{tr}}\) children in Ulam--Harris order;
these retained offspring numbers have law \(N^{(\varepsilon)}\).
If \(h\ge1\) is an integer and \(h\varepsilon^2\ge H_0\), then in this
retained process
\begin{equation}\label{eq:shrinking-tube-probability}
  \PP\left(\exists\abs{u}=h:
    d_{\cX}(Q_u^{(h)},q_{\omega_\varepsilon})\le C_0\varepsilon\right)
  \ge\exp\left\{-C_1\varepsilon^{-1}
                       \log\frac{\e}{\varepsilon}\right\}.
\end{equation}
\end{lemma}

\begin{proof}
Throughout this proof, \(C\) denotes a finite constant depending only on
\(\rho\) and the offspring law and may change from line to line.

\emph{Truncation and the action gap.}
Put \(H_N:=\EE[N\log_+N]\) and \(m_*:=(1+\rho)/2\), and choose
\(C_{\mathrm{tr}}\) so that
\(H_N/(C_{\mathrm{tr}}m_*)\le(\log\rho)/4\).  After decreasing
\(\varepsilon_0\), one has \(m_\varepsilon\ge m_*>1\) and
\begin{align}
  0\le\log\rho-\log m_\varepsilon
  &\le\frac{\rho-m_\varepsilon}{m_*}
   \le\frac{H_N\varepsilon}{C_{\mathrm{tr}}m_*}
   \le\frac14\varepsilon\log\rho.
  \label{eq:dynamic-mean-loss}
\end{align}
Indeed,
\[
  \rho-m_\varepsilon
  =\EE[(N-K_\varepsilon^{\mathrm{tr}})_+]
  \le\EE[N;N>K_\varepsilon^{\mathrm{tr}}]
  \le\frac{H_N}{\log K_\varepsilon^{\mathrm{tr}}}.
\]
By \eqref{eq:quadratic-scaling-L}, for every \(t\in[0,1]\),
\begin{align}
  J_t(\omega_\varepsilon)
  &\le(1-\varepsilon)^2t\log\rho
   \le(1-\varepsilon)t\log\rho\notag\\
  &\le t\left(\log m_\varepsilon
             -\frac12\varepsilon\log\rho\right).
  \label{eq:shrinking-action-margin}
\end{align}
Moreover,
\(\norm{\dot\omega_\varepsilon(t)}_1\le1-\varepsilon\) almost everywhere
and \(A(\omega_\varepsilon)=(1-\varepsilon)^2\Gamma_\rho\).

\emph{Uniform moment bound.}
For \(1\le i\le h\), set
\[
  v_{i,h}:=h\{\omega_\varepsilon(i/h)
             -\omega_\varepsilon((i-1)/h)\}.
\]
Jensen's inequality and \eqref{eq:shrinking-action-margin} give, for every
\(k\le h\),
\begin{equation}\label{eq:dynamic-discrete-action}
  \sum_{i=1}^kL(v_{i,h})
  \le k\left(\log m_\varepsilon
             -\frac12\varepsilon\log\rho\right),
  \qquad
  \norm{v_{i,h}}_1\le1-\varepsilon.
\end{equation}
Define \(\theta_{i,h}\), \(a_{i,h}\), \(H_{u,k}^{(h)}\), and
\(W_{k,h}\) by \eqref{eq:theta-in}, \eqref{eq:weights}, and
\eqref{eq:H-W}, with \(n=h\), \(m=m_\varepsilon\), and target path
\(\omega_\varepsilon\).  The martingale calculation following
\eqref{eq:H-W} applies to the retained process, so
\((W_{k,h})_{k\le h}\) is a nonnegative mean one martingale.

For \(p=1+q\), its one generation moment factor satisfies
\[
\begin{aligned}
  \log b_{i,h}(p)
  &:=
  \log\left(m_\varepsilon\int_Sa_{i,h}(s)^p\dd\mu(s)\right)\\
  &=-q\log m_\varepsilon+\psi((1+q)\theta_{i,h})
       -(1+q)\psi(\theta_{i,h})\\
  &\le q\{L(v_{i,h})-\log m_\varepsilon\}+Cq^2.
\end{aligned}
\]
The last constant is uniform even though the tilts need not remain in a
fixed compact set.  To see this, put
\(d_s:=\max_{r\in S}\ip{\theta}{r}-\ip{\theta}{s}\).  For \(a\ge1\),
\begin{equation}\label{eq:radial-hessian-bound}
  \frac{\dd^2}{\dd a^2}\psi(a\theta)
  =\Var_{a\theta}(d_\xi)
  \le\frac{\sum_{s\in S}d_s^2\e^{-ad_s}}
           {\sum_{s\in S}\e^{-ad_s}}
  \le4\sup_{x\ge0}x^2\e^{-x}.
\end{equation}
Here the denominator is at least one because \(d_s=0\) for a maximizing
step.  Choose
\(\widehat p_\varepsilon:=1+c_2\varepsilon\), with \(c_2>0\) sufficiently
small, and decrease \(\varepsilon_0\) so that
\(\widehat p_\varepsilon\le2\).
Then \eqref{eq:dynamic-discrete-action} gives, for some \(c_3>0\),
\begin{equation}\label{eq:dynamic-moment-decay}
  \prod_{i=1}^kb_{i,h}(\widehat p_\varepsilon)
  \le\e^{-c_3\varepsilon^2k},
  \qquad k\le h.
\end{equation}

Since \(\mu\) is uniform on \(S\), the tilted mass at \(s\) is
\(R_{i,h}(s)/4\), and hence \(R_{i,h}(s)\le4\).  Moreover,
\(\EE[(N^{(\varepsilon)})^{\widehat p_\varepsilon}]
\le (K_\varepsilon^{\mathrm{tr}})^{\widehat p_\varepsilon-1}\EE N
\le C\), because
\((\widehat p_\varepsilon-1)\log K_\varepsilon^{\mathrm{tr}}
\le c_2C_{\mathrm{tr}}+o(1)\).
Thus the offspring sums \(\Xi_{v,k,h}\) in \eqref{eq:mart-diff} satisfy
\[
  \sup_{\substack{0<\varepsilon\le\varepsilon_0,\ h\ge1\\
                  1\le k\le h,\ \abs{v}=k-1}}
  \EE\abs{\Xi_{v,k,h}-1}^{\widehat p_\varepsilon}<\infty.
\]
For \(1<p\le2\), the elementary inequality
\[
  \abs{x+y}^p
  \le\abs{x}^p+p\abs{x}^{p-2}xy+C\abs{y}^p
\]
holds with a universal constant; the middle term is interpreted as zero at
\(x=0\).  Taylor's formula covers \(\abs y\le\abs x/2\), and the
complementary case follows by direct comparison.  Conditional iteration
over the centered summands in \eqref{eq:mart-diff}, as in the proof of
\Cref{lem:Wp}, now gives
\begin{equation}\label{eq:dynamic-Wp}
\begin{aligned}
  \EE\abs{W_{k,h}-W_{k-1,h}}^{\widehat p_\varepsilon}
  &\le C\prod_{i=1}^{k-1}b_{i,h}(\widehat p_\varepsilon)
    \le C\e^{-c_3\varepsilon^2(k-1)},\\
  \EE\abs{W_{h,h}-1}^{\widehat p_\varepsilon}
  &\le C\sum_{k=1}^h\prod_{i=1}^{k-1}b_{i,h}(\widehat p_\varepsilon)
    \le C(1+\varepsilon^{-2}),\\
  \EE W_{h,h}^{\widehat p_\varepsilon}
  &\le C(1+\varepsilon^{-2}).
\end{aligned}
\end{equation}
The linear terms vanish under conditional expectation.  The constants in
the iteration remain uniform as \(\widehat p_\varepsilon\downarrow1\),
while the geometric series is bounded by
\(C(1+\varepsilon^{-2})\).

\emph{Tilted concentration and return to the branching law.}
Under the product tilted law \(\QQ_{h,\varepsilon}\), the increments are
independent with means \(v_{i,h}\).  For
\(M_k:=\sum_{i=1}^k(\xi_i-v_{i,h})\), the exponential supermartingale
inequality, applied to both signs and both coordinates, gives
\begin{equation}\label{eq:dynamic-Azuma}
  \QQ_{h,\varepsilon}\left(
    \max_{k\le h}\norm{M_k}_\infty>r\right)
  \le4\e^{-r^2/(2h)}.
\end{equation}
The interpolation calculation in \eqref{eq:horizontal-uniform-error}, with
\(n=h\) and target \(\omega_\varepsilon\), together with
\Cref{lem:exact-area} and \eqref{eq:area-lipschitz}, yields
\[
  d_{\cX}(Q_h,q_{\omega_\varepsilon})
  \le3\left\{
    \frac1h\max_{k\le h}\norm{M_k}_\infty+\frac1h
  \right\}.
\]
Take \(r=\varepsilon h\).  Since \(h\varepsilon^2\ge H_0\), increasing
\(H_0\) and then fixing \(C_0\) gives
\[
  \QQ_{h,\varepsilon}\left(
  d_{\cX}(Q_h,q_{\omega_\varepsilon})\le C_0\varepsilon\right)
  \ge\frac12.
\]

Let \(Y_{h,\varepsilon}\) be the restricted weighted mass defined as in
\eqref{eq:Y-def}, with horizon \(h\) and the event in the preceding
display.  The weighted many to one identity, the bound
\(0\le Y_{h,\varepsilon}\le W_{h,h}\), and \eqref{eq:dynamic-Wp} imply
\[
  \EE Y_{h,\varepsilon}\ge\frac12,
  \qquad
  \EE Y_{h,\varepsilon}^{\widehat p_\varepsilon}
  \le C(1+\varepsilon^{-2}).
\]
H\"older's inequality therefore gives
\[
  \PP(Y_{h,\varepsilon}>0)
  \ge
  \frac{(\EE Y_{h,\varepsilon})^{
    \widehat p_\varepsilon/(\widehat p_\varepsilon-1)}}
       {(\EE Y_{h,\varepsilon}^{\widehat p_\varepsilon})^{
         1/(\widehat p_\varepsilon-1)}}
  \ge\exp\left\{-C_1\varepsilon^{-1}
                        \log\frac{\e}{\varepsilon}\right\}.
\]
All weights are positive, so this is
\eqref{eq:shrinking-tube-probability}.
\end{proof}

\begin{proof}[Proof of \Cref{thm:below-saturation-window}]
Fix \(1<\rho<4\).
Throughout this proof, finite deterministic constants denoted by \(C\) may
change from line to line and depend only on \(\rho\) and the offspring law.

\emph{Upper estimate: simultaneous block control.}
For \(n\ge3\), put
\[
  \ell_n:=\left\lceil\sqrt{n\log n}\right\rceil,\qquad
  k_n^{\mathrm{blk}}:=\left\lfloor\frac n{\ell_n}\right\rfloor,\qquad
  \Delta_n:=\frac{\ell_n}{n}.
\]
For a particle \(u\) and \(i\ell_n\le\abs u\), let \(s_i(u)\) be its
horizontal displacement between generations \((i-1)\ell_n\) and
\(i\ell_n\), and set \(v_i(u):=s_i(u)/\ell_n\in D\).  Exponential Markov
inequality and \eqref{eq:L-def} give
\[
  \PP\left(\sum_{k=1}^{\ell_n}\xi_k=s\right)
  \le\e^{-\ell_nL(s/\ell_n)},\qquad s\in\ZZ^2,
\]
and each block has at most \((2\ell_n+1)^2\le(3\ell_n)^2\) possible
displacements.

For \(1\le j\le k_n^{\mathrm{blk}}\), let \(U_{j,n}\) count the particles in generation
\(j\ell_n\) satisfying
\[
  \ell_n\sum_{i=1}^jL(v_i(u))
  >j\ell_n\log\rho+4j\log(3\ell_n)+4\log n.
\]
There are at most \((3\ell_n)^{2j}\) block displacement sequences.  Applying
the many to one identity \cite{Shi2015} to each sequence gives
\[
\begin{aligned}
  \EE U_{j,n}
  &\le \rho^{j\ell_n}(3\ell_n)^{2j}
    \exp\{-j\ell_n\log\rho-4j\log(3\ell_n)-4\log n\}\\
  &=n^{-4}(3\ell_n)^{-2j}.
\end{aligned}
\]
Consequently,
\[
  \sum_{n\ge3}\PP\left(
    \sum_{j=1}^{k_n^{\mathrm{blk}}}U_{j,n}\ge1\right)
  \le\sum_{n\ge3}n^{-4}
    \sum_{j=1}^{k_n^{\mathrm{blk}}}(3\ell_n)^{-2j}
  <\infty.
\]
Define
\(\alpha_n:=\{4\log(3\ell_n)+4\log n\}/\ell_n\).
Borel--Cantelli therefore gives a single almost sure event on which, for all
sufficiently large \(n\), every generation \(n\) particle \(u\), and every
\(j\le k_n^{\mathrm{blk}}\),
\begin{equation}\label{eq:block-action-window}
  \ell_n\sum_{i=1}^jL(v_i(u))
  \le j\ell_n(\log\rho+\alpha_n).
\end{equation}
Here the left hand side is evaluated along the ancestor \(u_{j\ell_n}\).

\emph{Conversion from block action to central area.}
Let \(\overline\omega_{u,n}\) join the values of \(\omega_u^{(n)}\) at the
times \(i\Delta_n\), \(0\le i\le k_n^{\mathrm{blk}}\), by linear segments and remain
constant afterwards.  Choose an integer \(r_0\) with
\(r_0\log\rho>\log4\).  We delay this block path by \(r_0\) intervals and
discard its last \(r_0\) blocks, so that the initial action slack absorbs
the possible cost \(\log4\) of a partial block.  For all large \(n\),
\(k_n^{\mathrm{blk}}>r_0\), and we define
\[
  \kappa_n(t):=\min\{(t-r_0\Delta_n)_+,
                     (k_n^{\mathrm{blk}}-r_0)\Delta_n\},\qquad
  \widetilde\omega_{u,n}(t):=\overline\omega_{u,n}(\kappa_n(t)).
\]
If
\(t=(r_0+j-1+\sigma)\Delta_n\), where
\(1\le j\le k_n^{\mathrm{blk}}-r_0\) and \(0\le\sigma\le1\), then
\eqref{eq:block-action-window} and \eqref{eq:L-range} give
\[
\begin{aligned}
  J_t(\widetilde\omega_{u,n})
  &\le\Delta_n\{(j-1)(\log\rho+\alpha_n)+\sigma\log4\},\\
  t(\log\rho+\alpha_n)-J_t(\widetilde\omega_{u,n})
  &\ge\Delta_n\{
     r_0(\log\rho+\alpha_n)
     +\sigma(\log\rho+\alpha_n-\log4)\}\\
  &\ge\Delta_n(r_0\log\rho-\log4)>0.
\end{aligned}
\]
The constant parts satisfy the same action inequality, so
\(\widetilde\omega_{u,n}\in\cC_{\rho\e^{\alpha_n}}\).

On a complete block, both the original path and its linear segment remain
within \(\Delta_n\) of their common initial point, and the same estimate
applies to the final incomplete block.  Since \(\overline\omega_{u,n}\) is
Lipschitz with constant one in the maximum norm,
\[
  \norm{\omega_u^{(n)}-\overline\omega_{u,n}}_\infty\le2\Delta_n,
  \qquad
  \norm{\overline\omega_{u,n}-\widetilde\omega_{u,n}}_\infty
  \le r_0\Delta_n.
\]
Thus
\(\norm{\omega_u^{(n)}-\widetilde\omega_{u,n}}_\infty
\le(r_0+2)\Delta_n\).  Both paths belong to \(\cV\).  Hence
\eqref{eq:exact-area-identity} and \eqref{eq:area-lipschitz} imply
\begin{equation}\label{eq:upper-area-window}
  \frac{Z_u}{n^2}
  \le\Gamma_{\rho\e^{\alpha_n}}+3(r_0+2)\Delta_n.
\end{equation}
To remove the small enlargement of the offspring mean, for \(\alpha\ge0\)
set
\(r_\alpha:=(\log\rho/(\log\rho+\alpha))^{1/2}\).  If
\(\omega\in\cC_{\rho\e^\alpha}\), then
\eqref{eq:quadratic-scaling-L} shows that \(r_\alpha\omega\in\cC_\rho\).
Since area is quadratic,
\begin{equation}\label{eq:front-log-enlargement}
  \Gamma_{\rho\e^\alpha}
  \le\left(1+\frac{\alpha}{\log\rho}\right)\Gamma_\rho.
\end{equation}
Since
\[
  n^2\alpha_n=O(n^{3/2}\sqrt{\log n}),
  \qquad
  n^2\Delta_n=O(n^{3/2}\sqrt{\log n}),
\]
\eqref{eq:upper-area-window} gives, almost surely for all sufficiently large
\(n\),
\begin{equation}\label{eq:below-saturation-upper}
  M_n^Z\le\Gamma_\rho n^2+C_\rho n^{3/2}\sqrt{\log n}.
\end{equation}

\emph{Lower estimate: amplification of the one tree bound.}
Choose \(\omega^\star\in\cC_\rho\) with
\(A(\omega^\star)=\Gamma_\rho\).  Such a maximizer exists because
\(\cQ_\rho\) is compact and endpoint evaluation is continuous.  Apply
\Cref{lem:vanishing-margin-tracking} to this path, and use the truncation
and constants supplied there.
Set
\[
  \varepsilon_n:=\sqrt{\frac{\log n}{n}},\qquad
  g_n:=\left\lceil C_{\mathrm{amp}}\varepsilon_n^{-1}
                    \log\frac{\e}{\varepsilon_n}\right\rceil,\qquad
  h_n:=n-g_n,
\]
where \(C_{\mathrm{amp}}>2C_1/\log\rho\).  Then
\(g_n=O(n\varepsilon_n)=o(n)\) and
\(h_n\varepsilon_n^2\to\infty\).  In particular, for all sufficiently
large \(n\), \(\varepsilon_n\le\varepsilon_0\) and
\(h_n\varepsilon_n^2\ge H_0\).  By \eqref{eq:KS}, almost surely on
\(\cS\),
\[
  \mathcal N_{g_n}\ge\e^{g_n\log\rho/2}
\]
for all large \(n\).  Put
\(\mathsf M_n:=\lfloor\e^{g_n\log\rho/2}\rfloor\).  On the event
\(\{\mathcal N_{g_n}\ge\mathsf M_n\}\), select the first
\(\mathsf M_n\) particles in Ulam--Harris order and apply the truncation
\(N^{(\varepsilon_n)}\) independently in their descendant trees.  Each
selected tree has success probability at least
\[
  q_n:=\exp\left\{-C_1\varepsilon_n^{-1}
                        \log\frac{\e}{\varepsilon_n}\right\}.
\]
The choice of \(C_{\mathrm{amp}}\) is precisely what makes the number of
trees dominate the inverse success probability:
\[
  \log(\mathsf M_nq_n)
  \ge\left(\frac{C_{\mathrm{amp}}\log\rho}{2}-C_1\right)
      \varepsilon_n^{-1}\log\frac{\e}{\varepsilon_n}+O(1).
\]
Thus, for some \(c>0\),
\[
  \PP(\text{all selected trees fail}\mid\cF_{g_n})
  \le\e^{-\mathsf M_nq_n}
  \le\exp\left\{-\exp\left(
    c\varepsilon_n^{-1}\log\frac{\e}{\varepsilon_n}\right)\right\}
\]
on \(\{\mathcal N_{g_n}\ge\mathsf M_n\}\).  The last expression is
summable.  Multiplication by the indicator of this event and integration
give the same bound for the joint failure event.  Borel--Cantelli shows that,
almost surely on \(\cS\), one selected tree succeeds for every sufficiently
large \(n\).

Let \(v_n\) be the first root in Ulam--Harris order whose selected tree
succeeds, and let \(u_n\) be the first corresponding generation \(n\)
descendant.  Write
\[
  V(v_n)^{-1}V(u_n)
  =(\widetilde X_n,\widetilde Y_n,\widetilde Z_n).
\]
The tube estimate at time one gives
\[
  \widetilde Z_n
  \ge h_n^2\{(1-\varepsilon_n)^2\Gamma_\rho-C_0\varepsilon_n\}.
\]
The Heisenberg group law gives the noncommutative splice
\[
  Z_{u_n}=Z_{v_n}+\widetilde Z_n+X_{v_n}\widetilde Y_n.
\]
Since
\(\abs{X_{v_n}}\le g_n\), \(\abs{Z_{v_n}}\le g_n^2\), and
\(\abs{\widetilde Y_n}\le h_n\),
\[
\begin{aligned}
  M_n^Z
  &\ge h_n^2\{(1-\varepsilon_n)^2\Gamma_\rho-C_0\varepsilon_n\}
       -g_nh_n-g_n^2\\
  &\ge\Gamma_\rho n^2
       -C\{n^2\varepsilon_n+ng_n+g_n^2\}.
\end{aligned}
\]
Finally,
\[
  n^2\varepsilon_n=n^{3/2}\sqrt{\log n},\qquad
  ng_n=O(n^{3/2}\sqrt{\log n}),\qquad
  g_n^2=O(n\log n).
\]
Therefore, almost surely on \(\cS\), for all sufficiently large \(n\),
\begin{equation}\label{eq:below-saturation-lower}
  M_n^Z\ge\Gamma_\rho n^2-Cn^{3/2}\sqrt{\log n}.
\end{equation}
Combining \eqref{eq:below-saturation-upper} and
\eqref{eq:below-saturation-lower} proves
\eqref{eq:below-saturation-window}.
\end{proof}

\subsection{Finite generation corrections in the saturated regimes}
\label{sec:saturated-corrections}

We now return to the random particle cloud and prove the two parts of
\Cref{thm:saturated-corrections-main}.  In these regimes
\(\Gamma_\rho=1/4\).  The exact deterministic maximum
\(\lfloor n^2/4\rfloor\) differs from \(\Gamma_\rho n^2\) by less than one,
so both reference values give the same limits in
\Cref{thm:saturated-corrections-main}.  Above saturation the correction is
controlled by the first lineage that can continue forever in an extremal
horizontal direction.  At the threshold it is governed by a critical
minimal displacement theorem.

The next lemma separates the endpoint already reached from the largest area
that can be accumulated during the remaining steps.

\begin{lemma}\label{lem:future-area}
Suppose a partial word has first horizontal coordinate \(x\), and
\(k\) steps remain.  If \(\abs{x}\le k\), the largest additional central
coordinate that a deterministic continuation can create is
\begin{equation}\label{eq:future-area}
  F_k(x)=\left\lfloor\frac{(k+\abs{x})^2}{4}\right\rfloor.
\end{equation}
\end{lemma}

\begin{proof}
If the continuation uses \(h\) horizontal steps, it uses \(k-h\) vertical
steps, during which the absolute first coordinate is at most
\(\abs{x}+h\).  Conversely, this bound is attained by taking all horizontal
steps first in the direction of \(x\), followed by vertical steps with the
matching sign.  Therefore
\[
  F_k(x)
  =\max_{\substack{0\le h\le k\\ h\in\ZZ}}
    (\abs{x}+h)(k-h)
  =\left\lfloor\frac{(k+\abs{x})^2}{4}\right\rfloor.
\]
\end{proof}

Taking \(x=0\) and \(k=n\) in \Cref{lem:future-area} shows that the largest
central coordinate of a deterministic word of length \(n\) is
\(\lfloor n^2/4\rfloor\).  Reflecting every vertical increment changes the
sign of the central coordinate without changing the word length.  Thus, in
the saturated regimes, every such word satisfies
\[
  \abs{Z}\le\left\lfloor\frac{n^2}{4}\right\rfloor
  \le\Gamma_\rho n^2.
\]

\subsubsection{Above saturation}

When \(\rho>4\), the retained process in either extremal horizontal
direction is supercritical.  The corresponding minimal defects eventually
stabilize.

For \(m\ge0\), with the minimum of an empty set understood as \(+\infty\),
define the minimal horizontal defects at generation \(m\) by
\begin{equation}\label{eq:Dm-directional}
  \mathfrak d_m^+:=\min_{\abs{u}=m}(m-X_u),
  \qquad
  \mathfrak d_m^-:=\min_{\abs{u}=m}(m+X_u),
  \qquad
  \mathfrak d_m:=\mathfrak d_m^+\wedge\mathfrak d_m^-.
\end{equation}
On \(\cS\), these quantities are finite.  Each of \(\mathfrak d_m^+\) and \(\mathfrak d_m^-\)
is nondecreasing because every edge adds a nonnegative defect in the
corresponding direction.

\begin{proposition}[Eventual stabilization of horizontal defects]
\label{prop:oriented-defect}
If \(\rho>4\), then on \(\cS\), almost surely,
\[
  \mathfrak d_m^+=\mathfrak d^+,\qquad
  \mathfrak d_m^-=\mathfrak d^-
  \qquad\text{for all sufficiently large }m.
\]
\end{proposition}

\begin{proof}
\emph{Existence of an immortal directional ray.}
The descendants reached using only \(\R\) edges form a Galton--Watson
process obtained by independently retaining each child with probability
\(1/4\); its offspring mean is \(\rho/4>1\).  Denote its positive survival
probability by \(p_{\R}\).  Conditional on \(\cF_m\), the survival events
for these retained descendant trees from the
\(\mathcal N_m\) vertices in generation \(m\) are independent, so the
conditional probability that none of them occurs is
\((1-p_{\R})^{\mathcal N_m}\).  On \(\cS\),
  \eqref{eq:KS} gives \(\mathcal N_m\sim
  W_{\mathrm{GW}}\rho^m\to\infty\) almost surely, whereas the
population is eventually zero on \(\cS^c\).  Conditioning on \(\cF_m\)
therefore gives
\begin{align*}
  &\PP\bigl(\cS\cap
    \{\text{no vertex starts an infinite \(\R\) retained subtree}\}\bigr)\\
  &\qquad\le
  \EE\left[
    \one_{\{\mathcal N_m>0\}}
    (1-p_{\R})^{\mathcal N_m}
  \right]
  \longrightarrow0.
\end{align*}
The last limit follows by dominated convergence: the integrand tends to
zero on \(\cS\), while on extinction its indicator eventually vanishes.
Hence, almost surely on \(\cS\), some vertex starts an infinite \(\R\)
retained subtree, and therefore \(\mathfrak d^+<\infty\).  The argument for
\(\mathfrak d^-\) is identical.

\emph{Identification of the stabilized defect.}
The integer sequence \(\mathfrak d_m^+\) is nondecreasing and, after reaching a
vertex that realizes \(\mathfrak d^+\), is bounded above by \(\mathfrak d^+\).  It stabilizes
at an integer \(d\le \mathfrak d^+\).  Choose a generation \(m_0\) after
stabilization.  For every \(m\ge m_0\), there is a descendant of some
generation \(m_0\) vertex whose plus defect is \(d\); all vertices on the
intervening ancestral segment also have defect \(d\), because defects
cannot decrease.  One of the finitely many generation \(m_0\) vertices
therefore has descendants of defect \(d\) at arbitrarily large depths.
K\"onig's lemma supplies an infinite ray of constant defect \(d\), and
every edge on that ray is labelled \(\R\).  Hence \(\mathfrak d^+\le d\), proving
\(d=\mathfrak d^+\).  The minus direction is the same.
\end{proof}

\begin{proof}[Proof of \Cref{thm:saturated-corrections-main}\textup{(i)}]
\emph{Upper bound from stabilized defects.}
Fix a random \(m_0\) after which \Cref{prop:oriented-defect} holds.
Every generation \(m_0\) particle then satisfies
\[
  \abs{X_u}\le m_0-\mathfrak d.
\]
Its central coordinate is at most \(\Gamma_\rho m_0^2\).  For all
sufficiently large \(n\), \Cref{lem:future-area} gives
\[
  M_n^Z
  \le \Gamma_\rho m_0^2
     +\left\lfloor\frac{(n-\mathfrak d)^2}{4}\right\rfloor,
\]
and \(\Gamma_\rho=1/4\).  Consequently,
\[
  \Gamma_\rho n^2-M_n^Z
  \ge\frac{n^2-m_0^2}{4}
     -\left\lfloor\frac{(n-\mathfrak d)^2}{4}\right\rfloor
  =\frac{\mathfrak d}{2}n+O(1),
\]
where the random \(O(1)\) term is independent of \(n\).  Hence
\begin{equation}\label{eq:super-sat-liminf}
  \liminf_{n\to\infty}
  \frac{\Gamma_\rho n^2-M_n^Z}{n}\ge\frac{\mathfrak d}{2}.
\end{equation}

\emph{Construction of a matching \(\R/\U\) lineage.}
For the reverse inequality, suppose \(\mathfrak d=\mathfrak d^+\).  For each fixed
vertex \(v\), apply Kesten--Stigum to the descendant tree obtained by
retaining only \(\R\) edges.  Its retained offspring number \(N_{\R}\)
satisfies
\(0\le N_{\R}\le N\), so
\(\EE[N_{\R}\log_+N_{\R}]\le\EE[N\log_+N]<\infty\).  Fix
\(\varrho_{\R}\in(1,\rho/4)\).  Since the vertex set is countable, we may
work on a single event of probability one on which, for every
\(v\) whose retained tree survives, there is a random \(c_v>0\) such that,
with
\[
  \mathcal N_h^{\R}(v)
  :=\#\{w\in\mathbb T_v^{\R}:\abs{w}-\abs{v}=h\},
\]
the generation size satisfies
\[
  \mathcal N_h^{\R}(v)\ge c_v\varrho_{\R}^h
\]
for all sufficiently large \(h\).

Let \(\mathcal B_h(v)\) be the event that none of the vertices in that
generation starts an infinite subtree using only \(\U\) edges, and let
\(p_{\U}>0\) be the survival probability of such a tree.  For the fixed
vertex \(v\), let \(\cF_h^{\R}(v)\) contain its retained history through
generation \(h\), but not the offspring numbers or increments leaving
generation \(h\).  Conditional on this sigma field, the relevant
\(\U\) retained descendant trees are independent, and
\[
  \PP(\mathcal B_h(v)\mid\cF_h^{\R}(v))
  =(1-p_{\U})^{\mathcal N_h^{\R}(v)}.
\]
For integers \(j\ge1\), the event
\(\mathcal H_{j,h}(v):=
\{\mathcal N_h^{\R}(v)\ge j^{-1}\varrho_{\R}^h\}\) is
\(\cF_h^{\R}(v)\)-measurable, and therefore
\[
  \PP(\mathcal B_h(v)\cap\mathcal H_{j,h}(v))
  \le \exp\{-p_{\U}j^{-1}\varrho_{\R}^h\}.
\]
The right hand side is summable in \(h\).  Borel--Cantelli, followed by a
countable intersection over \(j\) and the vertices \(v\), shows that, on
survival of an \(\R\) retained tree, every sufficiently large retained
generation contains a vertex that starts an infinite \(\U\) retained
subtree.  We may now choose a vertex \(u_*\) realizing \(\mathfrak d^+\) whose
\(\R\) retained tree survives; the simultaneous almost sure statement
applies to it.

For each sufficiently large \(n\), choose nonnegative integers
\(n_{\R}^\star,n_{\U}^\star\) with
\[
  n_{\R}^\star+n_{\U}^\star=n-\mathfrak d,
  \qquad
  n_{\R}^\star n_{\U}^\star
  =\left\lfloor\frac{(n-\mathfrak d)^2}{4}\right\rfloor,
\]
and put \(h_n^+:=n_{\R}^\star-X_{u_*}\).  Since
  \(X_{u_*}=\abs{u_*}-\mathfrak d\), one has
\(h_n^++n_{\U}^\star=n-\abs{u_*}\), and \(h_n^+\ge0\) for all large \(n\).
The preceding
paragraph supplies a lineage whose remaining labels are
\(\R^{h_n^+}\U^{n_{\U}^\star}\).  Therefore
\[
\begin{aligned}
  Z_{\mathrm{final}}
  &=Z_{u_*}+X_{u_*}n_{\U}^\star
    +h_n^+n_{\U}^\star\\
  &=Z_{u_*}+n_{\R}^\star n_{\U}^\star
   =Z_{u_*}
    +\left\lfloor\frac{(n-\mathfrak d)^2}{4}\right\rfloor.
\end{aligned}
\]
The middle term in the first line is the cross term from the Heisenberg
group law.  Thus
\[
  M_n^Z\ge Z_{u_*}
       +\left\lfloor\frac{(n-\mathfrak d)^2}{4}\right\rfloor.
\]
Together with \eqref{eq:super-sat-liminf}, this proves
\eqref{eq:super-saturated-main} when
\(\mathfrak d=\mathfrak d^+\).

\emph{The symmetric direction.}
If \(\mathfrak d=\mathfrak d^-\), use the same construction with
\(\Lft\) and \(\Dwn\).  Choose
\(u_*\) realizing \(\mathfrak d^-\), set
\(h_n^-:=n_{\R}^\star+X_{u_*}\), and use a continuation
\(\Lft^{h_n^-}\Dwn^{n_{\U}^\star}\).  Since
\(X_{u_*}=\mathfrak d-\abs{u_*}\), one again has
\(h_n^-+n_{\U}^\star=n-\abs{u_*}\), while the group law gives
\[
  Z_{\mathrm{final}}
  =Z_{u_*}-X_{u_*}n_{\U}^\star
    +h_n^-n_{\U}^\star
  =Z_{u_*}+n_{\R}^\star n_{\U}^\star.
\]
This proves the reverse inequality in the remaining case.
\end{proof}

\subsubsection{The critical saturation point}

At \(\rho=4\), the extremal horizontal lineages are critical and their
minimal defect grows on the \(\log\log n\) scale.  We first state the
critical minimal displacement theorem in the precise form needed here.
For a branching random walk with bounded nonnegative integer displacements,
let \(K(j)\) be the number of children at displacement \(j\), let
\(K_*:=\sum_{j\ge0}K(j)\), and let \(\mathcal M_n\) be the minimum position
in generation \(n\).

\begin{proposition}[Critical minimal displacement theorem]
\label{prop:DH-theorem}
Assume that there is an integer \(j_{\max}<\infty\) such that \(K(j)=0\)
almost surely for \(j>j_{\max}\), that \(\EE K(j)>0\) for every
\(1\le j\le j_{\max}\), and
that
\[
  \EE K_*>1,\qquad \EE K_*^2<\infty,\qquad
  \EE K(0)=1,\qquad \PP(K(0)=1)<1.
\]
Then, on the nonextinction event of the branching random walk,
\begin{equation}\label{eq:DH-general}
  \frac{\mathcal M_n\log2}{\log\log n}\longrightarrow1
  \qquad\text{almost surely}.
\end{equation}
\end{proposition}

\begin{proof}
The critical minimal displacement theorem of Dekking and Host
\cite{DekkingHost1991}, restricted to bounded displacements, gives the result
under their standing convention \(\PP(K_*=0)=0\).  They state that the
Galton--Watson decomposition \cite{AthreyaNey1972} extends it to the case
\(\EE K_*>1\).

\emph{Upper comparison by the immortal skeleton.}
On nonextinction, consider the immortal skeleton formed by the particles
with infinite lines of descent.  For a skeleton particle, let
\(K^{\mathrm{sk}}(j)\) be the number of its children at displacement \(j\)
that also have infinite lines of descent, and let
\(p_{\mathrm{surv}}>0\) be the survival probability of the original
process.  Write \(\PP_{\mathrm{sk}}\) and \(\EE_{\mathrm{sk}}\) for
probability and expectation under the offspring law conditional on the
parent belonging to the skeleton.  Then
\[
  \EE_{\mathrm{sk}} K^{\mathrm{sk}}(j)
  =\frac{p_{\mathrm{surv}}\EE K(j)}{p_{\mathrm{surv}}}
  =\EE K(j),
\]
because each child has an independent probability \(p_{\mathrm{surv}}\)
of infinite descent.  Thus, with
\(K_*^{\mathrm{sk}}:=\sum_jK^{\mathrm{sk}}(j)\),
\[
  \EE_{\mathrm{sk}} K_*^{\mathrm{sk}}=\EE K_*>1,
  \qquad
  \EE_{\mathrm{sk}}[(K_*^{\mathrm{sk}})^2]
  \le p_{\mathrm{surv}}^{-1}\EE[K_*^2]<\infty.
\]
The skeleton also retains the critical zero displacement mean and the
positivity assumptions for \(1\le j\le j_{\max}\).  Moreover,
\(\PP_{\mathrm{sk}}(K^{\mathrm{sk}}(0)=1)<1\): otherwise the skeleton would contain an
infinite path of zero displacements, whereas the nondegenerate critical
Galton--Watson process with offspring law \(K(0)\) dies out almost surely.
Every skeleton particle has at least one skeleton child, so this process
satisfies the no extinction convention of the cited theorem.  If
\(\mathcal M_n^{\mathrm{sk}}\) denotes its generation \(n\) minimum, then
\(\mathcal M_n\le\mathcal M_n^{\mathrm{sk}}\).  The cited limit therefore
gives the required upper bound for \(\mathcal M_n\).

\emph{Lower comparison by enlargement.}
For the reverse bound, whenever \(K_*=0\), add one artificial child at
displacement one, and let artificial children reproduce independently
according to the same enlarged law.  The resulting branching random walk
contains the original process, never becomes extinct, and still satisfies
all the displayed critical, moment, and displacement assumptions.  Applying
the cited theorem to this enlarged process gives the matching lower bound:
if \(\mathcal M_n^{\mathrm{enl}}\) is its minimum, then
\(\mathcal M_n^{\mathrm{enl}}\le\mathcal M_n\).
The two comparisons prove \eqref{eq:DH-general} on nonextinction.  The two
applications below have displacement supports \(\{0,1,2\}\) and
\(\{0,1\}\), respectively.
\end{proof}

The next lemma identifies the conditional law of the first descendant tree
satisfying a given event.  In the construction below, the event is survival
of the retained process.

\begin{lemma}
\label{lem:first-success-tree}
Let \(\mathcal T_1,\ldots,\mathcal T_M\), conditionally on a sigma field
\(\mathcal G\), be independent copies of a marked random tree
\(\mathcal T\), where the finite nonnegative integer \(M\) is measurable with
respect
to \(\mathcal G\) and the descendant trees are independent of
\(\mathcal G\).  Let
\(\mathcal A\) be an event determined by one descendant tree with
\(p:=\PP(\mathcal T\in\mathcal A)>0\), and let \(i_*\) be the first index
for which \(\mathcal T_{i_*}\in\mathcal A\), with \(i_*:=\infty\) if no
such index exists.  Conditionally on
\(\{i_*\le M\}\) and \(\mathcal G\), the selected tree
\(\mathcal T_{i_*}\) has the law of \(\mathcal T\) conditioned on
\(\mathcal A\).
\end{lemma}

\begin{proof}
For any measurable set \(E\), conditional independence gives
\begin{align*}
  \PP(\mathcal T_{i_*}\in E,\ i_*\le M\mid\mathcal G)
  &=\sum_{j=1}^M(1-p)^{j-1}
    \PP(\mathcal T\in E\cap\mathcal A)\\
  &=\frac{1-(1-p)^M}{p}
    \PP(\mathcal T\in E\cap\mathcal A).
\end{align*}
Since
\(\PP(i_*\le M\mid\mathcal G)=1-(1-p)^M\), on \(\{M\ge1\}\) division by
this probability proves the asserted conditional law.  When \(M=0\), the
conditioning event is empty.
\end{proof}

At \(\rho=4\), define the plus defect of an edge as one minus the first
coordinate of its horizontal increment.  It takes the values \(0,1,2\)
with probabilities \(1/4,1/2,1/4\).  For \(j\ge0\), let \(N^+(j)\) be
the number of children of a particle whose edge has plus defect \(j\).
Then
\[
  \EE N^+(0)=1,
  \qquad \EE N^+(1)=2,
  \qquad \EE N^+(2)=1.
\]
Moreover, \(\EE N^2<\infty\) implies
\(\EE(\sum_jN^+(j))^2<\infty\), and
\(\PP(N^+(0)=1)<1\) because
\(\PP(N^+(0)=0)=\EE(3/4)^N>0\).  The total offspring mean is four, and
nonextinction is exactly \(\cS\).  Thus \Cref{prop:DH-theorem}, applied in
the plus and minus directions, gives simultaneously on \(\cS\),
\begin{equation}\label{eq:DH-defect-law}
  \frac{\mathfrak d_m^+}{\log\log m}\longrightarrow\frac1{\log2},
  \qquad
  \frac{\mathfrak d_m^-}{\log\log m}\longrightarrow\frac1{\log2}
  \quad\text{almost surely}.
\end{equation}
Hence
\[
  \frac{\mathfrak d_m}{\log\log m}\longrightarrow\frac1{\log2}
  \qquad\text{almost surely on }\cS.
\]

The following deterministic estimate converts the two block cost
into a bound on the loss of central area.

\begin{lemma}\label{lem:monotone-two-block}
Let \(n_{\mathrm{blk}}=k+\ell\), where \(\abs{k-\ell}\le1\).  Consider a
word of length \(n_{\mathrm{blk}}\) using only \(\R\) and \(\U\).  Let
\(d_{\U}\) be the number of \(\U\) steps among the first \(k\) letters,
let \(d_{\R}\) be the number of \(\R\) steps among the last \(\ell\)
letters, and put \(d_{\mathrm{blk}}:=d_{\U}+d_{\R}\).  If \(Z\) is its
central coordinate, then
\begin{equation}\label{eq:monotone-two-block-bound}
  k\ell-Z
  \le\frac{n_{\mathrm{blk}}}{2}d_{\mathrm{blk}}
     +2d_{\mathrm{blk}}^2+2d_{\mathrm{blk}}.
\end{equation}
\end{lemma}

\begin{proof}
For an \(\R/\U\) word, \(Z\) is the number of ordered pairs in which an
\(\R\) precedes a \(\U\).  If
\[
  R_{\mathrm{tot}}=k-d_{\U}+d_{\R},
  \qquad
  U_{\mathrm{tot}}=\ell+d_{\U}-d_{\R}
\]
are the total numbers of the two letters, then
\(Z=R_{\mathrm{tot}}U_{\mathrm{tot}}-\operatorname{Inv}\), where
\(\operatorname{Inv}\) is the
number of inversions \(\U\) before \(\R\).  Since
\[
  k\ell-R_{\mathrm{tot}}U_{\mathrm{tot}}
  =(\ell-k)(d_{\U}-d_{\R})+(d_{\U}-d_{\R})^2,
\]
its positive part is at most
\(d_{\mathrm{blk}}+d_{\mathrm{blk}}^2\).

Every inversion either starts with one of the \(d_{\U}\) early
\(\U\) letters or, after those pairs are removed, ends with one of the
\(d_{\R}\) late \(\R\) letters.  Hence
\[
  \operatorname{Inv}
  \le d_{\U}k+d_{\U}d_{\R}+d_{\R}\ell
  \le\frac{n_{\mathrm{blk}}}{2}d_{\mathrm{blk}}
      +d_{\mathrm{blk}}+d_{\U}d_{\R}.
\]
Since
\(d_{\U}d_{\R}\le d_{\mathrm{blk}}^2/4\), adding the preceding two bounds
gives \eqref{eq:monotone-two-block-bound}.
\end{proof}

\begin{proof}[Proof of \Cref{thm:saturated-corrections-main}\textup{(ii)}]
\emph{Almost sure lower bound on the loss.}
First take \(m=\lfloor n^{1/3}\rfloor\).  Every generation \(m\)
particle satisfies
\(
m-\abs{X_u}\ge \mathfrak d_m
\).
By \Cref{lem:future-area},
\[
  M_n^Z\le \Gamma_\rho m^2
    +\left\lfloor\frac{(n-\mathfrak d_m)^2}{4}\right\rfloor.
\]
Since \(\Gamma_\rho=1/4\),
\[
\begin{aligned}
  \Gamma_\rho n^2-M_n^Z
  &\ge\frac{n^2-m^2}{4}
    -\left\lfloor\frac{(n-\mathfrak d_m)^2}{4}\right\rfloor\\
  &=\frac n2\mathfrak d_m-\frac14\mathfrak d_m^2
    -\frac14m^2+O(1).
\end{aligned}
\]
Now \(\log\log m\sim\log\log n\),
\(\mathfrak d_m\sim(\log\log m)/\log2\) almost surely on \(\cS\), and
\(m^2=o(n\log\log n)\).  Hence
\begin{equation}\label{eq:critical-liminf}
  \liminf_{n\to\infty}
  \frac{\Gamma_\rho n^2-M_n^Z}{n\log\log n}
  \ge\frac1{2\log2}
  \qquad\text{almost surely on }\cS.
\end{equation}

\emph{Selection of a surviving \(\R/\U\) subtree.}
For the matching bound, choose deterministic integers
\[
  j_n\uparrow\infty,\qquad j_n=o(\log\log n),
\]
and put
\[
  \bar n:=n-j_n,
  \qquad
  k_n^{(1)}:=\lfloor\bar n/2\rfloor,
  \qquad
  k_n^{(2)}:=\bar n-k_n^{(1)}.
\]
From each generation \(j_n\) particle retain only edges labelled \(\R\) or
\(\U\).  The retained offspring number is a binomial
thinning of \(N\) with parameter \(1/2\); it has mean two and finite
second moment, so its survival probability \(p_{\R\U}\) is positive.
Conditionally on \(\cF_{j_n}\), the retained descendant trees are
independent.  Let \(w_n\) be the first generation \(j_n\) particle in
Ulam--Harris order whose retained subtree survives, when such a particle
exists.  Hence
\[
  \PP(w_n\text{ does not exist}\mid\cF_{j_n})
  =(1-p_{\R\U})^{\mathcal N_{j_n}}.
\]
Moreover,
\[
  \PP(w_n\text{ does not exist}\mid\cS)
  \le\frac{1}{\PP(\cS)}
  \EE\left[
    \one_{\{\mathcal N_{j_n}>0\}}
    (1-p_{\R\U})^{\mathcal N_{j_n}}
  \right]
  \longrightarrow0.
\]
Indeed, the integrand tends almost surely to zero: on \(\cS\) the
population tends to infinity, while on \(\cS^c\) it is eventually zero.
On the event that \(w_n\) exists, \Cref{lem:first-success-tree} shows that
its marked retained subtree
has exactly the retained tree law conditioned on survival.

\emph{Reduction to a homogeneous critical displacement.}
Count generations in the selected subtree from \(w_n\).  For a descendant
label sequence \(\eta=(\eta_1,\ldots,\eta_{\bar n})\), define its two block
cost by
\[
  D_{\mathrm{blk}}(\eta)
  :=\sum_{i=1}^{k_n^{(1)}}\one_{\{\eta_i=\U\}}
    +\sum_{i=k_n^{(1)}+1}^{\bar n}\one_{\{\eta_i=\R\}}.
\]
Let \(\widehat{\mathfrak d}_{\bar n}\) be the minimum of
\(D_{\mathrm{blk}}(\eta)\) over all generation \(\bar n\) descendant label
sequences in the selected subtree.
On the event that \(w_n\) does not exist, set
\(\widehat{\mathfrak d}_{\bar n}:=0\).  The map that swaps the labels
\(\R\leftrightarrow\U\) on all edges after generation
\(k_n^{(1)}\) preserves
the marked retained tree law.  It also preserves its survival event,
which depends only on the unlabelled retained tree.  Under this
measure preserving map, the two block cost becomes the homogeneous cost
that assigns zero to \(\R\) and one to \(\U\) at every generation.
Thus the first successful tree lemma and the label swap give
\[
  \mathcal L\left(
    \widehat{\mathfrak d}_{\bar n}
    \,\middle|\,w_n\text{ exists},\cF_{j_n}\right)
  =
  \mathcal L\left(
    \mathcal M_{\bar n}^{\mathrm{hom}}
    \,\middle|\,\text{homogeneous survival}\right),
\]
where \(\mathcal M_{\bar n}^{\mathrm{hom}}\) is the minimum displacement
of the homogeneous process.

For the homogeneous process, the expected numbers of zero cost and one
cost children are both \(\rho/4=1\); there are no other displacements.
Its total offspring mean is two, and its total offspring has finite second
moment.  Its number of zero cost children has the same law as \(N^+(0)\),
so \(\PP(N^{+}(0)=1)<1\) by the verification above.  By
\Cref{prop:DH-theorem} and the conditional law in
\Cref{lem:first-success-tree},
\begin{equation}\label{eq:good-defect-law}
  \frac{\widehat{\mathfrak d}_{\bar n}}{\log\log\bar n}
  \xrightarrow{\PP(\,\cdot\mid w_n\text{ exists})}
  \frac1{\log2}.
\end{equation}
Together with the vanishing probability that \(w_n\) does not exist,
this convergence holds under \(\PP(\,\cdot\mid\cS)\) for the selected
subtree construction.

\emph{Conversion from block defect to central area.}
On the event that \(w_n\) exists, choose a minimizing \(\R/\U\) word in
the selected subtree.  Its
\(d_{\U},d_{\R},d_{\mathrm{blk}}\) in
\Cref{lem:monotone-two-block} satisfy
\(d_{\mathrm{blk}}=\widehat{\mathfrak d}_{\bar n}\).  Let
\((X_n^{\mathrm{rel}},Y_n^{\mathrm{rel}},Z_n^{\mathrm{rel}})\) be the
Heisenberg position obtained by multiplying the \(\bar n\) edge labels after
\(w_n\).  The lemma gives
\[
  k_n^{(1)}k_n^{(2)}-Z_n^{\mathrm{rel}}
  \le\frac{\bar n}{2}\widehat{\mathfrak d}_{\bar n}
     +2\widehat{\mathfrak d}_{\bar n}^2
     +2\widehat{\mathfrak d}_{\bar n}.
\]
The Heisenberg product formula writes the full central coordinate as
\[
  Z_{w_n}+X_{w_n}Y_n^{\mathrm{rel}}+Z_n^{\mathrm{rel}}.
\]
Here \(\abs{X_{w_n}}\le j_n\),
\(\abs{Y_n^{\mathrm{rel}}}\le\bar n\), and
\(\abs{Z_{w_n}}\le\Gamma_\rho j_n^2\).  Moreover
\(\abs{\Gamma_\rho n^2-k_n^{(1)}k_n^{(2)}}\le Cnj_n\) for a
deterministic constant \(C\).  Consequently, on the event that \(w_n\)
exists,
\[
  \Gamma_\rho n^2-M_n^Z
  \le\frac n2\widehat{\mathfrak d}_{\bar n}
    +C\{\widehat{\mathfrak d}_{\bar n}^2
       +\widehat{\mathfrak d}_{\bar n}+nj_n\}.
\]
Under the conditional law \(\PP(\,\cdot\mid\cS)\),
\eqref{eq:good-defect-law} says that the family
\(\widehat{\mathfrak d}_{\bar n}/\log\log n\) is tight, because
\(\log\log\bar n\sim\log\log n\).  With both \(O_{\PP}\) and
\(o_{\PP}\) understood relative to this conditional law, this is
equivalently
\[
  \widehat{\mathfrak d}_{\bar n}=O_{\PP}(\log\log n).
\]
Hence the quadratic defect error is \(o_{\PP}(n\log\log n)\), while the
initial generation error satisfies \(nj_n=o(n\log\log n)\).  Dividing the
preceding area estimate by \(n\log\log n\) gives
\[
\begin{aligned}
  \frac{\Gamma_\rho n^2-M_n^Z}{n\log\log n}
  &\le\frac12
    \frac{\widehat{\mathfrak d}_{\bar n}}{\log\log n}
    +o_{\PP}(1)\\
  &\xrightarrow{\PP}\frac1{2\log2}
\end{aligned}
\]
under \(\PP(\,\cdot\mid\cS)\).  Together with
\eqref{eq:critical-liminf}, this proves
\eqref{eq:critical-saturated-main}.
\end{proof}

\section{Geometry of the limit shape}\label{sec:geometry}

This section proves the geometric results for \(\cR_\rho\) stated in
Section~2.  We begin with its projection, fibers, symmetries, and interior.
We then determine the saturated shape and compare shapes at different
offspring means.  Finally, cyclic rooting and planar anisotropic isoperimetry
identify the origin fiber, bound every fiber below saturation, and yield
universal nonconvexity.

Throughout this and the next section, write \(B:=\log\rho\) for a fixed
offspring mean \(\rho>1\).  For
\(v=(f,g)\in L^2([0,1];\RR^2)\), define
\begin{equation}\label{eq:area-control}
  \mathfrak A(v)
  :=\int_0^1\left(\int_0^t f(s)\dd s\right)g(t)\dd t.
\end{equation}
Thus \(A(\omega)=\mathfrak A(\dot\omega)\) whenever
\(\omega\in\ac([0,1];\RR^2)\) starts at the origin and has square
integrable velocity.
Let \(\mathsf J:\RR^2\to\RR^2\) denote the quarter turn
\begin{equation}\label{eq:J-quarter-turn}
  \mathsf J(x,y):=(-y,x).
\end{equation}

\begin{analystcheck}

\subsection{\texorpdfstring{\additionmark{Basic structure of the limit
shape}}{Basic structure of the limit shape}}

\begin{revisioncheck}
We first determine the horizontal projection and the interval structure of
the vertical fibers.  We then prove that the limit shape has nonempty
interior, identify the fibers over the boundary of the horizontal
projection, and establish its symmetries and Carnot star shapedness.
\end{revisioncheck}

\begin{proposition}[Projection and interval fibers]\label{prop:fibers}
For every \(\rho>1\),
\begin{equation}\label{eq:horizontal-proj}
  \cD_\rho=\{(a,b)\in\RR^2:L(a,b)\le\log\rho\},
\end{equation}
and every nonempty fiber \revisionmark{\(\cI_\rho(a,b)\)} is a compact
interval.
\begin{analystadditioncheck}
In addition, the portion of the saddle sheet above \(\cD_\rho\) lies in the
limit shape:
\begin{equation}\label{eq:saddle-sheet-main}
  \left\{\left(a,b,\frac{ab}{2}\right):(a,b)\in\cD_\rho\right\}
  \subset\cR_\rho.
\end{equation}
\end{analystadditioncheck}
\end{proposition}

\begin{proof}
If \(\omega\in\cC_\rho\) ends at \((a,b)\), Jensen's inequality gives
\[
  L(a,b)
  =L\left(\int_0^1\dot\omega(t)\dd t\right)
  \le\int_0^1L(\dot\omega(t))\dd t
  \le\log\rho.
\]
Conversely, if \(L(a,b)\le\log\rho\), the constant velocity path \(\omega(t)=t(a,b)\) satisfies
\(
J_t(\omega)=tL(a,b)\le t\log\rho
\)
for every \(t\).  This proves \eqref{eq:horizontal-proj}.
\begin{analystadditioncheck}
The same constant velocity path has
\[
  A(\omega)=\int_0^1ta\,b\dd t=\frac{ab}{2}.
\]
It therefore proves \eqref{eq:saddle-sheet-main}.
\end{analystadditioncheck}

Fix \((a,b)\in\cD_\rho\).  Identify an admissible path with its velocity
control \(v=(v_1,v_2)\), and define the set \(\cU_{a,b}\) of admissible
velocity controls with endpoint \((a,b)\) by
\[
  \cU_{a,b}:=\left\{v:
  \begin{array}{l}
  v\in L^1([0,1];\RR^2),\\
  v(t)\in D\text{ for a.e. }t,\\
  \displaystyle\int_0^1v(t)\dd t=(a,b),\\[1mm]
  \displaystyle\int_0^tL(v(s))\dd s\le t\log\rho
  \text{ for every }t\in[0,1]
  \end{array}
  \right\}.
\]
Because \(D\) is bounded, every \(v\in\cU_{a,b}\) belongs to
\(L^\infty([0,1];\RR^2)\), so \(\mathfrak A(v)\) is defined by
\eqref{eq:area-control}.
\begin{revisioncheck}
The set \(\cU_{a,b}\) is convex.  Indeed, if
\(v,w\in\cU_{a,b}\) and \(0\le\lambda\le1\), convexity of \(L\) gives
\end{revisioncheck}
\[
  \int_0^tL(\lambda v+(1-\lambda)w)\dd s
  \le\lambda\int_0^tL(v(s))\dd s
     +(1-\lambda)\int_0^tL(w(s))\dd s
  \le t\log\rho.
\]
For \(v\in\cU_{a,b}\), set
\(x_v(t):=\int_0^t v_1(s)\dd s\).  The corresponding path has area
\(\mathfrak A(v)\).  Equip the controls with the vector \(L^1\) norm
\[
  \norm{v-w}_{L^1}:=\int_0^1\norm{v(t)-w(t)}_1\dd t.
\]
The functional \(\mathfrak A\) is continuous in this norm on
controls taking values in \(D\).  Indeed,
\begin{align*}
  \abs{\mathfrak A(v)-\mathfrak A(w)}
  &\le\norm{x_v-x_w}_\infty\norm{v_2}_{L^1}
      +\norm{x_w}_\infty\norm{v_2-w_2}_{L^1}\\
  &\le2\norm{v-w}_{L^1}.
\end{align*}
Thus \(\mathfrak A(\cU_{a,b})=\revisionmark{\cI_\rho(a,b)}\) is connected,
hence an interval.  It is closed and bounded because it is the fiber of the
compact set \(\cR_\rho\), so it is a compact interval.
\end{proof}

\begin{analystadditioncheck}
The next construction strengthens the fiber description by producing an
open set of admissible endpoints.

\begin{proposition}[Nonempty interior]\label{prop:nonempty-interior}
For every \(\rho>1\),
\begin{equation}\label{eq:nonempty-interior-main}
  \operatorname{int}(\cR_\rho)\ne\varnothing.
\end{equation}
\end{proposition}

\begin{proof}
For parameters \(\alpha,\beta,\gamma\), consider the piecewise constant
velocity
\[
  v_{\alpha,\beta,\gamma}(t)
  :=
  \begin{cases}
    (\alpha,0),&0\le t<1/3,\\
    (0,\gamma),&1/3\le t<2/3,\\
    (\beta,0),&2/3\le t\le1.
  \end{cases}
\]
Continuity of \(L\) at zero and \(L(0)=0\) give \(\varepsilon\in(0,1/2)\) such
that \(L(w)<\log\rho\) whenever
\(\norm{w}_\infty<\varepsilon\).  For
\(\abs{\alpha},\abs{\beta},\abs{\gamma}<\varepsilon\), all accumulated
action inequalities hold and are strict for \(t\in(0,1]\).  Hence the
corresponding paths belong to \(\cC_\rho\).  Their lifted endpoints are
\[
  \mathcal F(\alpha,\beta,\gamma)
  =
  \left(
  \frac{\alpha+\beta}{3},
  \frac{\gamma}{3},
  \frac{\alpha\gamma}{9}
  \right),
\]
and
\[
  \det D\mathcal F(\alpha,\beta,\gamma)=\frac{\gamma}{81}.
\]
Choose a sufficiently small parameter triple with \(\gamma\ne0\).  The
inverse function theorem shows that the image under \(\mathcal F\) of a
smaller open
neighborhood contains an open subset of \(\RR^3\).  This image lies in
\(\cR_\rho\), proving \eqref{eq:nonempty-interior-main}.
\end{proof}

Although interior fibers may have positive length, strict convexity forces
the fibers above the boundary of the horizontal projection to collapse.

\begin{proposition}[Fibers on the action level]\label{prop:action-level-fiber}
If \(1<\rho<4\) and \(L(a,b)=\log\rho\), then
\begin{equation}\label{eq:action-level-singleton-main}
  \cI_\rho(a,b)=\left\{\frac{ab}{2}\right\}.
\end{equation}
\end{proposition}

\begin{proof}
Let \(\omega\in\cC_\rho\) end at \((a,b)\), and write
\(v=\dot\omega\).  Jensen's inequality and the constraint at time one give
\[
  \log\rho
  =L\left(\int_0^1v(t)\dd t\right)
  \le\int_0^1L(v(t))\dd t
  \le\log\rho.
\]
Suppose first that \((a,b)\in\operatorname{int}D\), and put
\(\theta=\nabla L(a,b)\).  The Bregman divergence
\[
  L(w)-L(a,b)-\ip{\theta}{w-(a,b)}
\]
is nonnegative on \(D\) and vanishes only at \(w=(a,b)\).  Its integral at
\(w=v(t)\) is zero, because equality holds in the preceding display and
\(\int_0^1v(t)\dd t=(a,b)\).  Hence \(v(t)=(a,b)\) for almost every \(t\).

If \((a,b)\) lies in the relative interior of an edge of \(D\), an exposing
linear functional first forces \(v(t)\) to lie on the same edge almost
everywhere.  Writing the endpoints of that edge as \(s_1,s_2\), its points
have the form \(\lambda s_1+(1-\lambda)s_2\), and
\eqref{eq:entropy-L} reduces there to
\[
  L=\log4+\lambda\log\lambda+(1-\lambda)\log(1-\lambda).
\]
Strict convexity in \(\lambda\) again makes the velocity constant.  A vertex
cannot occur because \(L=\log4>\log\rho\).  Thus in every case
\(\omega(t)=t(a,b)\), and \(A(\omega)=ab/2\).  This proves
\eqref{eq:action-level-singleton-main}.
\end{proof}
\end{analystadditioncheck}

\begin{revisioncheck}
We next establish the transformations of the limit shape used in the
saturated calculation and the nonconvexity argument.
\end{revisioncheck}

\begin{proposition}[Symmetries and Carnot star shapedness]\label{prop:sym-star}
For every \(\rho>1\), the set \(\cR_\rho\) is invariant under
\begin{equation}\label{eq:symmetries}
  \begin{aligned}
  (a,b,c)&\mapsto(-a,b,-c),&
  (a,b,c)&\mapsto(a,-b,-c),\\
  (a,b,c)&\mapsto(b,a,ab-c).&&
  \end{aligned}
\end{equation}
Moreover, it is Carnot star shaped:
\[
  \delta_r\cR_\rho\subset\cR_\rho,
  \qquad 0\le r\le1.
\]
\end{proposition}

\begin{proof}
The rate \(L\) is invariant under \((v_1,v_2)\mapsto(-v_1,v_2)\), \((v_1,-v_2)\), and \((v_2,v_1)\).  Replacing \((x,y)\) by \((-x,y)\) or \((x,-y)\) changes the area sign and gives the first two symmetries.  Exchanging \(x\) and \(y\), integration by parts gives
\[
  \int_0^1y\dd x=x(1)y(1)-\int_0^1x\dd y,
\]
which gives the third symmetry.

If \(\omega\in\cC_\rho\), then \eqref{eq:quadratic-scaling-L} gives
\[
  J_t(r\omega)\le r^2J_t(\omega)\le J_t(\omega)\le t\log\rho,
\]
so \(r\omega\in\cC_\rho\).  Moreover,
\(
A(r\omega)=r^2A(\omega).
\)
Taking endpoints proves Carnot star shapedness.
These symmetries also imply
\(\max\cI_\rho(s,s)-\max\cI_\rho(s,-s)=s^2\) whenever
\((s,s)\in\cD_\rho\): coordinate exchange centers
\(\cI_\rho(s,s)\) at \(s^2/2\), while reflection of the second coordinate
gives \(\cI_\rho(s,-s)=-\cI_\rho(s,s)\).
\end{proof}

\subsection{\texorpdfstring{\additionmark{Saturation and comparison across
offspring means}}{Saturation and comparison across offspring means}}

\begin{revisioncheck}
We first show that at \(\rho=4\) the action constraints impose no
restriction beyond \(\dot\omega\in D\) almost everywhere.  We then compute
the resulting limit shape and use the quadratic scaling of \(L\) to
compare different offspring means.
\end{revisioncheck}

\begin{proposition}[Saturation]\label{prop:saturation}
For every \(\rho>1\), \(0\le\Gamma_\rho\le1/4\).  For \(\rho\ge4\),
\revisionmark{\(\cC_\rho=\cC_4\), \eqref{eq:saturation} holds, and
\(\Gamma_\rho=1/4\).}
\end{proposition}

\begin{proof}
The zero path gives \(\Gamma_\rho\ge0\).  Let
\(\omega=(x,y)\in\cC_\rho\).  Since \(\dot\omega\in D\) almost everywhere,
\(\Var(x)+\Var(y)\le1\).  Also
\(\norm{x}_\infty\le\Var(x)\), so
\begin{equation}\label{eq:area-quarter}
  A(\omega)
  \le\int_0^1\abs{x\dot y}\dd t
  \le\Var(x)\Var(y)\le\frac14.
\end{equation}

By \eqref{eq:L-range}, every absolutely continuous path starting at zero and
having velocity in \(D\) satisfies
\(
J_t(\omega)\le t\log4.
\)
Thus \(\cC_4\) is precisely the set of all such horizontal paths, and the
same is true for every \(\rho\ge4\); hence \(\cC_\rho=\cC_4\).  The path that
travels at unit \(\ell^1\) speed from \((0,0)\) to \((1/2,0)\) and then to
\((1/2,1/2)\) has total duration one and area \(1/4\).  Therefore the upper
bound is attained.
\end{proof}

\begin{analystadditioncheck}
\begin{revisioncheck}
At saturation, \(\cR_4\) is the set reachable with horizontal
\(\ell^1\) length at most one.  In sub-Finsler terminology, this is the unit
ball, whose boundary has piecewise quadratic panels; see
\cite{DuchinMooney2014}.  The following calculation identifies every panel
in the coordinates used here.
\end{revisioncheck}

\begin{proposition}[Complete saturated limit shape]
\label{prop:complete-saturated-shape}
For every \(\rho\ge4\), \eqref{eq:complete-saturated-shape-main} holds.
In particular, for \(0\le s\le1/2\),
\begin{equation}\label{eq:diagonal-fiber}
  \max\cI_4(s,s)=\frac{(1+2s)^2}{16}.
\end{equation}
\end{proposition}

\begin{proof}
\emph{Upper boundary in the first quadrant.}
By \Cref{prop:saturation}, it is enough to determine \(\cR_4\).  First fix
\(a,b\ge0\) with \(a+b\le1\), and let
\(\omega=(x,y)\in\cC_4\) end at \((a,b)\).  Write
\[
  H_x:=\int_0^1\abs{\dot x(t)}\dd t,
  \qquad
  H_y:=\int_0^1\abs{\dot y(t)}\dd t.
\]
Then \(H_x\ge a\), \(H_y\ge b\), and \(H_x+H_y\le1\).  Put
\[
  m_x:=\min_{0\le t\le1}x(t),
  \qquad
  M_x:=\max_{0\le t\le1}x(t),
  \qquad
  d:=M_x-m_x.
\]
Since \(x(0)=0\), one has \(m_x\le0\).  If the path visits \(m_x\) before
\(M_x\), its horizontal variation is at least
\[
  (0-m_x)+(M_x-m_x)+(M_x-a)=2d-a.
\]
If it visits \(M_x\) first, the corresponding lower bound is
\[
  M_x+(M_x-m_x)+(a-m_x)=2d+a.
\]
Thus in both cases
\begin{equation}\label{eq:saturated-width-bound}
  d\le\frac{H_x+a}{2}.
\end{equation}

Let \(H_y^+\) and \(H_y^-\) be the positive and negative variations of
\(y\).
Since the net vertical displacement is \(b\),
\[
  H_y^+=\frac{H_y+b}{2},
  \qquad
  H_y^-=\frac{H_y-b}{2}.
\]
The Jordan decomposition of \(\dd y\), together with \(H_y^+-H_y^-=b\),
gives
\begin{align}
  A(\omega)
  &\le M_xH_y^+-m_xH_y^-
   =dH_y^++m_x(H_y^+-H_y^-)\notag\\
  &\le dH_y^+
   \le\frac{(H_x+a)(H_y+b)}{4}.
  \label{eq:saturated-area-bound}
\end{align}
Here the second line uses \(m_x\le0\) and \(b\ge0\).
The last expression is increasing in each of \(H_x\) and \(H_y\).  Its
maximum under \(H_x\ge a\), \(H_y\ge b\), and \(H_x+H_y\le1\) therefore
occurs with \(H_x+H_y=1\).  It remains to maximize
\[
  \frac{(H_x+a)(1-H_x+b)}4,
  \qquad a\le H_x\le1-b.
\]
The unconstrained maximizer is \(H_{x,0}=(1+b-a)/2\).  Hence
\begin{equation}\label{eq:saturated-first-quadrant-max}
  \sup_{\substack{\omega\in\cC_4\\ \omega(1)=(a,b)}}A(\omega)
  =
  \begin{cases}
    \displaystyle\frac{(1+a+b)^2}{16},
      &3a-b\le1\ \text{and}\ 3b-a\le1,\\[2mm]
    \displaystyle\frac{a(1-a+b)}2,
      &3a-b\ge1,\\[2mm]
    \displaystyle\frac{b(1+a-b)}2,
      &3b-a\ge1.
  \end{cases}
\end{equation}

\emph{An extremizing coordinate path.}
Let \(H_x^*\) be a maximizing value in the preceding one dimensional
problem and set
\[
  H_y^*:=1-H_x^*,\qquad
  d^*:=\frac{H_x^*+a}{2},\qquad
  H_{y,+}^*:=\frac{H_y^*+b}{2},\qquad
  H_{y,-}^*:=\frac{H_y^*-b}{2}.
\]
Follow the \revisionmark{path consisting of four coordinate segments}
\[
  (0,0)\longrightarrow(0,-H_{y,-}^*)
  \longrightarrow(d^*,-H_{y,-}^*)
  \longrightarrow(d^*,b)\longrightarrow(a,b).
\]
All four segment lengths are nonnegative, their sum is
\[
  H_{y,-}^*+d^*+H_{y,+}^*+(d^*-a)=H_y^*+H_x^*=1,
\]
and the area is \(d^*H_{y,+}^*\).  Parametrizing at unit \(\ell^1\) speed
therefore
gives a path in \(\cC_4\) attaining
\eqref{eq:saturated-first-quadrant-max}.

\emph{Symmetry and completion of the fibers.}
For an endpoint \((a,b)\) of arbitrary sign, define the reversed rooted path
\[
  \omega^{\leftarrow}(t):=(a,b)-\omega(1-t).
\]
It has the same endpoint and horizontal length as \(\omega\), belongs to
\(\cC_4\), and satisfies
\begin{equation}\label{eq:saturated-reversal-area}
  A(\omega^{\leftarrow})=ab-A(\omega).
\end{equation}
Thus every saturated fiber is symmetric about \(ab/2\).  The transformations
in \eqref{eq:symmetries} show that its centered half width depends only on
\((\abs{a},\abs{b})\) and is symmetric in these two arguments.  Subtracting
\(ab/2\) from \eqref{eq:saturated-first-quadrant-max} gives exactly
\(\Delta_4(\abs{a},\abs{b})\).  Finally, \Cref{prop:fibers} fills every height
between the two fiber endpoints, proving
\eqref{eq:complete-saturated-shape-main}.

For \(a=b=s\) with \(0\le s\le1/2\), the first case in
\eqref{eq:saturated-first-quadrant-max} applies and gives
\eqref{eq:diagonal-fiber}.
\end{proof}

The exact saturated shape and the quadratic scaling of the rate function now
give comparisons between arbitrary offspring means.

\begin{corollary}[Scaling across offspring means and comparison with the
saturated shape]
\label{cor:saturation-sandwich}
For \(1<\rho_1\le\rho_2\),
\begin{equation}\label{eq:cross-rate-shape-main}
  \delta_{\sqrt{\log\rho_1/\log\rho_2}}\cR_{\rho_2}
  \subset\cR_{\rho_1}\subset\cR_{\rho_2}.
\end{equation}
For \(1<\rho<4\), this gives
\eqref{eq:saturation-sandwich-main} and
\eqref{eq:saturation-Hausdorff-main}.
\end{corollary}

\begin{proof}
Put
\[
  r:=\sqrt{\frac{\log\rho_1}{\log\rho_2}}.
\]
If \(\omega\in\cC_{\rho_2}\), then
\eqref{eq:quadratic-scaling-L} gives
\[
  J_t(r\omega)
  =\int_0^tL(r\dot\omega(s))\dd s
  \le r^2J_t(\omega)
  \le r^2t\log\rho_2
  =t\log\rho_1.
\]
Thus \(r\omega\in\cC_{\rho_1}\).  Since
\(A_\cdot(r\omega)=r^2A_\cdot(\omega)\), taking lifted endpoints gives the
first inclusion in \eqref{eq:cross-rate-shape-main}.  The second follows
from \(\cC_{\rho_1}\subset\cC_{\rho_2}\).

Now take \(\rho_1=\rho\) and \(\rho_2=4\).  Then
\(r=r_\rho=\sqrt{\log\rho/\log4}\), so
\eqref{eq:saturation-sandwich-main} follows.  For
\((a,b,c)\in\cR_4\), the point
\(\delta_{r_\rho}(a,b,c)\) belongs to \(\cR_\rho\).  Since
\(a^2+b^2\le1\) and \(\abs{c}\le1/4\),
\[
  \norm{(a,b,c)-\delta_{r_\rho}(a,b,c)}_2^2
  =(1-r_\rho)^2(a^2+b^2)+(1-r_\rho^2)^2c^2
  \le(1-r_\rho)^2+\frac{(1-r_\rho^2)^2}{16}.
\]
The other directed distance is zero because
\(\cR_\rho\subset\cR_4\).  This proves
\eqref{eq:saturation-Hausdorff-main}.
\end{proof}
\end{analystadditioncheck}

\subsection{Wulff duality, origin fibers, and nonconvexity}

\begin{revisioncheck}
At the horizontal origin, every admissible path is a closed loop.  Cyclic
rooting converts a terminal action constraint for such a loop into the full
family of constraints at intermediate times.  Planar anisotropic
isoperimetry then gives the exact origin fiber, bounds the remaining fibers,
and supplies the construction used for universal nonconvexity.
\end{revisioncheck}

\subsubsection{\texorpdfstring{\revisionmark{Cyclic rooting and anisotropic
isoperimetry}}{Cyclic rooting and anisotropic isoperimetry}}

\begin{lemma}\label{lem:cyclic-rooting}
\begin{revisioncheck}
Fix \(1<\rho<4\).  Let \(\gamma:[0,1]\to\RR^2\) be an absolutely continuous closed curve
such that
\[
  \int_0^1L(\dot\gamma(t))\dd t\le\log\rho.
\]
There is a cyclic shift of \(\gamma\), followed by a translation in
\(\RR^2\), that starts at the origin, has the same action and signed area,
and belongs to \(\cC_\rho\).  Consequently,
\begin{equation}\label{eq:Z0-total-action}
  Z_0(\rho)=\max\left\{
  \int_0^1x(t)\dot y(t)\dd t:
  \begin{array}{l}
  \gamma=(x,y)\text{ is closed},\\
  \displaystyle\int_0^1L(\dot\gamma(t))\dd t\le\log\rho
  \end{array}\right\}.
\end{equation}
\end{revisioncheck}
\end{lemma}

\begin{proof}
Put
\[
  \mathfrak b_\gamma(t):=\int_0^t\{L(\dot\gamma(s))-B\}\dd s
\]
and choose \(t_0\) at which \(\mathfrak b_\gamma\) attains its maximum.
Cyclically extend
the velocity and set
\[
  \widetilde\gamma(t):=\gamma((t_0+t)\bmod1)-\gamma(t_0),
\]
with the continuous interpretation at the wrap time.  If \(t_0+t\le1\),
then
\[
  \int_0^t\{L(\dot{\widetilde\gamma}(s))-B\}\dd s
  =\mathfrak b_\gamma(t_0+t)-\mathfrak b_\gamma(t_0)\le0.
\]
If \(t_0+t>1\), the same integral equals
\[
  \mathfrak b_\gamma(1)-\mathfrak b_\gamma(t_0)
    +\mathfrak b_\gamma(t_0+t-1)
  \le \mathfrak b_\gamma(1)\le0.
\]
Thus \revisionmark{\(J_t(\widetilde\gamma)\le tB\) for every
\(t\in[0,1]\)}.  Translation does not change the
line integral because the curve is closed, and cyclic shifting changes
neither the action nor the line integral.
\begin{revisioncheck}
Conversely, every \(\omega\in\cC_\rho\) ending at the horizontal origin is
a closed curve with total action at most \(B\), so it is included in the
class on the right hand side of \eqref{eq:Z0-total-action}.
\end{revisioncheck}
The maximum is attained because the origin fiber of the compact set
\revisionmark{\(\cR_\rho\)} is compact.
\end{proof}

\begin{revisioncheck}
We shall use the planar Wulff inequality in a form that also permits
self intersections.  For a centrally symmetric convex body
\(K\subset\RR^2\) containing the origin in its interior, let \(h_K\) be
its support function as defined in \eqref{eq:support-function-def}, and set
\[
  \mathcal P_K(\gamma)
  :=\int_0^1h_K(\dot\gamma(t))\dd t
\]
for an absolutely continuous closed curve \(\gamma\).  Because \(K\) is
centrally symmetric, \(h_K\) is a norm and \(\mathcal P_K\) is independent
of the orientation of the parametrization.
\end{revisioncheck}

\begin{lemma}
\label{lem:anisotropic-loop}
Let \(K\subset\RR^2\) be a centrally symmetric convex body with
\(0\in\operatorname{int}K\).  Every absolutely continuous closed curve
\(\gamma=(x,y)\) satisfies
\begin{equation}\label{eq:anisotropic-loop}
  \mathcal P_K(\gamma)^2
  \ge4|K|\left|\int_0^1x(t)\dot y(t)\dd t\right|.
\end{equation}
\end{lemma}

\begin{proof}
For a positively oriented Jordan curve, the planar Wulff inequality
applied to \(-\mathsf J K\) gives \eqref{eq:anisotropic-loop}.  Indeed, if
\(\tau\) and \(\nu=-\mathsf J\tau\) are the unit tangent and outer normal,
then
\[
  h_{-\mathsf J K}(\nu)=h_K(\mathsf J\nu)=h_K(\tau),
\]
so its anisotropic perimeter is exactly \(\mathcal P_K(\gamma)\), while
\(|-\mathsf J K|=|K|\).
\begin{revisioncheck}
In the normalization used here, the squared anisotropic perimeter is at
least four times the product of the Euclidean areas of the Wulff shape and
the enclosed set; see \cite{EspositoFuscoTrombetti2005}.
\end{revisioncheck}

We next pass from Jordan curves to arbitrary polygonal loops.  After an
arbitrarily small perturbation of the vertices, one may assume that the
polygon has only finitely many transverse crossings and no overlapping
edges.  Subdivide its oriented edges at the crossings.  The incoming and
outgoing edge multiplicities then balance at every vertex.  Starting from
an unused edge and following unused outgoing edges until returning to the
initial vertex produces an oriented cycle.  Splitting at any repeated
vertex and iterating this procedure decomposes the polygon into finitely
many oriented simple cycles \(\gamma_j\), with
edge multiplicities partitioning those of the original loop.  Line
integrals and anisotropic lengths are additive under this decomposition.
Writing \(A_j\) for the signed Euclidean area and \(P_j\) for the
anisotropic length of \(\gamma_j\), we obtain
\[
  \left|\int x\dd y\right|
  =\left|\sum_jA_j\right|
  \le\sum_j|A_j|
  \le\frac1{4|K|}\sum_jP_j^2
  \le\frac1{4|K|}\left(\sum_jP_j\right)^2.
\]
The same conclusion follows for a nongeneric polygon by perturbation and
continuity.

Finally, choose refining partitions of a rectifiable closed curve for
which the inscribed polygonal lengths in the norm \(h_K\) converge to
\(\mathcal P_K(\gamma)\).  The polygons converge uniformly to \(\gamma\)
and have a common total variation bound.  Riemann--Stieltjes integration
by parts, exactly as in \Cref{lem:area-continuity}, shows that their
algebraic areas converge to \(\int x\dd y\).  Passing to the limit proves
\eqref{eq:anisotropic-loop}.
\end{proof}

\subsubsection{\texorpdfstring{\revisionmark{The Wulff parameter and the exact
origin fiber}}{The Wulff parameter and the exact origin fiber}}

\begin{revisioncheck}
To specialize the preceding inequality to the Cram\'er transform, we need
the area of its log Laplace sublevel sets and the unique scalar parameter at
which the resulting Wulff bound is minimized.
\end{revisioncheck}

\begin{hyperaddition}
\emph{Hyperbolic representation of the dual potential.}
Use the coordinates \(p=(\alpha+\beta)/2\), \(q=(\alpha-\beta)/2\) from
\eqref{eq:psi-cosh-factor}.  The map
\[
  X(p,q):=(\cosh p\cosh q,\sinh p\cosh q,\sinh q)
\]
parametrizes the hyperboloid model
\[
  \mathbf H^2:=\{X\in\RR^3:-X_0^2+X_1^2+X_2^2=-1,\ X_0>0\}
\]
of the hyperbolic plane of curvature \(-1\).  The induced metric in these
Fermi coordinates is \(\cosh^2q\,\mathrm dp^2+\mathrm dq^2\).  If
\(o=(1,0,0)\), the hyperboloid distance formula gives
\[
  \cosh d(o,X(p,q))=X_0=\cosh p\cosh q,
  \qquad
  \psi(\alpha,\beta)=\log\cosh d(o,X(p,q)).
\]
Consequently, the image of \(K_h\) is the closed hyperbolic disk with
center \(o\) and radius \(r_h:=\operatorname{arcosh}(\e^h)\).
The area forms satisfy
\[
  \mathrm dA_\theta=2\,\mathrm dp\,\mathrm dq,
  \qquad
  \mathrm dA_{\mathrm{hyp}}=\cosh q\,\mathrm dp\,\mathrm dq.
\]
Hence the Euclidean area required by the Wulff inequality is
\[
  a(h)=2\int_{\{X:\,d(o,X)\le r_h\}}
              \operatorname{sech}q\,\mathrm dA_{\mathrm{hyp}}.
\]

This auxiliary metric is distinct from the Hessian metric of \(\psi\).
In the same coordinates, the latter is
\(\operatorname{sech}^2p\,\mathrm dp^2+
\operatorname{sech}^2q\,\mathrm dq^2\), which becomes
\(\mathrm d\xi^2+\mathrm d\eta^2\) under
\(\xi=\arcsin(\tanh p)\), \(\eta=\arcsin(\tanh q)\).
\end{hyperaddition}

\begin{hyperreplacement}
\begin{lemma}\label{lem:ah-asymptotics}
The function \(a\) in \eqref{eq:Kh-ah-main} is smooth and strictly
increasing on \((0,\infty)\).  In terms of the complete elliptic integral
\(\mathbf K\),
\[
  a'(h)=8\mathbf K\!\left(\sqrt{1-\e^{-2h}}\right),
  \qquad
  a(h)=8\int_0^h\mathbf K\!\left(\sqrt{1-\e^{-2s}}\right)\dd s.
\]
In particular,
\begin{align}
  a(h)&=4\pi h+\pi h^2+O(h^3),&&h\downarrow0,\label{eq:ah-small}\\
  a'(h)&=4\pi+2\pi h+O(h^2),&&h\downarrow0,\label{eq:ah-derivative-small}\\
  a(h)&=4(h+\log4)^2+O(1),&&h\to\infty.\label{eq:ah-large}
\end{align}
Also, \(a'(h)=8(h+\log4)+O(h\e^{-2h})\) as \(h\to\infty\).
\end{lemma}

\begin{proof}
In hyperbolic polar coordinates about \(o\),
\[
  X=(\cosh r,\sinh r\cos\varphi,\sinh r\sin\varphi),
  \qquad
  \mathrm dA_{\mathrm{hyp}}=\sinh r\,\mathrm dr\,\mathrm d\varphi.
\]
Since \(\sinh q=X_2\), the weighted area formula becomes
\[
  a(h)=2\int_0^{r_h}\int_0^{2\pi}
       \frac{\sinh r}{\sqrt{1+\sinh^2r\sin^2\varphi}}
       \dd\varphi\dd r.
\]
For \(h>0\), the integrand is smooth and
\(r_h'=\cosh r_h/\sinh r_h\).  Therefore
\begin{align*}
  a'(h)
  &=2\cosh r_h\int_0^{2\pi}
       \frac{\dd\varphi}{\sqrt{1+\sinh^2r_h\sin^2\varphi}}\\
  &=2\int_0^{2\pi}
       \frac{\dd\varphi}{\sqrt{1-\tanh^2r_h\cos^2\varphi}}
   =8\mathbf K(\tanh r_h).
\end{align*}
Now \(\tanh r_h=\sqrt{1-\e^{-2h}}\), which proves the formula for
\(a'\), its positivity, and its smoothness on \((0,\infty)\).  Since
\(a(h)\to0\) as \(h\downarrow0\), integration gives the formula for
\(a\).  For later use, the coarea formula also gives
\[
  a'(h)=\int_{\partial K_h}\frac{1}{\norm{\nabla\psi(\theta)}_2}
                    \dd\mathcal H^1(\theta),
\]
where \(\mathcal H^1\) is one dimensional Hausdorff measure.  The level is
regular because \(\nabla\psi\) vanishes only at the origin.

As \(k\to0\), expansion of the defining integrand gives
\(\mathbf K(k)=\frac\pi2(1+k^2/4+O(k^4))\).  Since
\(k^2=1-\e^{-2h}=2h+O(h^2)\), substitution proves
\eqref{eq:ah-derivative-small}; integration proves \eqref{eq:ah-small}.

For large \(h\), the complementary modulus is
\(k':=\sqrt{1-k^2}=\e^{-h}\).  The standard expansion is
\[
  \mathbf K(k)=\log\frac4{k'}
       +O\!\left(k'^2\log\frac1{k'}\right),
  \qquad k'\downarrow0.
\]
It gives \(a'(h)=8(h+\log4)+O(h\e^{-2h})\).  The error is integrable
at infinity, so \(a(h)-4(h+\log4)^2\) has a finite limit.  This proves
\eqref{eq:ah-large}.
\end{proof}
\end{hyperreplacement}

\begin{analystadditioncheck}
\begin{lemma}
\label{lem:unique-Wulff-parameter}
For \(h>0\), set
\begin{equation}\label{eq:Wulff-parameter-map}
  \mathfrak m(h):=\frac{2a(h)}{a'(h)}-h.
\end{equation}
Then \(\mathfrak m\) is a strictly increasing bijection from
\((0,\infty)\) onto
\((0,\log4)\).  Consequently, for every \(B\in(0,\log4)\), the function
\begin{equation}\label{eq:Wulff-quotient}
  \Phi_B(h):=\frac{(B+h)^2}{4a(h)},\qquad h>0,
\end{equation}
has a unique minimizer \(h_B=\mathfrak m^{-1}(B)\), and
\begin{equation}\label{eq:Wulff-minimum-value}
  \Phi_B(h_B)=\frac{a(h_B)}{a'(h_B)^2}.
\end{equation}
\end{lemma}

\begin{proof}
\emph{Strict concavity of \(\sqrt a\).}
By \Cref{lem:ah-asymptotics}, \(a\) is smooth and satisfies
\(a(h)>0\) and \(a'(h)>0\) for \(h>0\).  Put
\(\mathfrak s(h):=\sqrt{a(h)}\).  For \(0<h_0<h_1\), \(0<\lambda<1\), and
\(h_\lambda=(1-\lambda)h_0+\lambda h_1\), convexity of \(\psi\) gives
\begin{equation}\label{eq:Kh-Minkowski-inclusion}
  (1-\lambda)K_{h_0}+\lambda K_{h_1}\subset K_{h_\lambda}.
\end{equation}
This inclusion is proper.  Indeed, strict convexity of \(\psi\) makes each
\(K_h\) strictly convex.  Fix a nonzero direction \(e\), and let
\(\theta_i\) be the unique point of \(K_{h_i}\) maximizing
\(\ip{e}{\theta}\), \(i=0,1\).  Each \(\theta_i\) lies on
\(\partial K_{h_i}\), and hence \(\psi(\theta_i)=h_i\).  In particular,
\(\theta_0\ne\theta_1\).  The point
\[
  \theta_\lambda:=(1-\lambda)\theta_0+\lambda\theta_1
\]
maximizes \(\ip{e}{\theta}\) over the set on the left of
\eqref{eq:Kh-Minkowski-inclusion}, whereas strict convexity gives
\[
  \psi(\theta_\lambda)
  <(1-\lambda)h_0+\lambda h_1=h_\lambda.
\]
Thus \(\theta_\lambda\) is an interior point of \(K_{h_\lambda}\), and
\[
  \ip{e}{\theta_\lambda}
  <\sup_{\theta\in K_{h_\lambda}}\ip{e}{\theta}.
\]
The set
\[
  \left\{\theta\in\operatorname{int}K_{h_\lambda}:
  \ip{e}{\theta}>\ip{e}{\theta_\lambda}\right\}
\]
is nonempty, open, and disjoint from the Minkowski sum in
\eqref{eq:Kh-Minkowski-inclusion}.  Therefore
\[
  \left|(1-\lambda)K_{h_0}+\lambda K_{h_1}\right|
  <|K_{h_\lambda}|=a(h_\lambda).
\]
The planar Brunn--Minkowski inequality now yields
\[
  \mathfrak s(h_\lambda)
  >\left|(1-\lambda)K_{h_0}+\lambda K_{h_1}\right|^{1/2}
  \ge(1-\lambda)\mathfrak s(h_0)+\lambda\mathfrak s(h_1).
\]
Hence \(\mathfrak s\) is strictly concave.

\emph{Range and monotonicity of the parameter map.}
Since
\(a=\mathfrak s^2\) and \(\mathfrak s'>0\),
\begin{equation}\label{eq:Wulff-map-derivative}
  \mathfrak m(h)=\frac{\mathfrak s(h)}{\mathfrak s'(h)}-h,
  \qquad
  \mathfrak m'(h)=-\frac{\mathfrak s(h)\mathfrak s''(h)}
              {\mathfrak s'(h)^2}\ge0.
\end{equation}
If \(\mathfrak m\) took the same value at two distinct points, it would be
constant between them.  Equation~\eqref{eq:Wulff-map-derivative} would then
make \(\mathfrak s\) affine on that interval, contrary to strict concavity.
Thus \(\mathfrak m\) is strictly increasing.

The expansions \eqref{eq:ah-small} and
\eqref{eq:ah-derivative-small} give
\[
  \mathfrak m(h)=h+O(h^2)\longrightarrow0
  \qquad(h\downarrow0).
\]
\begin{hyperreplacement}
For large \(h\), \Cref{lem:ah-asymptotics} gives both
\(a(h)=4(h+\log4)^2+O(1)\) and
\(a'(h)=8(h+\log4)+O(h\e^{-2h})\).  Substitution yields
\[
  \mathfrak m(h)=\frac{2a(h)}{a'(h)}-h
  =\log4+O(h^{-1})\longrightarrow\log4
  \qquad(h\to\infty).
\]
\end{hyperreplacement}
This proves the asserted bijection.

\emph{The unique minimizer.}
Differentiation of \eqref{eq:Wulff-quotient} gives
\[
  \Phi_B'(h)
  =\frac{(B+h)a'(h)}{4a(h)^2}\{\mathfrak m(h)-B\}.
\]
Thus \(\Phi_B\) decreases before \(h_B=\mathfrak m^{-1}(B)\) and increases after
\(h_B\), proving uniqueness.  At \(h=h_B\),
\(B+h_B=2a(h_B)/a'(h_B)\), which gives
\eqref{eq:Wulff-minimum-value}.
\end{proof}
\end{analystadditioncheck}

\begin{revisioncheck}
The preceding Wulff inequality and parameter analysis now give the exact
origin fiber.
\end{revisioncheck}

\begin{proof}[\revisionmark{Proof of \Cref{thm:origin-fiber-main}}]
\emph{Wulff upper bound.}
For \(v\in\RR^2\), Fenchel duality gives
\begin{equation}\label{eq:support-Fenchel}
  h_{K_h}(v)\le L(v)+h.
\end{equation}
Let \(\gamma\) be any closed curve with total action at most \(B\).
If necessary, reverse its orientation so that \(A(\gamma)\ge0\); the
action is unchanged because \(L\) is even.  Integrating
\eqref{eq:support-Fenchel} gives
\[
  \mathcal P_{K_h}(\gamma)
  =\int_0^1h_{K_h}(\dot\gamma(t))\dd t
  \le B+h.
\]
Since \(K_h\) is centrally symmetric, \Cref{lem:anisotropic-loop} gives
\[
  A(\gamma)\le\frac{(B+h)^2}{4a(h)}.
\]
Taking the infimum over \(h\), followed by
\Cref{lem:cyclic-rooting}, proves the upper bound in
\eqref{eq:origin-Wulff-main}.

By \Cref{lem:unique-Wulff-parameter}, the quotient has the unique
minimizer \(h_B\in(0,\infty)\), characterized by
\[
  B=\frac{2a(h_B)}{a'(h_B)}-h_B.
\]
For \(B=\log\rho\), this point is \(h_\rho^\star\) in
\eqref{eq:origin-critical-main}.

\emph{The matching Hamiltonian loop.}
The Hamiltonian flow
\begin{equation}\label{eq:Hamilton-level-flow}
  \dot\theta_h(t)=a'(h)\mathsf J\nabla\psi(\theta_h(t)),
  \qquad \theta_h(t)\in\partial K_h,
\end{equation}
has period one; indeed, its period is
\[
  \frac1{a'(h)}
  \int_{\partial K_h}\frac{1}{\norm{\nabla\psi}_2}
       \dd\mathcal H^1=1
\]
by the coarea formula.
\begin{hyperaddition}
In the Fermi coordinates above, this flow becomes
\[
  \dot p=\frac{a'(h)}2\tanh q,\qquad
  \dot q=-\frac{a'(h)}2\tanh p.
\]
Its image under \(X\) is the hyperbolic circle of radius \(r_h\).
The time variable is fixed by \eqref{eq:Hamilton-level-flow}; it is this
parametrization that gives the Legendre equality used below.
\end{hyperaddition}
Define
\[
  \gamma_h(t):=-\frac1{a'(h)}\mathsf J\theta_h(t).
\]
Then \(\dot\gamma_h=\nabla\psi(\theta_h)\), so Legendre equality and the
\begin{revisioncheck}
divergence theorem give the following identity, where \(\nu(\theta)\)
is the outward unit normal to \(K_h\):
\end{revisioncheck}
\begin{align*}
  \int_0^1L(\dot\gamma_h(t))\dd t
  &=\int_0^1\{\ip{\theta_h(t)}{\nabla\psi(\theta_h(t))}-h\}\dd t\\
  &=\frac1{a'(h)}
    \int_{\partial K_h}\ip{\theta}{\nu(\theta)}
       \dd\mathcal H^1(\theta)-h\\
  &=\frac{2a(h)}{a'(h)}-h.
\end{align*}
At \(h=h_B\), this action equals \(B\).  The curve bounds the rotated
dilation \(-\mathsf J K_{h_B}/a'(h_B)\), and hence has area
\(a(h_B)/a'(h_B)^2\).  By \eqref{eq:Wulff-minimum-value}, this is the
upper bound in \eqref{eq:origin-Wulff-main}.
\begin{revisioncheck}
\emph{Filling the origin fiber.}
Take \(B=\log\rho\), set \(\theta_\rho:=\theta_{h_B}\), and let
\(t_\rho\) be the cyclic root supplied by \Cref{lem:cyclic-rooting}.  The
translated rooted loop is
\[
  -\frac1{a'(h_B)}\mathsf J
  \bigl(\theta_\rho(t_\rho+t)-\theta_\rho(t_\rho)\bigr).
\]
The resulting cyclically rooted and translated Wulff loop belongs to \(\cC_\rho\)
and has the same area as \(\gamma_{h_B}\).
\end{revisioncheck}
\begin{analystadditioncheck}
Coordinate reflection changes the area sign while preserving the horizontal
origin.  The symmetry in \Cref{prop:sym-star}, together with the compact
interval property in \Cref{prop:fibers}, therefore gives
\eqref{eq:origin-fiber-interval-main}.
\end{analystadditioncheck}
\end{proof}

\begin{analystadditioncheck}
\subsubsection{Fiber bounds and universal nonconvexity}

\begin{revisioncheck}
The exact origin fiber controls all other fibers by rotation, dilation, and
concatenation.  We first derive quantitative inner and outer bounds and then
use the Wulff optimizer to construct the nonconvexity witness.
\end{revisioncheck}

For \(p\in\cD_\rho\), define the Minkowski functional of the horizontal
projection by
\begin{equation}\label{eq:fiber-gauge}
  \tau_\rho(p):=\inf\{r>0:p\in r\cD_\rho\}.
\end{equation}
By \Cref{prop:fibers}, \(0\in\operatorname{int}(\cD_\rho)\).  Hence
\(0\le\tau_\rho(p)\le1\) for \(p\in\cD_\rho\), and \(\tau_\rho(0)=0\).

\begin{proposition}[Bounds for fibers below saturation]
\label{prop:fiber-bounds-below-saturation}
For every \(1<\rho<4\), every \(p=(a,b)\in\cD_\rho\), and every
\(c\in\cI_\rho(a,b)\),
\begin{equation}\label{eq:fiber-outer-below-saturation}
  \left|c-\frac{ab}{2}\right|
  \le
  \min\left\{2Z_0(\rho),
  4Z_0(\rho)-\frac{a^2+b^2}{4}\right\}.
\end{equation}
Moreover,
\begin{equation}\label{eq:fiber-inner-below-saturation}
  \left[
  \frac{ab}{2}-(1-\tau_\rho(p))^2Z_0(\rho),
  \frac{ab}{2}+(1-\tau_\rho(p))^2Z_0(\rho)
  \right]
  \subset\cI_\rho(a,b).
\end{equation}
\end{proposition}

\begin{proof}
\emph{Outer bounds by symmetric concatenation.}
Fix \(\omega=(x,y)\in\cC_\rho\) with endpoint \((a,b,c)\), and put
\(p=(a,b)\).  Then \(c=A(\omega)\).  If a path
\(\alpha=(\alpha_1,\alpha_2)\) starts at zero and ends at
\((\bar a,\bar b)\), define
\[
  A^\circ(\alpha):=\frac12\int
  (\alpha_1\dot\alpha_2-\alpha_2\dot\alpha_1)\dd t
  =A(\alpha)-\frac{\bar a\bar b}{2},
  \qquad
  \mathfrak a_\omega:=A^\circ(\omega)=c-\frac{ab}{2}.
\]
For paths \(\beta_1,\ldots,\beta_m\), rooted at zero and with
endpoints \(p_1,\ldots,p_m\in\RR^2\), let
\(\beta_1*\cdots*\beta_m\) denote their time ordered concatenation, with
each piece translated to begin where the preceding piece ends.  Direct
splitting of the line integral gives
\begin{equation}\label{eq:centered-area-concatenation}
  A^\circ(\beta_1*\cdots*\beta_m)
  =\sum_{j=1}^m A^\circ(\beta_j)
   +\frac12\sum_{1\le i<j\le m}
     \det(p_i,p_j).
\end{equation}

Let \(\mathsf I\) be the identity map on \(\RR^2\).  For
\(G\in\{\mathsf I,\mathsf J,-\mathsf I,-\mathsf J\}\) and
\(0<\lambda\le1\), set
\[
  \omega^{G,\lambda}(r):=\lambda G\omega(r/\lambda),
  \qquad 0\le r\le\lambda.
\]
The symmetries of \(L\) give, for \(0\le r\le\lambda\),
\begin{align}
  \int_0^r L(\dot\omega^{G,\lambda}(s))\dd s
  &=\lambda J_{r/\lambda}(\omega)\le rB,
    \label{eq:scaled-piece-action}\\
  \omega^{G,\lambda}(\lambda)&=\lambda Gp,
  &A^\circ(\omega^{G,\lambda})&=\lambda^2\mathfrak a_\omega.
    \label{eq:scaled-piece-data}
\end{align}
Thus a concatenation of equal length pieces remains in \(\cC_\rho\):
after \(k\) complete pieces, its action at time \(k\lambda+r\) is at most
\(B(k\lambda+r)\).

\revisionmark{The two pieces with scale \(1/2\) and
\(G=\mathsf I,-\mathsf I\) form a closed path of area
\(\mathfrak a_\omega/2\).}  The exact origin fiber gives
\begin{equation}\label{eq:fiber-first-outer-proof}
  |\mathfrak a_\omega|\le2Z_0(\rho).
\end{equation}

Define the two closed concatenations
\begin{align*}
  \sigma_+
  &:={}\omega^{\mathsf I,1/4}*\omega^{\mathsf J,1/4}
       *\omega^{-\mathsf I,1/4}*\omega^{-\mathsf J,1/4},\\
  \sigma_-
  &:={}\omega^{\mathsf I,1/4}*\omega^{-\mathsf J,1/4}
       *\omega^{-\mathsf I,1/4}*\omega^{\mathsf J,1/4}.
\end{align*}
Their respective piece displacements are
\[
  \frac14(p,\mathsf Jp,-p,-\mathsf Jp)
  \quad\text{and}\quad
  \frac14(p,-\mathsf Jp,-p,\mathsf Jp).
\]
Substitution in \eqref{eq:centered-area-concatenation} gives
\[
  A(\sigma_\pm)
  =\frac{\mathfrak a_\omega}{4}
   \mathbin{\pm}\frac{a^2+b^2}{16}.
\]
Both absolute values are at most \(Z_0(\rho)\), and hence
\[
  \frac{|\mathfrak a_\omega|}{4}
  +\frac{a^2+b^2}{16}\le Z_0(\rho).
\]
Together with \eqref{eq:fiber-first-outer-proof}, this proves
\eqref{eq:fiber-outer-below-saturation}.

\emph{Inner bound by a terminal loop.}
If \(p=0\), the maximizing loop from
\Cref{thm:origin-fiber-main}, its reflection in one coordinate, and
\Cref{prop:fibers} give \([-Z_0(\rho),Z_0(\rho)]\).

Suppose that \(p\ne0\), put \(\tau:=\tau_\rho(p)\in(0,1]\), and set
\(s:=1-\tau\).  Since \(\cD_\rho\) is compact and convex and contains zero
in its interior, the infimum in \eqref{eq:fiber-gauge} is attained.  Hence
\(p/\tau\in\cD_\rho\), so
\(L(p/\tau)\le B\).  Let \(\gamma_\rho^+\in\cC_\rho\) be a maximizing closed
origin loop, and let
\(\gamma_\rho^-:=(\gamma_{\rho,1}^+,-\gamma_{\rho,2}^+)\).  Coordinate reflection
preserves the action inequalities and gives
\(A(\gamma_\rho^\pm)=\pm Z_0(\rho)\).  For \(s>0\), define
\[
  \omega_\pm^{\mathrm{fib}}(t):=
  \begin{cases}
    (t/\tau)p,&0\le t\le\tau,\\
    p+s\gamma_\rho^\pm((t-\tau)/s),&\tau\le t\le1.
  \end{cases}
\]
When \(s=0\), retain only the first line.  If \(s>0\), then for
\(0\le r\le s\),
\[
  J_{\tau+r}(\omega_\pm^{\mathrm{fib}})
  =\tau L(p/\tau)+sJ_{r/s}(\gamma_\rho^\pm)
  \le B(\tau+r),
\]
and before time \(\tau\) the bound follows from \(L(p/\tau)\le B\).
Thus \(\omega_\pm^{\mathrm{fib}}\in\cC_\rho\), and
\[
  A(\omega_\pm^{\mathrm{fib}})
  =\frac{ab}{2}\mathbin{\pm}s^2Z_0(\rho).
\]
The interval property in \Cref{prop:fibers} now proves
\eqref{eq:fiber-inner-below-saturation}.
\end{proof}
\end{analystadditioncheck}

\begin{analystadditioncheck}
\begin{proof}[Proof of \Cref{thm:geometry-summary}]
Item~\textup{(i)} follows from \Cref{prop:saturation} and
\Cref{prop:complete-saturated-shape}.  Item~\textup{(ii)} is
\Cref{cor:saturation-sandwich} with \(\rho_1=\rho\) and \(\rho_2=4\).
\end{proof}
\end{analystadditioncheck}

\begin{revisioncheck}
We conclude with the nonconvexity argument, first below saturation and then at
and above the threshold.
\end{revisioncheck}

\begin{proof}[Proof of \Cref{thm:nonconvex-main}]
\emph{Below saturation.}
\revisionmark{Let \(1<\rho<4\) and let \(h_B\) be the unique Wulff
parameter from \Cref{lem:unique-Wulff-parameter}.}  Let
\(\theta_{h_B}\) be a period one solution of
\eqref{eq:Hamilton-level-flow}, extend it periodically, and define the
unrooted maximizing loop
\[
  \gamma_0(t):=-\frac1{a'(h_B)}\mathsf J\theta_{h_B}(t),
  \qquad t\in\RR.
\]
The construction in the proof of \Cref{thm:origin-fiber-main} gives
\[
  \int_0^1L(\dot\gamma_0(t))\dd t=B,
  \qquad
  A(\gamma_0)=Z_0(\rho).
\]
Since
\(\psi(\mathsf J\theta)=\psi(\theta)\) and
\(\nabla\psi(\mathsf J\theta)=\mathsf J\nabla\psi(\theta)\), the Hamiltonian
vector field
\(\mathsf V(\theta):=a'(h_B)\mathsf J\nabla\psi(\theta)\) in
\eqref{eq:Hamilton-level-flow} satisfies
\(\mathsf V(\mathsf J\theta)=\mathsf J\mathsf V(\theta)\).  Hence a quarter
turn maps each orbit segment between consecutive coordinate axes to the
next segment and preserves its traversal time.  These four times sum to the
period one, so each equals \(1/4\).  Choose the phase at an axis crossing.
Then
\(\theta_{h_B}(1/4)=\mathsf J\theta_{h_B}(0)\), and uniqueness of the
Hamiltonian flow gives
\[
  \theta_{h_B}(t+1/4)=\mathsf J\theta_{h_B}(t).
\]
Thus the action density
\(
\ell_0(t):=L(\dot\gamma_0(t))
\)
has period \(1/4\), and every quarter has action \(B/4\).  The function
\[
  \mathfrak g(t):=\int_0^t\{\ell_0(s)-B\}\dd s
\]
is consequently quarter periodic.  Choose \(t_0\) at which
\(\mathfrak g\) is maximal.
\revisionmark{Rooting the loop at any of the four times \(t_0+k/4\),
\(0\le k\le3\), produces a shifted path satisfying
the action constraints.}
\begin{revisioncheck}
For \(k\in\{0,1,2,3\}\), let \(\chi^{(k)}\) be the chord from the point at
time \(t_0+k/4\) to its antipode.  Then
\(\chi^{(k+1)}=\mathsf J\chi^{(k)}\), with the superscript read modulo four,
so successive products of the two coordinates have opposite signs.  Choose
\(k_0\) so that \(\chi:=\chi^{(k_0)}=(\chi_1,\chi_2)\) satisfies
\begin{equation}\label{eq:chord-positive-product}
  \chi_1\chi_2\ge0.
\end{equation}
Set \(s_0:=t_0+k_0/4\).

Translate this root to the origin by setting
\[
  \gamma(t):=\gamma_0(s_0+t)-\gamma_0(s_0),
  \qquad 0\le t\le1.
\]
The maximality of \(\mathfrak g(s_0)\), as in
\Cref{lem:cyclic-rooting}, gives
\(J_t(\gamma)\le tB\) for every \(t\in[0,1]\).
Central symmetry gives
\[
  \chi=\gamma(1/2),
  \qquad
  \gamma(t+1/2)=\chi-\gamma(t)
  \quad(0\le t\le1/2).
\]
Define the line integrals over the two halves by
\[
  A_1:=\int_0^{1/2}\gamma_1(t)\dot\gamma_2(t)\dd t,
  \qquad
  A_2:=\int_{1/2}^1\gamma_1(t)\dot\gamma_2(t)\dd t.
\]
Changing variables in the second integral gives
\[
  A_2=A_1-\chi_1\chi_2,
  \qquad
  \revisionmark{Z_0(\rho)}=2A_1-\chi_1\chi_2.
\]
Define the open path
\[
  \omega^\sharp(t):=2\gamma(t/2),\qquad0\le t\le1.
\]
Its velocity is \(\dot\gamma(t/2)\), and hence
\[
  J_t(\omega^\sharp)=2J_{t/2}(\gamma)\le Bt.
\]
Moreover,
\begin{equation}\label{eq:sharp-area-above-origin}
  A(\omega^\sharp)=4A_1
  =2\revisionmark{Z_0(\rho)}+2\chi_1\chi_2
  >\revisionmark{Z_0(\rho)}.
\end{equation}
Here \revisionmark{\(Z_0(\rho)>0\)}, for example by a positively oriented rectangular loop
scaled so that its velocity lies in a neighborhood of zero on which
\(L\le B\), and \eqref{eq:chord-positive-product} gives the strict inequality.
Both \(\omega^\sharp\) and \(-\omega^\sharp\) are admissible and have
the same area.  Their endpoints
\[
  (2\chi_1,2\chi_2,A(\omega^\sharp))
  \quad\text{and}\quad
  (-2\chi_1,-2\chi_2,A(\omega^\sharp))
\]
belong to \revisionmark{\(\cR_\rho\)}, but their midpoint lies strictly
above the origin fiber by \eqref{eq:sharp-area-above-origin}.  Thus
\revisionmark{\(\cR_\rho\)} is nonconvex.
\end{revisioncheck}

\begin{analystadditioncheck}
\emph{At and above saturation.}
At \(\rho=4\), \eqref{eq:diagonal-fiber} gives the two points
\[
  P_0:=\left(0,0,\frac1{16}\right),
  \qquad
  P_1:=\left(\frac14,\frac14,\frac9{64}\right)
\]
belong to \(\cR_4\).  Their Euclidean midpoint is
\[
  M:=\left(\frac18,\frac18,\frac{13}{128}\right).
\]
But
\[
  \max\cI_4(1/8,1/8)=\frac{25}{256}
  <\frac{26}{256}=\frac{13}{128},
\]
so \(M\notin\cR_4\).  Saturation gives the same conclusion for every
\(\rho\ge4\), completing the proof.
\end{analystadditioncheck}
\end{proof}

\end{analystcheck}

\section{Variational analysis of the central front}
\label{sec:central-front}

\begin{analystcheck}

\begin{revisioncheck}
This section develops the deterministic analysis of \(\Gamma_\rho\).  We
first establish global concavity, then derive the sharp weak branching
expansion and the rigidity of maximizing paths.  Along the way we evaluate
the total action relaxation exactly, and we conclude with the approach to
saturation from below.
\end{revisioncheck}

\subsection{\texorpdfstring{\revisionmark{Global concavity}}
{Global concavity}}

\begin{revisioncheck}
We prove \Cref{thm:Gamma-global} in two stages.  On each finite grid, a
calibration in square variables turns the first order optimality condition
into global maximality and yields an affine upper support for the finite
grid value.  Concavity then passes to continuous time through dyadic
approximation.  Using the extension \(\Gamma_1=0\) in
\Cref{thm:Gamma-global}, for \(B\ge0\) set
\end{revisioncheck}
\[
  \mathfrak G(B):=\Gamma_{\e^B}.
\]
For \(B>0\), the definitions of \(\cC_{\e^B}\) and \(\mathfrak A\) give
\begin{equation}\label{eq:G-control-form}
  \mathfrak G(B)=\sup\left\{
  \mathfrak A(v):
  \begin{array}{l}
  v\in L^\infty([0,1];D),\\
  \displaystyle\int_0^tL(v(s))\dd s\le Bt
  \quad\text{for every }t\in[0,1]
  \end{array}
  \right\}.
\end{equation}

\subsubsection{\texorpdfstring{\revisionmark{Hidden convexity on a finite
grid}}{Hidden convexity on a finite grid}}

Fix an integer \(m\ge1\), put \(\Delta:=1/m\), and fix
\(\alpha\in(0,1)\).  For \(1\le i\le m\), let
\(\mathsf C_i:=((i-1)\Delta,i\Delta]\), with zero included in
\(\mathsf C_1\), and identify
\(w=(w_1,\ldots,w_m)\), where
\(w_i=(w_i^{(1)},w_i^{(2)})\), with the piecewise constant control equal to
\(w_i\) on \(\mathsf C_i\).  Set
\[
  r_i:=w_i^{(1)}+w_i^{(2)},
  \qquad s_i:=w_i^{(1)}-w_i^{(2)}.
\]
Write \(w^{(1)}(t)\) and \(w^{(2)}(t)\) for the two component functions of
this control.
The diagonal velocity box, viewed in \((r,s)\) coordinates, is
\[
  \mathfrak B_{m,\alpha}
  :=\left\{w:\abs{r_i}\le\alpha,\ \abs{s_i}\le\alpha,
  \ 1\le i\le m\right\}.
\]
When writing membership in this box or its normal cone, we identify \(w\)
with the diagonal coordinate vector
\((r_1,s_1,\ldots,r_m,s_m)\).
For \(x_i:=\Delta\sum_{j=1}^iw_j^{(1)}\), with \(x_0:=0\), the exact area of
the polygonal path is
\begin{equation}\label{eq:grid-area}
  \mathfrak A_m(w)
  :=\Delta\sum_{i=1}^m
  \left(x_{i-1}+\frac{\Delta}{2}w_i^{(1)}\right)w_i^{(2)}.
\end{equation}
Let \(\mathfrak G_{m,\alpha}(B)\) be the maximum of
\(\mathfrak A_m\) over \(\mathfrak B_{m,\alpha}\) subject to
\begin{equation}\label{eq:grid-action-constraints}
  \Delta\sum_{i=1}^kL(w_i)\le Bk\Delta,
  \qquad 1\le k\le m.
\end{equation}

For \(\vartheta\ge0\), define the extended convex function
\[
  f_{\vartheta,\alpha}(u)
  :=\begin{cases}
  \vartheta\phi(u),&\abs{u}\le\alpha,\\
  +\infty,&\abs{u}>\alpha.
  \end{cases}
\]
Its convex conjugate is
\begin{equation}\label{eq:grid-conjugate}
  \Pi_\alpha(\vartheta,\ell)
  :=f_{\vartheta,\alpha}^*(\ell)
  =\sup_{\abs{u}\le\alpha}\{\ell u-\vartheta\phi(u)\}.
\end{equation}
For a closed convex set \(C\) and \(z\in C\), we use the normal cone
convention
\[
  \mathrm N_C(z):=\{\eta:\ip{\eta}{y-z}\le0\text{ for every }y\in C\}.
\]

The next lemma is the key finite dimensional step: it promotes a first order
condition to a global conclusion.

\begin{lemma}
\label{lem:grid-calibration}
Let \((r_1,s_1,\ldots,r_m,s_m)\) be the full coordinate vector, and write
\(\nabla_{r,s}\) for its gradient.
Let \(\vartheta_i\ge0\), \(1\le i\le m\), and suppose that
\(w^*\in\mathfrak B_{m,\alpha}\) has nonnegative Cartesian components,
\(\mathfrak A_m(w^*)>0\), and
\begin{equation}\label{eq:grid-normal-stationarity}
  \left.\nabla_{r,s}\left\{
  \mathfrak A_m(w)-\Delta\sum_{i=1}^m\vartheta_iL(w_i)
  \right\}\right|_{w=w^*}
  \in\mathrm N_{\mathfrak B_{m,\alpha}}(w^*).
\end{equation}
Then \(w^*\) is a global maximizer on \(\mathfrak B_{m,\alpha}\) of the
functional in braces.
\end{lemma}

\begin{proof}
\emph{Reduction to nonnegative controls.}
We first reduce arbitrary competitors to controls with nonnegative Cartesian
components.  For a piecewise constant control \(w\), put
\(\widehat w_i:=(\abs{w_i^{(1)}},\abs{w_i^{(2)}})\) in every time cell.
The symmetries of \(L\) preserve both its value and the diagonal box.  Define
\[
  x(t):=\int_0^tw^{(1)}(u)\dd u,
  \qquad
  \widehat x(t):=\int_0^t\abs{w^{(1)}(u)}\dd u.
\]
Then
\[
  x(t)w^{(2)}(t)
  \le\abs{x(t)}\abs{w^{(2)}(t)}
  \le\widehat x(t)\abs{w^{(2)}(t)}.
\]
Thus \(\mathfrak A_m(\widehat w)\ge\mathfrak A_m(w)\), while the penalizing
term is unchanged.  It is enough to consider nonnegative controls.

\emph{The calibration functional.}
For such a control, let
\[
  p(t):=\int_t^1w^{(2)}(u)\dd u
\]
denote the remaining vertical displacement, and set
\[
  P:=\frac{p+x}{2},\qquad Q:=\frac{p-x}{2}.
\]
They satisfy
\begin{equation}\label{eq:grid-PQ-relations}
  \begin{aligned}
  Q(0)&=P(0),\qquad Q(1)=-P(1),\\
  s_i&=2\dot P(t),\qquad r_i=-2\dot Q(t),
  \qquad\text{for almost every }t\in\mathsf C_i.
  \end{aligned}
\end{equation}
Moreover, Fubini's theorem and integration by parts give
\begin{equation}\label{eq:grid-PQ-area}
  \mathfrak A_m(w)
  =\int_0^1p(t)w^{(1)}(t)\dd t
  =\int_0^1x(t)w^{(2)}(t)\dd t
  =\int_0^1\{Q\dot P-P\dot Q\}\dd t.
\end{equation}
Write
\[
  P_i:=P(i\Delta),\quad Q_i:=Q(i\Delta),\quad
  \overline P_i:=\frac{P_{i-1}+P_i}{2},\quad
  \overline Q_i:=\frac{Q_{i-1}+Q_i}{2},
\]
and let \(D_\Delta P_i:=(P_i-P_{i-1})/\Delta\), with the analogous
definition for \(Q\).  Direct differentiation of \eqref{eq:grid-area} gives
\begin{equation}\label{eq:grid-area-gradients}
  \frac{\partial\mathfrak A_m}{\partial r_i}=\Delta\overline P_i,
  \qquad
  \frac{\partial\mathfrak A_m}{\partial s_i}=\Delta\overline Q_i.
\end{equation}

Using \eqref{eq:grid-PQ-area}, integrating the term
\(-Q\dot P\) by parts and then using the boundary identities in
\eqref{eq:grid-PQ-relations}, we obtain
\begin{equation}\label{eq:grid-penalty-identity}
  \Delta\sum_{i=1}^m\vartheta_iL(w_i)-\mathfrak A_m(w)
  =\Delta\sum_{i=1}^m
  \left\{
  f_{\vartheta_i,\alpha}(2D_\Delta P_i)
  +f_{\vartheta_i,\alpha}(r_i)-\overline P_i r_i
  \right\}
  +P_0^2+P_m^2.
\end{equation}
Fenchel's inequality applied to the \(r_i\) terms gives
\begin{equation}\label{eq:grid-calibration-lower}
  \Delta\sum_{i=1}^m\vartheta_iL(w_i)-\mathfrak A_m(w)
  \ge\mathcal J_\vartheta(P),
\end{equation}
where
\begin{equation}\label{eq:grid-J-functional}
  \mathcal J_\vartheta(P)
  :=\Delta\sum_{i=1}^m
  \left\{
  f_{\vartheta_i,\alpha}(2D_\Delta P_i)
  -\Pi_\alpha(\vartheta_i,\overline P_i)
  \right\}
  +P_0^2+P_m^2.
\end{equation}

\emph{Convexity in the square variables.}
For the nonnegative controls under consideration, \(x,p\), and hence every
\(P_i\), are nonnegative.  We show that \(\mathcal J_\vartheta\) is convex
in the nodal variables \(U_i:=P_i^2\).  For two nonnegative nodal vectors
\(P^0,P^1\), set
\[
  P_i^\lambda
  :=\sqrt{(1-\lambda)(P_i^0)^2+\lambda(P_i^1)^2},
  \qquad 0\le\lambda\le1.
\]
The reverse triangle inequality in \(\RR^2\) gives
\begin{equation}\label{eq:grid-square-convexity}
  \abs{P_i^\lambda-P_{i-1}^\lambda}^2
  \le(1-\lambda)\abs{P_i^0-P_{i-1}^0}^2
    +\lambda\abs{P_i^1-P_{i-1}^1}^2.
\end{equation}
The effective domain of the \(i\)th kinetic term in the square variables is
\[
  \left\{(U_{i-1},U_i):
  \abs{\sqrt{U_i}-\sqrt{U_{i-1}}}^2
  \le(\alpha\Delta/2)^2\right\},
\]
which is convex by \eqref{eq:grid-square-convexity}.  For
\(\vartheta_i>0\), the series \eqref{eq:phi-series} shows that
\[
  z\longmapsto
  \vartheta_i\phi\left(\frac{2\sqrt z}{\Delta}\right)
\]
is increasing and convex.  When \(\vartheta_i=0\), its finite part vanishes.
Combining monotonicity and convexity with
\eqref{eq:grid-square-convexity} proves convexity of every extended kinetic
term in \((U_{i-1},U_i)\).

For the potential, \eqref{eq:grid-conjugate} and scalar Fenchel duality give,
for \(z\ge0\),
\begin{equation}\label{eq:grid-Pi-square}
  \Pi_\alpha(\vartheta,\sqrt z)
  =\begin{cases}
  \displaystyle
  \vartheta\log\cosh\left(\frac{\sqrt z}{\vartheta}\right),
    &\vartheta>0,\quad
     \sqrt z\le\vartheta\operatorname{arctanh}\alpha,\\[2mm]
  \alpha\sqrt z-\vartheta\phi(\alpha),
    &\vartheta>0,\quad
     \sqrt z\ge\vartheta\operatorname{arctanh}\alpha,\\[1mm]
  \alpha\sqrt z,&\vartheta=0.
  \end{cases}
\end{equation}
This function is increasing and concave.  Indeed, on the first branch its
derivative is
\[
  \frac{1}{2\vartheta}
  \frac{\tanh(\sqrt z/\vartheta)}{\sqrt z/\vartheta},
\]
and \(u\mapsto\tanh(u)/u\) is decreasing on \((0,\infty)\), because
\(\sinh u\cosh u>u\).  At zero the quotient is understood by continuity.
The second branch is concave, and
\revisionmark{its derivative agrees with that of the first branch at the
junction}.  The assertion for \(\vartheta=0\) is
immediate.  Finally,
\[
  (U,V)\longmapsto
  \left(\frac{\sqrt U+\sqrt V}{2}\right)^2
  =\frac{U+V+2\sqrt{UV}}{4}
\]
is concave.  Composition with the increasing concave function in
\eqref{eq:grid-Pi-square} shows that the negative potential in
\eqref{eq:grid-J-functional} is convex in \((U_{i-1},U_i)\).  Hence
\begin{equation}\label{eq:grid-hidden-convexity}
  \widehat{\mathcal J}_\vartheta(U)
  :=\mathcal J_\vartheta(\sqrt{U_0},\ldots,\sqrt{U_m})
\end{equation}
is an extended convex function on \([0,\infty)^{m+1}\).

\emph{Stationarity and global minimality.}
Let \(x^*,p^*,P^*,Q^*\) be the functions associated with \(w^*\), and set
\(U_i^*:=(P_i^*)^2\).  In the next display, \(\partial\) denotes the convex
subdifferential, and \(\partial_\ell\) is taken in the second variable.  The
normal cone condition
\eqref{eq:grid-normal-stationarity}, together with
\eqref{eq:grid-area-gradients} and conjugacy, is equivalent to
\begin{equation}\label{eq:grid-Fenchel-relations}
  r_i^*=-2D_\Delta Q_i^*
    \in\partial_\ell\Pi_\alpha(\vartheta_i,\overline P_i^*),
  \qquad
  \overline Q_i^*\in
    \partial f_{\vartheta_i,\alpha}(2D_\Delta P_i^*).
\end{equation}
The first relation gives equality in
\eqref{eq:grid-calibration-lower} at \(w^*\).

We also have \(P_i^*>0\) at every node.  Indeed,
\(2P_i^*=x^*(i\Delta)+p^*(i\Delta)\).  If this vanished, nonnegativity
would force the first component of \(w^*\) to vanish before that node and
the second component to vanish after it.  The area integrand would then
vanish on both sides of the node, contrary to
\(\mathfrak A_m(w^*)>0\).

Let \(U\) be any point in the effective domain of
\(\widehat{\mathcal J}_\vartheta\), set \(P_i:=\sqrt{U_i}\), and put
\[
  U^\epsilon:=(1-\epsilon)U^*+\epsilon U,
  \qquad
  P_i^\epsilon:=\sqrt{U_i^\epsilon},
  \qquad
  \delta P_i:=\left.\frac{\dd}{\dd\epsilon}P_i^\epsilon
  \right|_{\epsilon=0+}
  =\frac{U_i-U_i^*}{2P_i^*}.
\]
The second relation in \eqref{eq:grid-Fenchel-relations} shows that the right
directional derivative of the \(i\)th kinetic term is at least
\begin{equation}\label{eq:grid-kinetic-direction}
  2\overline Q_i^*(\delta P_i-\delta P_{i-1}).
\end{equation}
Since \(\overline P_i^*>0\), the maximizer in
\eqref{eq:grid-conjugate} is unique, also when \(\vartheta_i=0\).
Thus the derivative of \(\Pi_\alpha\) with respect to its second variable is
\(r_i^*=-2D_\Delta Q_i^*\), and the derivative of the negative potential is
\begin{equation}\label{eq:grid-potential-direction}
  (Q_i^*-Q_{i-1}^*)(\delta P_i+\delta P_{i-1}).
\end{equation}
The sum of \eqref{eq:grid-kinetic-direction} and
\eqref{eq:grid-potential-direction} equals
\[
  2Q_i^*\delta P_i-2Q_{i-1}^*\delta P_{i-1}.
\]
Summation over \(i\) telescopes.  The derivative of the endpoint term in
\eqref{eq:grid-J-functional} cancels the two remaining terms by
\eqref{eq:grid-PQ-relations}.  Therefore
\[
  \left.\frac{\dd}{\dd\epsilon}
  \widehat{\mathcal J}_\vartheta(U^\epsilon)
  \right|_{\epsilon=0+}\ge0.
\]
Convexity now gives
\(\mathcal J_\vartheta(P)\ge\mathcal J_\vartheta(P^*)\).  For every
nonnegative competitor \(w\), we consequently have
\[
  \Delta\sum_{i=1}^m\vartheta_iL(w_i)-\mathfrak A_m(w)
  \ge\mathcal J_\vartheta(P)
  \ge\mathcal J_\vartheta(P^*)
  =\Delta\sum_{i=1}^m\vartheta_iL(w_i^*)-\mathfrak A_m(w^*).
\]
Here the first inequality is \eqref{eq:grid-calibration-lower}, and the last
equality is Fenchel equality at \(w^*\).  The initial reduction to
nonnegative controls proves the global maximality claimed in the lemma.
\end{proof}

\begin{lemma}
\label{lem:grid-concavity}
For every \(m\ge1\) and \(\alpha\in(0,1)\), the function
\(B\mapsto\mathfrak G_{m,\alpha}(B)\) is concave on \((0,\infty)\).
\end{lemma}

\begin{proof}
The feasible set is compact and convex, so a maximizer exists.  Since
the zero control lies in the interior of \(\mathfrak B_{m,\alpha}\) and
satisfies every inequality in \eqref{eq:grid-action-constraints} strictly
when \(B>0\), the Slater condition holds.  Taking absolute values of both
Cartesian components preserves feasibility and cannot decrease the area, so
we may choose a global maximizer \(w^*\) with nonnegative components.  A
sufficiently small constant control with both components positive is feasible
and has positive area; hence \(\mathfrak A_m(w^*)>0\).

For any feasible \(\widetilde w\), the segment from \(w^*\) to
\(\widetilde w\) is feasible.  Differentiating maximality along this segment
shows that \(\nabla_{r,s}\mathfrak A_m(w^*)\) belongs to the normal cone of
the feasible set in diagonal coordinates.
The finite dimensional Slater normal cone rule
\cite{Rockafellar1970} therefore supplies multipliers
\(\lambda_k\ge0\), \(1\le k\le m\), such that, with
\begin{equation}\label{eq:grid-tail-weights}
  \vartheta_i:=\sum_{k=i}^m\lambda_k,
  \qquad
  \beta:=\Delta\sum_{k=1}^mk\lambda_k,
\end{equation}
the normal cone condition \eqref{eq:grid-normal-stationarity} holds and
\begin{equation}\label{eq:grid-complementarity}
  \lambda_k\left\{
  \Delta\sum_{i=1}^kL(w_i^*)-Bk\Delta
  \right\}=0,
  \qquad 1\le k\le m.
\end{equation}
By \Cref{lem:grid-calibration}, \(w^*\) globally maximizes the corresponding
penalized functional on \(\mathfrak B_{m,\alpha}\).

If \(\widetilde w\) is feasible at an action level \(B'>0\), finite
summation gives
\begin{equation}\label{eq:grid-multiplier-pairing}
  \Delta\sum_{i=1}^m\vartheta_iL(\widetilde w_i)
  =\sum_{k=1}^m\lambda_k
    \left(\Delta\sum_{i=1}^kL(\widetilde w_i)\right)
  \le B'\beta.
\end{equation}
By \eqref{eq:grid-complementarity}, equality holds with \(B\beta\) at
\(w^*\).  The global maximality supplied by
\Cref{lem:grid-calibration} now yields
\begin{equation}\label{eq:grid-affine-support}
  \mathfrak G_{m,\alpha}(B')
  \le\mathfrak G_{m,\alpha}(B)+\beta(B'-B),
  \qquad B'>0.
\end{equation}
To see concavity directly, fix \(B_0,B_1>0\) and
\(\theta\in[0,1]\), and apply \eqref{eq:grid-affine-support} at
\(B=(1-\theta)B_0+\theta B_1\), first with \(B'=B_0\) and then with
\(B'=B_1\).  The weighted slope terms cancel, giving
\[
  \mathfrak G_{m,\alpha}((1-\theta)B_0+\theta B_1)
  \ge(1-\theta)\mathfrak G_{m,\alpha}(B_0)
    +\theta\mathfrak G_{m,\alpha}(B_1).
\]
\end{proof}

\subsubsection{\texorpdfstring{\revisionmark{Passage to continuous time}}
{Passage to continuous time}}

\begin{proof}[Proof of \Cref{thm:Gamma-global}]
Fix \(\alpha\in(0,1)\) and use the dyadic meshes \(m=2^n\).  Repeating
each cell velocity on its two child cells preserves the polygonal path and
its exact area.  On each original cell, the difference between \(Bt\) and
the accumulated action is affine and nonnegative at both endpoints.  The
repeated control therefore satisfies every constraint on the refined grid,
so
\[
  \mathfrak G_{2^n,\alpha}(B)
  \le\mathfrak G_{2^{n+1},\alpha}(B).
\]
\revisionmark{The finite increasing limit}
\[
  \mathfrak G_\alpha(B)
  :=\lim_{n\to\infty}\mathfrak G_{2^n,\alpha}(B)
\]
is concave by \Cref{lem:grid-concavity}.  Choose any increasing sequence
\(\alpha_j\uparrow1\).  Since the velocity boxes are nested,
\revisionmark{the finite increasing limit}
\[
  \mathfrak G_*(B):=\lim_{j\to\infty}\mathfrak G_{\alpha_j}(B)
\]
is also concave and is independent of the chosen sequence.

\begin{revisioncheck}
Every control admissible for \(\mathfrak G_{m,\alpha}(B)\) is feasible in
\eqref{eq:G-control-form}, because the accumulated action difference is
affine on each cell.  Hence \(\mathfrak G_*(B)\le\mathfrak G(B)\).
\end{revisioncheck}

For the reverse inequality, let \(v=(f,g)\in L^\infty([0,1];D)\) be
feasible in \eqref{eq:G-control-form}.  For
\(1\le k\le2^n\), set
\(\mathsf C_{n,k}:=((k-1)2^{-n},k2^{-n}]\), again including zero in the
first cell.  Define the dyadic cell average
\(v^{(n)}=(f^{(n)},g^{(n)})\) by
\[
  v^{(n)}(t)
  :=\frac1{\abs{\mathsf C_{n,k}}}
    \int_{\mathsf C_{n,k}}v(s)\dd s,
  \qquad t\in\mathsf C_{n,k}.
\]
Jensen's inequality gives
\eqref{eq:grid-action-constraints} at every grid node, and the accumulated
action difference is affine between successive nodes.  Thus \(v^{(n)}\) is
feasible on that grid.  Dyadic cell averaging converges strongly in
\(L^2([0,1];\RR^2)\), so \(v^{(n)}\to v\) in \(L^2\).  Moreover,
\begin{align*}
  \abs{\mathfrak A(v^{(n)})-\mathfrak A(v)}
  &\le\norm{f^{(n)}-f}_{L^2}\norm{g^{(n)}}_{L^2}
    +\norm{f}_{L^2}\norm{g^{(n)}-g}_{L^2}\longrightarrow0,
\end{align*}
where Cauchy--Schwarz gives
\(\norm{\int_0^\cdot f(s)\dd s}_{L^2}\le\norm f_{L^2}\).

For each \(j\), the scaled control \(\alpha_jv^{(n)}\) lies in
\(\mathfrak B_{2^n,\alpha_j}\).  By \eqref{eq:quadratic-scaling-L}, it
remains feasible at level \(B\), and its area is
\(\alpha_j^2\mathfrak A(v^{(n)})\).  Letting \(n\to\infty\) first gives
\[
  \mathfrak G_{\alpha_j}(B)\ge\alpha_j^2\mathfrak A(v).
\]
Now let \(j\to\infty\) to obtain
\[
  \mathfrak G_*(B)\ge\mathfrak A(v).
\]
Taking the supremum over feasible \(v\) proves
\(\mathfrak G_*(B)=\mathfrak G(B)\), and hence concavity on
\((0,\infty)\).

The lower bound \eqref{eq:L-energy-lower} and Cauchy--Schwarz give, for every
control feasible at level \(B\),
\[
  \abs{\mathfrak A(v)}
  \le\norm{f}_{L^2}\norm{g}_{L^2}
  \le\frac12\int_0^1\norm{v(t)}_2^2\dd t
  \le\frac B2.
\]
\begin{revisioncheck}
The zero control is feasible, so \(0\le\mathfrak G(B)\le B/2\).  Thus
\(\mathfrak G(B)\downarrow0=\mathfrak G(0)\) as \(B\downarrow0\), which
extends concavity to \([0,\infty)\) and gives continuity at zero.  On
\((0,\infty)\), continuity follows from finiteness and concavity.
Monotonicity follows from the nesting of the feasible sets in
\eqref{eq:G-control-form}.

Finally, \(\rho\mapsto\log\rho\) is increasing and concave.  The composition
of a nondecreasing concave function with an increasing concave function is
concave.  Since \(\Gamma_\rho=\mathfrak G(\log\rho)\), the asserted
monotonicity, concavity, and continuity in \(\rho\) follow.  The saturation
assertion follows from \Cref{prop:saturation}.
\end{revisioncheck}
\end{proof}

\subsection{Weak branching asymptotics}

\begin{revisioncheck}
Near \(\rho=1\), a quadratic variational problem determines the sharp
expansion and the limiting velocity profile.  For
\(v\in L^2([0,1];\RR^2)\), define its energy by
\[
  \mathcal E(v):=\int_0^1\norm{v(t)}_2^2\dd t.
\]
Recall the rotating quadratic extremal \(v_0\) from \eqref{eq:v0-main}.
\end{revisioncheck}

\subsubsection{Quadratic bound and stability}

The quadratic approximation begins with the sharp relation between energy
and Heisenberg area.

\begin{lemma}\label{lem:energy-area}
For every absolutely continuous \(\omega=(x,y)\) with \(\omega(0)=0\),
\begin{equation}\label{eq:area-energy}
  \abs{A(\omega)}
  \le\frac1\pi\int_0^1\norm{\dot\omega(t)}_2^2\dd t.
\end{equation}
\revisionmark{Consequently, for every \(\rho>1\),
\(\Gamma_\rho\le(\log\rho)/\pi\).}
\end{lemma}

\begin{proof}
\begin{revisioncheck}
If the energy on the right of \eqref{eq:area-energy} is infinite, the
inequality is automatic.  We may therefore assume that both derivatives
belong to \(L^2\).
\end{revisioncheck}
For \(x(0)=0\), the mixed boundary Wirtinger inequality gives
\[
  \norm{x}_{L^2}\le\frac2\pi\norm{\dot x}_{L^2}.
\]
Consequently, Cauchy--Schwarz and \(2ab\le a^2+b^2\) give
\begin{align*}
  \abs{A(\omega)}
  &\le\norm{x}_{L^2}\norm{\dot y}_{L^2}
  \le\frac2\pi\norm{\dot x}_{L^2}\norm{\dot y}_{L^2}\\
  &\le\frac1\pi
  \left(\norm{\dot x}_{L^2}^2+\norm{\dot y}_{L^2}^2\right).
\end{align*}
If \(\omega\in\revisionmark{\cC_\rho}\), then \Cref{lem:convex-facts} gives
\[
  \int_0^1\norm{\dot\omega}_2^2\dd t
  \le\int_0^1L(\dot\omega)\dd t\le\revisionmark{\log\rho}.
\]
Taking the supremum proves the claim.
\end{proof}

The leading inequality has a one dimensional family of extremizers.  The
next lemma quantifies the spectral gap away from that family.

\begin{lemma}\label{lem:quadratic-stability}
There is \(c_*>0\) such that, for every
\(v\in L^2([0,1];\RR^2)\),
\begin{equation}\label{eq:quadratic-stability}
  \frac{\mathcal E(v)}{\pi}-\revisionmark{\mathfrak A(v)}
  \ge c_*\inf_{\lambda\in\RR}\norm{v-\lambda v_0}_{L^2}^2.
\end{equation}
The optimal constant is \(c_*=2/(3\pi)\).
\end{lemma}

\begin{proof}
Let \(T:L^2(0,1)\to L^2(0,1)\) be the Volterra operator
\[
  (Tf)(t):=\int_0^tf(s)\dd s
\]
\begin{revisioncheck}
and let \(T^*\) denote its Hilbert space adjoint.  On
\(L^2\oplus L^2\), define
\end{revisioncheck}
\[
  \mathsf S:=\frac12
  \begin{pmatrix}0&T^*\\ T&0\end{pmatrix}.
\]
Then
\(\revisionmark{\mathfrak A((f,g))}=\langle\mathsf S(f,g),(f,g)\rangle\).
To identify the singular values, set \(u:=T^*Tf\).  Then
\(u''=-f\), \(u'(0)=0\), and \(u(1)=0\).  The eigenvalue equation for
\(T^*T\) therefore gives the singular values
\[
  \frac{2}{(2j+1)\pi},\qquad j=0,1,2,\ldots.
\]
The top eigenvalue of \(\mathsf S\) is \(1/\pi\), it is simple, and its
unit eigenvector is \(v_0\).  The largest spectral value on
\(v_0^\perp\) is \(1/(3\pi)\).
Expanding \(v\) in an orthonormal eigenbasis yields
\[
  \left\langle\left(
    \frac1\pi\operatorname{Id}_{L^2\oplus L^2}-\mathsf S
  \right)v,v\right\rangle
  \ge\frac{2}{3\pi}
       \norm{v-\langle v,v_0\rangle v_0}_{L^2}^2,
\]
which is \eqref{eq:quadratic-stability}.  Equality holds for an eigenvector
of \(\mathsf S\) with eigenvalue \(1/(3\pi)\), so the constant cannot be
increased.
\end{proof}

\subsubsection{\texorpdfstring{\additionmark{Exact total action relaxation and
strict separation}}{Exact total action relaxation and strict separation}}

\begin{additioncheck}
Retaining only the action constraint at time one gives the relaxed value
\begin{equation}\label{eq:relaxed-Gamma}
  \overline\Gamma_\rho
  :=\sup\left\{A(\omega):
  \begin{array}{l}
    \omega\in\ac([0,1];\RR^2),\quad \omega(0)=0,\\
    \displaystyle\int_0^1L(\dot\omega(t))\dd t\le\log\rho
  \end{array}
  \right\},
  \qquad 1<\rho<4.
\end{equation}
\end{additioncheck}

\begin{proposition}[Total action relaxation and strict separation]
\label{prop:relaxation-main}
\begin{analystadditioncheck}
For every \(1<\rho<4\),
\begin{equation}\label{eq:relaxed-exact-main}
  \overline\Gamma_\rho
  =\min_{h>0}\frac{(\log\rho+h)^2}{a(h)}
  =4Z_0(\rho),
\end{equation}
and
\begin{equation}\label{eq:strict-relaxation-gap-main}
  \Gamma_\rho<\overline\Gamma_\rho.
\end{equation}
\end{analystadditioncheck}
\end{proposition}

\begin{analystadditioncheck}
The exact origin loop also determines the relaxation in
\eqref{eq:relaxed-Gamma}.  The proof exhibits an optimizer, identifies the
equality case after taking coordinatewise absolute values, and explains why
that optimizer violates an intermediate action constraint below saturation.

\begin{proof}[\revisionmark{Proof of \Cref{prop:relaxation-main}}]
\emph{Exact value.}
Let \(\omega\) be feasible in \eqref{eq:relaxed-Gamma}, put
\(v=(v_1,v_2):=\dot\omega\), and write
\[
  \widehat v(t):=(|v_1(t)|,|v_2(t)|),
  \qquad
  \widehat\omega(t):=\int_0^t\widehat v(s)\dd s
  =:(\widehat x(t),\widehat y(t)).
\]
The coordinate reflection symmetries of \(L\) give
\(L(\widehat v)=L(v)\), while
\begin{equation}\label{eq:positive-folding-area}
  \mathfrak A(v)
  \le\int_0^1\left(\int_0^t|v_1(s)|\dd s\right)|v_2(t)|\dd t
  =A(\widehat\omega).
\end{equation}
It is therefore enough to consider velocities with nonnegative coordinates.
Put \((x_+,y_+):=\widehat\omega(1)\) and
\[
  \varpi(t):=(y_+-\widehat y(t),\widehat x(t)).
\]
Then \(\varpi\) runs monotonically from \((y_+,0)\) to \((0,x_+)\),
\(\dot\varpi=\mathsf J\widehat v\), and
\[
  A(\widehat\omega)
  =-\int_0^1\varpi_2(t)\dot\varpi_1(t)\dd t.
\]
Reflect the curve and the region between it and the coordinate axes in
both axes.  The resulting closed curve has signed area
\(4A(\widehat\omega)\) and \(K_h\) anisotropic perimeter
\[
  4\int_0^1h_{K_h}(\dot\varpi(t))\dd t
  =4\int_0^1h_{K_h}(\widehat v(t))\dd t,
\]
where the equality uses the quarter turn symmetry of \(K_h\).
\Cref{lem:anisotropic-loop} and \eqref{eq:support-Fenchel} imply
\begin{align}
  A(\widehat\omega)
  &\le\frac1{a(h)}
    \left\{\int_0^1h_{K_h}(\widehat v(t))\dd t\right\}^2\notag\\
  &\le\frac1{a(h)}
    \left\{\int_0^1(L(\widehat v(t))+h)\dd t\right\}^2
  \le\frac{(B+h)^2}{a(h)}.
  \label{eq:relaxed-Wulff-upper}
\end{align}
Taking the infimum over \(h>0\) and using
\eqref{eq:origin-Wulff-main} gives
\(\overline\Gamma_\rho\le4Z_0(\rho)\).

We next use the period one Hamiltonian orbit from the proof of
\Cref{thm:origin-fiber-main}.  Let \(h_B\) be the unique minimizer from
\Cref{lem:unique-Wulff-parameter}, and phase \(\theta_{h_B}\) so that
\(\theta_{h_B}(0)\) lies on the positive horizontal axis and
\(\theta_{h_B}(1/4)\) lies on the positive vertical axis.  Define
\begin{equation}\label{eq:relaxed-equality-curve}
  \varpi_B(t):=\frac4{a'(h_B)}\theta_{h_B}(t/4),
  \qquad
  \omega_B(t):=
  \bigl(\varpi_{B,2}(t),
        \varpi_{B,1}(0)-\varpi_{B,1}(t)\bigr).
\end{equation}
The flow equation gives
\[
  \dot\varpi_B(t)=\mathsf J\nabla\psi(\theta_{h_B}(t/4)),
  \qquad
  \dot\omega_B(t)=\nabla\psi(\theta_{h_B}(t/4)).
\]
The four quarter arcs have equal traversal time under the flow and equal
action.  Since
the complete loop in the proof of \Cref{thm:origin-fiber-main} has total
action \(B\),
\[
  \int_0^1L(\dot\omega_B(t))\dd t
  =4\int_0^{1/4}
    L(\nabla\psi(\theta_{h_B}(s)))\dd s
  =B.
\]
The region between \(\varpi_B\) and the coordinate axes is one quarter of the
dilation \(4K_{h_B}/a'(h_B)\).  Therefore
\begin{equation}\label{eq:relaxed-equality-area}
  A(\omega_B)
  =\frac14\left\{\frac4{a'(h_B)}\right\}^2a(h_B)
  =\frac{4a(h_B)}{a'(h_B)^2}
  =4Z_0(\rho).
\end{equation}
This proves \eqref{eq:relaxed-exact-main}.

\emph{Strict separation: equality forces the Wulff arc.}
Suppose, to the contrary, that
\(\Gamma_\rho=\overline\Gamma_\rho\).  Compactness of \(\cQ_\rho\) gives a
maximizing path \(\omega\in\cC_\rho\).  Replacing the two coordinates of
its velocity by their absolute values preserves every action integral
\(J_t\) and cannot decrease the area by
\eqref{eq:positive-folding-area}.  The resulting path is consequently both
admissible and a maximizer of the relaxed problem.  At the unique minimizer
\(h_B\), equality must hold at every step of
\eqref{eq:relaxed-Wulff-upper}.  In particular, the total action is \(B\),
Fenchel equality holds almost everywhere, and the reflected region is an
equality case of the planar Wulff inequality.  Indeed, monotonicity of the
two coordinates of \(\varpi\) makes the four reflected arcs the boundary
of a finite perimeter set, apart from possible duplicated portions on the
coordinate axes.  Fenchel equality gives
\(\widehat v=\nabla\psi(\theta)\) with
\(\theta\in\partial K_{h_B}\); since \(h_B>0\), this velocity is nonzero
almost everywhere.  Moreover, any nonconstant duplicated portion adds
positive anisotropic length without changing the set or its area, so
equality rules it out.  The curve is therefore the boundary of the reflected
region with multiplicity one.  The equality
characterization in \cite{EspositoFuscoTrombetti2005} makes this region a
translate and a positive dilation of \(-\mathsf J K_{h_B}=K_{h_B}\), where the last
equality follows from quarter turn symmetry.  Symmetry about both coordinate
axes forces the translation to vanish.  Thus the curve \(\varpi\) above has the form
\[
  \varpi(t)=\lambda\theta(t),
  \qquad \theta(t)\in\partial K_{h_B},
  \qquad \lambda>0,
\]
on the first quadrant arc.  Strict convexity of \(\psi\) and Fenchel
equality give
\[
  \widehat v(t)=\nabla\psi(\theta(t)),
  \qquad
  \lambda\dot\theta(t)=\mathsf J\nabla\psi(\theta(t)).
\]
The flow in \eqref{eq:Hamilton-level-flow} takes time \(1/4\) to traverse
one quadrant, whereas the present path takes time one.  Hence
\(\lambda=4/a'(h_B)\), and the folded path is exactly the curve in
\eqref{eq:relaxed-equality-curve}, up to the coordinate symmetries.

\emph{Strict separation: violation of an intermediate constraint.}
It remains to compare the action density at the initial axis endpoint with
its average.  On
\(\partial K_{h_B}\), put
\[
  \Lambda(\alpha,\beta)
  :=\alpha\,\partial_1\psi(\alpha,\beta)
    +\beta\,\partial_2\psi(\alpha,\beta)-\psi(\alpha,\beta).
\]
Thus \(\Lambda(\theta)=L(\nabla\psi(\theta))\).
\hyperadd{Although \(\partial K_{h_B}\) corresponds to a hyperbolic circle,
the action density \(\Lambda\) varies along it.  The calculation below
locates its maxima and shows how they violate an intermediate constraint.}
In the first quadrant set
\(u=(\alpha+\beta)/2\), \(w=(\alpha-\beta)/2\), and
\[
  \mathfrak e(r):=r\tanh r-\log\cosh r,
  \qquad
  \mathfrak f(r):=\frac{2r}{\sinh(2r)}.
\]
Then \(\psi=\log\cosh u+\log\cosh w\) and
\(\Lambda=\mathfrak e(u)+\mathfrak e(w)\).  Along a level set, moving from
the diagonal toward
an axis gives
\begin{equation}\label{eq:action-density-monotonicity}
  \frac{\dd\Lambda}{\dd w}
  =\tanh w\{\mathfrak f(w)-\mathfrak f(u)\}>0,
  \qquad 0<w<u,
\end{equation}
because \(\mathfrak f\) is strictly decreasing on \((0,\infty)\).  Indeed, the
numerator in
\[
  \mathfrak f'(r)=
  \frac{2\{\sinh(2r)-2r\cosh(2r)\}}{\sinh^2(2r)}
\]
vanishes at zero and has derivative \(-4r\sinh(2r)<0\) for \(r>0\).
Coordinate
exchange gives the same conclusion on the other half of the quadrant.
Thus \(\Lambda\) is nonconstant and is strictly largest at the two axis
endpoints.  Its average along the equality curve is \(B\), so the curve
starting at an axis satisfies
\[
  L(\dot\omega_B(0))>B.
\]
Continuity gives \(J_t(\omega_B)>Bt\) for all sufficiently small
\(t>0\), contradicting its assumed membership in \(\cC_\rho\).  This proves
\eqref{eq:strict-relaxation-gap-main}.
\end{proof}

\begin{corollary}[Weak branching asymptotics of the relaxed front]
\label{cor:relaxed-weak-branching}
As \(B=\log\rho\downarrow0\),
\begin{equation}\label{eq:weak-upper-second}
  \overline\Gamma_\rho
  =\frac B\pi-\frac{B^2}{4\pi}+O(B^3).
\end{equation}
\end{corollary}

\begin{proof}
First,
\(h_B\to0\) as \(B\downarrow0\): evaluating the quotient in
\eqref{eq:origin-Wulff-main} at \(h=B\) gives \(Z_0(\e^B)=O(B)\), whereas
the quotient is bounded away from zero when \(h\) stays away from zero.
Using \eqref{eq:ah-small} and \eqref{eq:ah-derivative-small} in the critical
point identity \eqref{eq:origin-critical-main} gives
\[
  B=h_B-\frac{h_B^2}{2}+O(h_B^3),
  \qquad
  h_B=B+\frac{B^2}{2}+O(B^3).
\]
Using \eqref{eq:relaxed-exact-main},
\[
  \overline\Gamma_\rho
  =\frac{4a(h_B)}{a'(h_B)^2}
  =\frac{h_B}{\pi}-\frac{3h_B^2}{4\pi}+O(h_B^3)
\]
and substitution of the preceding expansion for \(h_B\) proves
\eqref{eq:weak-upper-second}.
\end{proof}

\end{analystadditioncheck}

\subsubsection{\texorpdfstring{\revisionmark{A matching lower bound with
constant action density}}{A matching lower bound with constant action density}}

The relaxed problem provides the upper estimate.  To obtain the matching
lower estimate, we adjust the quadratic extremal pointwise so that its action
density equals \(B\) at every time.

\begin{proposition}[Second order admissible test path]\label{prop:weak-lower-second}
\begin{revisioncheck}
As \(\rho\downarrow1\),
\begin{equation}\label{eq:weak-lower-second}
  \Gamma_\rho
  \ge\frac{\log\rho}{\pi}-\frac{(\log\rho)^2}{4\pi}
    +O\bigl((\log\rho)^3\bigr).
\end{equation}
\end{revisioncheck}
\end{proposition}

\begin{proof}
\begin{revisioncheck}
For all sufficiently small \(B\), the implicit function theorem applied
uniformly to \(L(sv_0(t))\) as a function of \(s^2\) gives a unique
positive continuous function \(s_B\) such that
\[
  L(s_B(t)v_0(t))=B,
  \qquad 0\le t\le1,
\]
and \(\norm{s_B}_\infty\to0\) as \(B\downarrow0\).
\end{revisioncheck}
The expansion \eqref{eq:L-local-expansion}, applied after putting
\(r=s^2\), gives uniformly in \(t\),
\begin{equation}\label{eq:sB-pointwise-expansion}
  s_B(t)=\sqrt B\left(1-\frac B2Q_4(v_0(t))+O(B^2)\right).
\end{equation}
Set \(v_B^{\mathrm{test}}(t)=s_B(t)v_0(t)\).
\revisionmark{Define
\(\omega_B^{\mathrm{test}}(t):=\int_0^t v_B^{\mathrm{test}}(s)\dd s\).
Its action density equals \(B\) pointwise, so
\(J_t(\omega_B^{\mathrm{test}})=tB\) for every \(t\in[0,1]\), and hence
\(\omega_B^{\mathrm{test}}\in\cC_{\e^B}\).}

The Fr\'echet derivative of the quadratic functional
\revisionmark{\(\mathfrak A\)} at
\(v_0\) is
\begin{equation}\label{eq:area-derivative-v0}
  D\revisionmark{\mathfrak A}(v_0)[w]
  =\frac2\pi\int_0^1v_0(t)\cdot w(t)\dd t,
\end{equation}
because \(\mathsf Sv_0=\pi^{-1}v_0\).  From
\eqref{eq:sB-pointwise-expansion},
\[
  v_B^{\mathrm{test}}
  =\sqrt B\,v_0+B^{3/2}w_0+O_{L^\infty}(B^{5/2}),
  \qquad
  w_0=-\frac12Q_4(v_0)v_0.
\]
Since \revisionmark{\(\mathfrak A\)} is quadratic,
\begin{align*}
  \revisionmark{\mathfrak A}(v_B^{\mathrm{test}})
  &=B\revisionmark{\mathfrak A}(v_0)
    +B^2D\revisionmark{\mathfrak A}(v_0)[w_0]+O(B^3)\\
  &=\frac B\pi-\frac{B^2}\pi
    \int_0^1Q_4(v_0(t))\dd t+O(B^3)\\
  &=\frac B\pi-\frac{B^2}{4\pi}+O(B^3).
\end{align*}
\end{proof}

\begin{proof}[Proof of \Cref{thm:weak-branching}]
\begin{revisioncheck}
The upper estimate is \eqref{eq:weak-upper-second}, and the matching lower
estimate is \Cref{prop:weak-lower-second}.  Together with
\(\Gamma_\rho\le\overline\Gamma_\rho\), they give the
\additionmark{third order remainder} in \eqref{eq:weak-two-sided}.
A maximizer \(\omega_\rho\) in the definition of \(\Gamma_\rho\)
exists because \(\cQ_\rho\) is compact and endpoint evaluation is
continuous.
\begin{analystadditioncheck}
  Put \(v_B:=\dot\omega_\rho\), and
  \(E_B:=\mathcal E(v_B)\).  Since
  \(\mathfrak A(v_B)=\Gamma_\rho\),
  \Cref{lem:convex-facts,lem:energy-area} and the expansion just proved give
\[
  B-\frac{B^2}{4}+O(B^3)\le E_B\le B
\]
and
\begin{equation}\label{eq:quadratic-deficit-small}
  0\le\frac{E_B}{\pi}-\mathfrak A(v_B)\le CB^2,
\end{equation}
where \(C\) is independent of all sufficiently small \(B\).
In particular, \(E_B=B+O(B^2)\), while
\Cref{lem:quadratic-stability} yields
\[
  \inf_{\lambda\in\RR}\norm{v_B-\lambda v_0}_{L^2}=O(B).
\]
Write \(v_B=c_Bv_0+w_B\), where \(w_B\perp v_0\).  Then
\(\norm{w_B}_{L^2}=O(B)\) and \(c_B^2=B+O(B^2)\), so
\[
  \inf_{\sigma\in\{-1,1\}}
  \norm{B^{-1/2}v_B-\sigma v_0}_{L^2}\longrightarrow0.
\]
This is \eqref{eq:weak-rigidity-main}.
\end{analystadditioncheck}
\end{revisioncheck}
\end{proof}

\begin{analystadditioncheck}
\begin{revisioncheck}
Consequently,
\(0<\overline\Gamma_\rho-\Gamma_\rho
=4Z_0(\rho)-\Gamma_\rho=O((\log\rho)^3)\) as \(\rho\downarrow1\).
\end{revisioncheck}
The identity and strict positivity follow from
\Cref{prop:relaxation-main}; the order follows by subtracting
\eqref{eq:weak-two-sided} from \eqref{eq:weak-upper-second}.
\end{analystadditioncheck}

\subsection{Approach to saturation from below}

The weak branching analysis above describes the central front near
\(\rho=1\).  A coordinate path gives the estimate near \(\rho=4\).

\begin{proposition}[Approach to saturation from below]\label{prop:near-four}
\begin{revisioncheck}
For \(1<\rho<4\), set \(\varepsilon:=\log(4/\rho)\).  As
\(\rho\uparrow4\), equivalently \(\varepsilon\downarrow0\),
\begin{equation}\label{eq:near-four-bound}
  0\le\frac14-\Gamma_\rho
  \le\frac{\varepsilon}{2\log(1/\varepsilon)}(1+o(1)).
\end{equation}
\end{revisioncheck}
\end{proposition}

\begin{proof}[Proof of \Cref{prop:near-four}]
\begin{revisioncheck}
Along a coordinate axis, direct Legendre transformation of
\(\psi(\theta,0)=2\log\cosh(\theta/2)\) gives
\[
  L(s,0)=(1+s)\log(1+s)+(1-s)\log(1-s),
  \qquad0\le s\le1.
\]
Let \(s_\rho\in(0,1)\) be the unique solution of
\(L(s_\rho,0)=\log\rho\).  The path with velocity \((s_\rho,0)\) on
\([0,1/2]\) and \((0,s_\rho)\) on \([1/2,1]\) has action density
\(\log\rho\) pointwise and therefore belongs to \(\cC_\rho\).  Its area is
\(s_\rho^2/4\), so
\[
  \frac14-\Gamma_\rho\le\frac{1-s_\rho^2}{4}.
\]
Write \(\upsilon_\rho:=1-s_\rho\).  Recalling
\(\varepsilon=\log(4/\rho)\) from \Cref{prop:near-four}, expansion of the
displayed formula for \(L\) gives
\[
  \varepsilon
  =\upsilon_\rho
    \left(\log\frac2{\upsilon_\rho}+1\right)+O(\upsilon_\rho^2).
\]
Consequently,
\(
\upsilon_\rho\sim\varepsilon/\log(1/\varepsilon)
\),
and substitution proves \eqref{eq:near-four-bound}.  The lower bound
\(1/4-\Gamma_\rho\ge0\) follows from \Cref{prop:saturation}.
\end{revisioncheck}
\end{proof}

\begin{analystadditioncheck}
\begin{remark}[Open variational questions below saturation]
For \(1<\rho<4\), it remains open to classify the paths maximizing
\(\Gamma_\rho\), determine their contact sets
\[
  \left\{t\in[0,1]:J_t(\omega)=t\log\rho\right\},
\]
and establish uniqueness modulo the symmetries in \eqref{eq:symmetries}.
Closed form formulas for \(\Gamma_\rho\) and the endpoints of a general
fiber \(\cI_\rho(a,b)\) are also unknown.  The strict inequality
\(\Gamma_\rho<4Z_0(\rho)\) in \Cref{prop:relaxation-main} shows that the
intermediate action constraints cannot be recovered from the constraint at
time one.
\end{remark}
\end{analystadditioncheck}

\end{analystcheck}

\section*{Data availability and conflict of interest}
\begin{additioncheck}
The data and code for Figure~\ref{fig:rho4-cloud-envelopes} are available
as ancillary files.
\end{additioncheck}
The authors declare no conflict of interest.

\section*{Declaration on the use of AI tools}
AI tools were used to assist with numerical simulations and language
polishing of the manuscript.
The authors take full responsibility for the mathematical content and the final text.

\section*{Acknowledgments}

W.S. is supported by the National Natural
Science Foundation of China (No.~12401171) and by the National
Key R\&D Program of China (No.~2022YFA1006500). J.L. is supported by the National Natural
Science Foundation of China (No.~12571127). 

\bibliographystyle{amsplain}
\bibliography{ref}

\end{document}